\documentclass{article} 
\usepackage{iclr2025_conference,times}

\usepackage{hyperref}
\usepackage{url}

\title{LancBiO: Dynamic Lanczos-aided Bilevel Optimization via Krylov Subspace}

\author{Yan Yang\\
LSEC, AMSS\\
Chinese Academy of Sciences\\
University of Chinese Academy of Sciences\\
\texttt{yangyan@amss.ac.cn} \\
\And
Bin~Gao\thanks{Corresponding author.}\ \ \ \& Ya-xiang~Yuan\\
LSEC, AMSS\\
Chinese Academy of Sciences\\
\texttt{\{gaobin,yyx\}@lsec.cc.ac.cn} \\
}

\usepackage[utf8]{inputenc} 
\usepackage[T1]{fontenc}    
\usepackage{hyperref}       
\hypersetup{colorlinks=true,linkcolor=blue,citecolor=blue}
\usepackage{url}            
\usepackage{booktabs}       
\usepackage{amsfonts}       
\usepackage{nicefrac}       
\usepackage{microtype}      

\usepackage{algorithm}
\usepackage{algorithmic}

\usepackage{graphicx}
\usepackage{wrapfig}
\usepackage{caption}

\usepackage{standalone}
\usepackage{tikz,array}
\usetikzlibrary{intersections}
\usetikzlibrary{arrows.meta,calc,decorations.text}
\usetikzlibrary{matrix,positioning,decorations.pathreplacing,shapes.geometric}
\usepackage{bm}
\usepackage{tikz-3dplot}
\usetikzlibrary{3d,calc,arrows.meta}
\usepackage{etoolbox}
\usetikzlibrary{patterns}
\usetikzlibrary{backgrounds}
\usetikzlibrary{3d, arrows.meta, positioning, shapes}
\usetikzlibrary{patterns.meta}

\usetikzlibrary{external}
\usepackage{amsmath}
\usepackage{amssymb}
\usepackage{mathtools}
\usepackage{amsthm}

\usepackage[capitalize,noabbrev]{cleveref}

\theoremstyle{plain}
\newtheorem{theorem}{Theorem}[section]
\newtheorem{proposition}[theorem]{Proposition}
\newtheorem{lemma}[theorem]{Lemma}

\theoremstyle{definition}
\newtheorem{definition}[theorem]{Definition}
\newtheorem{assumption}[theorem]{Assumption}
\theoremstyle{remark}
\newtheorem{remark}[theorem]{Remark}

\usepackage{booktabs}
\usepackage{colortbl}
\usepackage{multirow} 
\usepackage{float}
\usepackage{arydshln}

\newcommand{\norm}[1]{\left\| {#1} \right\| }
\newcommand{\abs}[1]{\left | {#1} \right |}
\newcommand{\kh}[1]{\left ( {#1} \right ) }
\newcommand{\hkh}[1]{\left\{ {#1} \right\}}

\newcommand{\revise}[1]{{#1}}

\DeclareMathOperator*{\argmin}{arg\,min}

\newcommand{\boldt}[1]{\textbf{#1}}
\usepackage[normalem]{ulem}

\newcommand{\expnumber}[2]{{#1}\mathrm{e}{#2}}

\iclrfinalcopy
\begin{document}

\maketitle

\begin{abstract}
Bilevel optimization, with broad applications in machine learning, has an intricate hierarchical structure. Gradient-based methods have emerged as a common approach to large-scale bilevel problems. However, the computation of the hyper-gradient, which involves~a Hessian inverse vector product, confines the efficiency and is regarded as~a bottleneck. To circumvent the inverse, we construct a sequence of low-dimensional approximate Krylov subspaces with the aid of the Lanczos process. As a result, the constructed subspace is able to dynamically and incrementally approximate the Hessian inverse vector product with less effort and thus leads to a favorable estimate of the hyper-gradient. Moreover, we propose a~provable subspace-based framework for bilevel problems where one central step is to solve a small-size tridiagonal linear system. To the best of our knowledge, this is the first time that subspace techniques are incorporated into bilevel optimization. This successful trial not only enjoys $\mathcal{O}(\epsilon^{-1})$ convergence rate but also demonstrates efficiency in a synthetic problem and two deep learning tasks.
\end{abstract}

\section{Introduction}\label{sec:intro}
Bilevel optimization, in which upper-level and lower-level problems are nested with each other, mirrors a~multitude of applications, e.g., game theory~\citep{stackelberg1952stack}, hyper-parameter selection \citep{ye2023difference}, data poisoning \citep{liu2024poisoning}, meta-learning~\citep{bertinetto2018meta}, neural architecture search~\citep{liu2018darts,wang2022zarts}, adversarial training~\citep{wang2021adver}, reinforcement learning~\citep{chakraborty2024parl,thoma2024contextualBiRL,yang2025sobiRL}, computer vision \citep{liu2021CVBiO}. In this paper, we consider the bilevel problem:
\begin{equation}\label{eq:standar_bio}
    \begin{array}{cl}
    \min\limits_{x \in \mathbb{R}^{d_x}}& \varphi(x):=f\left(x,y^*(x)\right)
    \\
    \mathrm{s.\,t.}& y^*(x) \in \argmin\limits_{y \in \mathbb{R}^{d_y}} g(x,y),    
    \end{array}
\end{equation}
where the upper-level function~$f$ and the lower-level function~$g$ are defined on~$\mathbb{R}^{d_x}\times\mathbb{R}^{d_y}$. $\varphi$~is called the \emph{hyper-objective}, and the gradient of $\varphi(x)$~is referred to as the \emph{hyper-gradient} \citep{pedregosa2016hyperparameter,grazzi2020iteration,yang2023accelerating,chen2023igmf} if it exists. In contrast to standard single-level optimization problems, bilevel optimization is inherently challenging due to its intertwined structure. Specifically, the formulation~\eqref{eq:standar_bio} underscores the crucial role of the lower-level solution~$y^*(x)$ in each update of~$x$. 

One of the focal points in recent bilevel methods has shifted towards nonconvex upper-level problems coupled with strongly convex lower-level problems~\citep{ghadimi2018approximation,ji2021stocbio,dagreou2022soba,li2022fsla,hong2023ttsa,hu2024contextualstoc}. This configuration ensures that $y^*(x)$ is a single-valued function of~$x$, i.e., $y^*(x)=\argmin_{y\in \mathbb{R}^{d_y}} g(x,y)$. Subsequently, $\nabla \varphi(x)$ can be computed via the implicit function theorem following~\citep{ghadimi2018approximation},
\begin{equation}\label{eq:hypergradient}
    \begin{aligned}
        \nabla \varphi (x)=&~\nabla _xf\left( x,y^*\left( x \right) \right) -\nabla^2 _{xy}g\left( x,y^*\left( x \right) \right) \left[ \nabla _{yy}^{2}g\left( x,y^*\left( x \right) \right) \right] ^{-1}\nabla _yf\left( x,y^*\left( x \right) \right).
    \end{aligned}
\end{equation}
Gradient methods based on the hyper-gradient, $x_{k+1}=x_k-\lambda\nabla \varphi (x_k),$ are known as the approximate implicit differentiation (AID) based methods~\citep{ji2021stocbio,liu2023average,huang2024HJFBIO}. However, the computation of the hyper-gradient~\eqref{eq:hypergradient} suffers from two pains: 1)~solving the lower-level problem to obtain $y^*(x)$; 2)~assembling the \emph{Hessian inverse vector product}
\begin{equation}\label{eq:def_v_star}
    v^*(x) := \left[ \nabla _{yy}^{2}g\left( x,y^*\left( x \right) \right) \right] ^{-1}\nabla _yf\left( x,y^*\left( x \right) \right),
\end{equation}
or equivalently, solving a~large linear system {in terms of $v$},
\begin{equation}\label{eq:lin_sys}
    \nabla^2_{yy}g(x,y^*(x))v=\nabla_yf(x,y^*(x)).
\end{equation}
To this end, it is beneficial to adopt a few inner iterations to approximate~$y^*(x)$ and~$v^*(x)$ within each outer iteration~({i.e.}, the update of~$x$). Note that the approximation accuracy of~$v^*(x)$ is crucial for AID-based methods; see~\citep{ji2022will,li2022fsla}. {Specifically, the left of \cref{fig:stocbio}~confirms that the more inner iterations, the higher quality of the estimate of $v^*$, and the more enhanced descent of the objective function within the same number of outer iterations.}


\begin{figure*}[tbp]
\centering
\begin{minipage}{0.49\textwidth}
    \centering
    \!\!\includegraphics[width=1\linewidth]{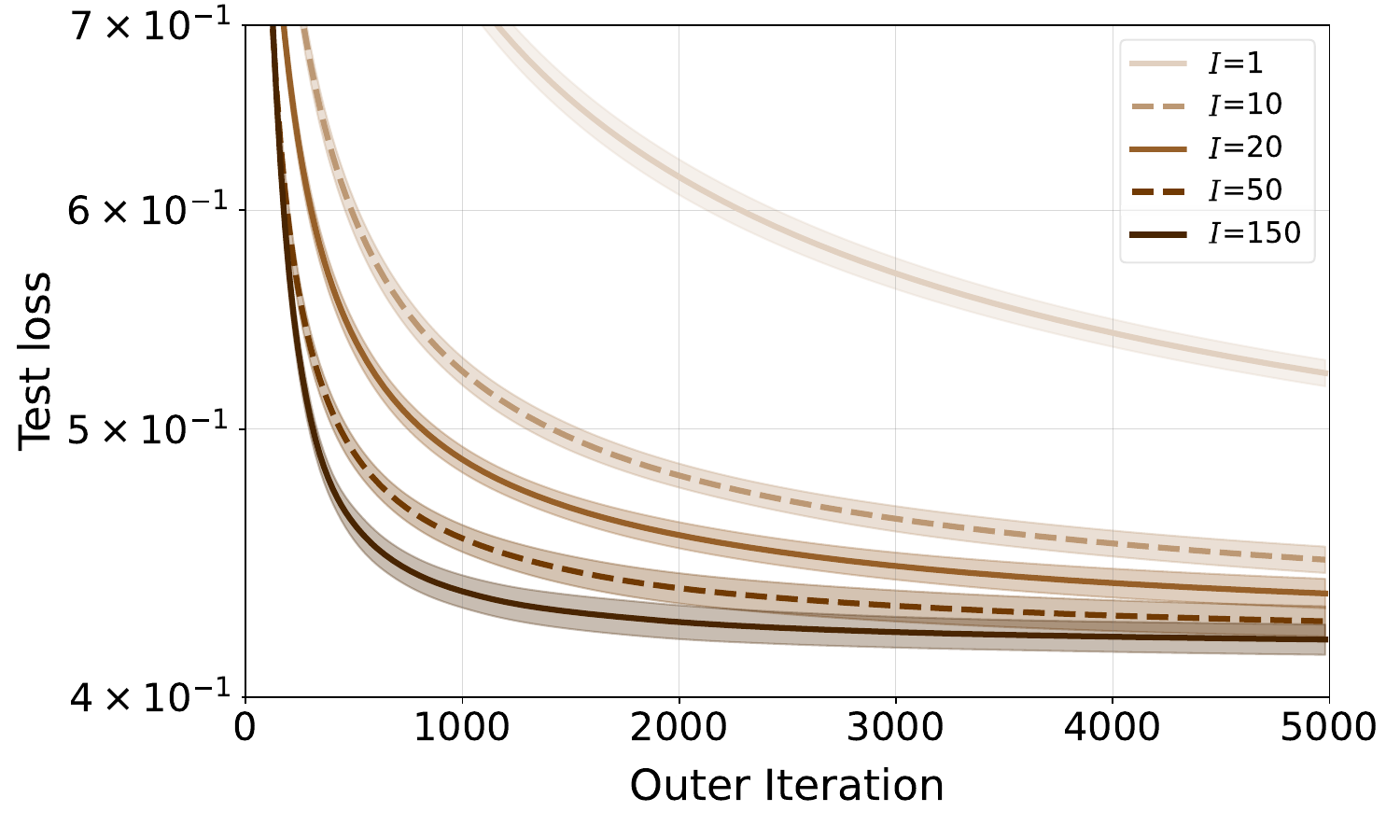}
\end{minipage}
\,
\begin{minipage}{0.493\textwidth}
    \centering
    \vspace{0.08cm}
    \includegraphics[width=1\linewidth]{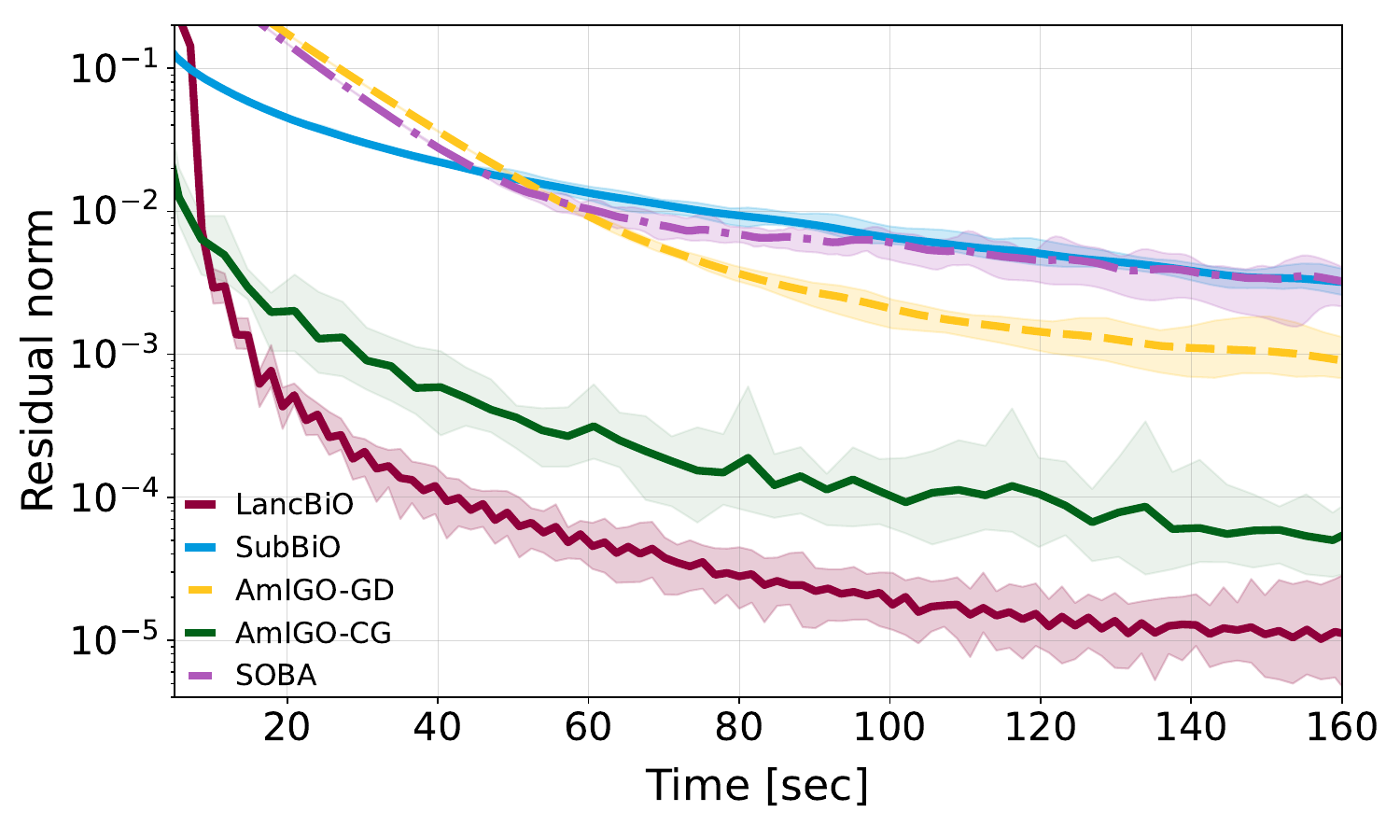}
\end{minipage}
\caption{\boldt{Left:} test loss for the method stocBiO \citep{ji2021stocbio} with different inner iterations $I$ to approximate the Hessian inverse vector product; \boldt{Right:} Estimation error of the Hessian inverse vector product in hyper-data cleaning task with corruption rate $0.5$ for different methods:  LancBiO and SubBiO (ours), AmIGO \citep{arbel2022amortized}, and SOBA~\citep{dagreou2022soba}.}
\label{fig:stocbio}
\end{figure*}


\textbf{Approximation:} Existing efforts are dedicated to approximating~$v^*$ in different fashions by regulating the number of inner iterations, e.g., the Neumann series approximation~\citep{ghadimi2018approximation,ji2021stocbio} for the inverse, gradient descent~\citep{arbel2022amortized,dagreou2022soba} and conjugate gradient descent~\citep{pedregosa2016hyperparameter,yang2023accelerating} for the linear system.

\textbf{Amortization:} Moreover, there are studies aimed at amortizing the cost of approximation through outer iterations. These methods include using the inner estimate from the previous outer iteration as~a~warm start for the current outer iteration~\citep{ji2021stocbio,arbel2022amortized,dagreou2022soba,ji2022will,li2022fsla,xiao2023galet}, or employing a~refined step size control \citep{hong2023ttsa}.

Subspace techniques, widely adopted in nonlinear optimization~\citep{yuan2014review}, approximately solve large-scale problems in lower-dimensional subspaces, which not only reduce the computational cost significantly but also enjoy favorable theoretical properties as in full space models. Taking into account the above two principles, it is reasonable to consider subspace techniques in bilevel optimization. Specifically, we can efficiently amortize the construction of low-dimensional subspaces and sequentially solve linear systems~\eqref{eq:lin_sys} in these subspaces to approximate~$v^*$ accurately. 

\subsection{Contributions}
In this paper, taking advantage of the Krylov subspace and the Lanczos process, we develop an innovative subspace-based framework---LancBiO, which features an efficient and accurate approximation of the Hessian inverse vector product~$v^*$ in the hyper-gradient---for bilevel optimization. The main contributions are summarized as follows.

Firstly, we build up a dynamic process for constructing low-dimensional subspaces that are tailored from the Krylov subspace for bilevel optimization. This process effectively reduces the large-scale subproblem~\eqref{eq:lin_sys} to the small-size tridiagonal linear system, which draws on the spirit of the Lanczos process. To the best of our knowledge, this is the first time that the subspace technique is leveraged in bilevel optimization.

Moreover, the constructed subspaces enable us to dynamically and incrementally approximate~$v^*$ across outer iterations, thereby achieving an enhanced estimate of the hyper-gradient; the right of \cref{fig:stocbio} illustrates that the proposed LancBiO reaches the best estimation error for~$v^*$. Hence, we provide~a new perspective for approximating the Hessian inverse vector product in bilevel optimization. Specifically, the number of Hessian-vector products averages at $(1+{1}/{m})$ per outer iteration with the subspace dimension $m$, which is favorably comparable with the existing methods.

Finally, we offer analysis to circumvent the instability in the process of approximating subspaces, with the result that LancBiO can profit from the benign properties of the Krylov subspace. We prove that the proposed method LancBiO is globally convergent with the convergence rate $\mathcal{O}(\epsilon^{-1})$. In addition, the efficiency of LancBiO is validated by {a synthetic problem and two deep learning tasks}. 

\subsection{Related Work}
A detailed introduction to bilevel optimization methods can be found in~\cref{sec:bilevel}.

\textbf{Krylov subspace methods:}
Subspace techniques have gained significant recognition in the realm of numerical linear algebra~\citep{parlett1998symmetric,saad2011numerical,golub2013matrix} and nonlinear optimization~\citep{yuan2014review,liu2021subspace}. Specifically, numerous optimization methods utilized subspace techniques to improve efficiency, including acceleration technique~\citep{wen2020filter}, diagonal preconditioning~\citep{gao2023diagonal}, and derivative-free optimization methods~\citep{cartis2023scalablesub}. Krylov subspace~\citep{krylov1931numerical}, due to its special structure,
\begin{equation*}
    \mathcal{K}_N(A,b) := \operatorname{span}\left\{b, A b, A^2 b, \ldots, A^{N-1} b\right\}
\end{equation*}
with the dimension~$N$ for~a matrix~$A$ and~a vector~$b$, exhibits advantageous properties in convex quadratic optimization~\citep{nesterov2018lectures}, eigenvalue computation~\citep{kuczynski1992eigenvec}, and regularized nonconvex quadratic problems~\citep{carmon2018krylov}. Krylov subspace has been widely considered in large-scale optimization such as trust region methods~\citep{gould1999tr}, trace maximization problems~\citep{liu2013limited}, and cubic Newton methods~\citep{cartis2011adaptive,jiang2024krylov}. Lanczos process~\citep{lanczos1950lanczos} is an orthogonal projection method onto the Krylov subspace, which reduces a dense symmetric matrix to a tridiagonal form. Details of the Krylov subspace and the Lanczos process are summarized in \cref{sec:krylov_lanczos}.

\textbf{Approximation of the Hessian inverse vector product~:} It~is cumbersome to compute the Hessian inverse vector product in bilevel optimization. To bypass it, several strategies implemented through inner iterations were proposed, e.g., the Neumann series approximation~\citep{ghadimi2018approximation,ji2021stocbio}, gradient descent~\citep{arbel2022amortized,dagreou2022soba}, and conjugate gradient descent~\citep{pedregosa2016hyperparameter,arbel2022amortized,yang2023accelerating}. Alternatively, the previous information was exploited in~\citep{ji2021stocbio,arbel2022amortized} as a warm start for outer iterations; \citet{ramzi2022shine}~suggested approximating the Hessian inverse in the manner of quasi-Newton; \citet{dagreou2022soba} and \citet{li2022fsla} proposed the frameworks without inner iterations to approximate the Hessian inverse vector product.

\section{Subspace-based Algorithms}\label{sec:framework}
In this section, we dive into the development of bilevel optimization algorithms for solving~\eqref{eq:standar_bio}, which dynamically construct subspaces to approximate the Hessian inverse vector product. 

The (hyper-)gradient descent method carries out the $k$-th outer iteration as $ x_{k+1} = x_k-\lambda\nabla \varphi (x_k)$, where the hyper-gradient is exactly computed by $\nabla \varphi (x_k)=\nabla _xf\left(x_k,y^*_k \right) -\nabla^2 _{xy}g\left( x_k,y^*_k \right)v^*_k$ with $y^*_k:=y^*(x_k)$ and $v^*_k:=v^*(x_k)$ defined in~\eqref{eq:def_v_star}. In~view of the computational intricacy of $y^*_k$ and $v^*_k$, it is commonly concerned with the following estimator for the hyper-gradient 
\begin{equation}\label{eq:hypergradient_esti}
    \widetilde{\nabla} \varphi\left(x_k,y_k,v_k\right):=\nabla_x f\left(x_k, y_k\right)-\nabla^2_{xy} g\left(x_k, y_k\right) v_k,
\end{equation}
where $y_k$ is an approximation of $y^*_k$.
Denote $A_k=\nabla_{yy}^2g(x_k,y_k)$ and $b_k=\nabla_yf(x_k,y_k)$. The vector $v_k$~serves as the (approximate) solution of the quadratic optimization problem
\begin{equation}\label{eq:v_quadra}
   \min_{v\in \mathcal{S}_k}\frac{1}{2}v^\top A_kv-v^\top b_k,
\end{equation}
where $\mathcal{S}_k$ is the full space $\mathbb{R}^{d_y}$ and the exact solution is~$A_k^{-1}b^{}_k$. Subsequently, in order to reduce the computational cost, it is natural to ask:
\begin{center}
    \textit{Can we construct a low-dimensional subspace $\mathcal{S}_k$ such that \\the solution of the quadratic problem~\eqref{eq:v_quadra} satisfactorily approximates $A_k^{-1}b^{}_k$?}
\end{center}

General subspace constructions introduced in the existing subspace methods \citep{yuan2014review,liu2021subspace} are not straightforward and not exploited in the bilevel setting, rendering the exploration of appropriate subspaces challenging. In~the following subsections, we construct approximate Krylov subspaces and propose~an elaborate subspace-based framework for bilevel problems.

\subsection{Why Krylov subspace: the SubBiO algorithm}
In light of the Neumann series for a suitable $\eta\in\mathbb{R}$, $A^{-1}b^{}=\eta \sum_{i=0}^{\infty}(I-\eta A)^ib^{}$, it is observed from~\cref{sec:krylov_lanczos} that $A^{-1}b^{}$ belongs to a Krylov subspace for some $N>0$, i.e.,
\begin{equation*}
    A^{-1}b^{}\in\mathcal{K}_{N}(A,b)=\mathcal{K}_{N}(I-\eta A,b).
\end{equation*}
Hence, it is reasonable to consider a Krylov subspace for the construction of $\mathcal{S}_k$.

Given a constant $n\ll N$, we consider an approximation of $A^{-1}b^{}$ in a lower-dimensional Krylov subspace $\mathcal{K}_n(A,b)$, i.e., {$v_n\in \mathcal{K}_n(A,b)=\mathcal{K}_{n}(I-\eta A,b)$} and $v_n=\sum_{i=0}^{n-1}{c_i\left( I-\eta A \right)^ib} \approx A^{-1}b$. Note that the approximation~$v_n$ is composed of the set $\{\left( I-\eta A \right)^ib\}_{i=0}^{n-1}$ in the sense of the Neumann series. Moreover, we observe that $(I-\eta A)v_n\in\mathcal{K}_{n+1}(A,b)$ and hence we can recursively choose
\begin{equation*}
    v_{n+1} \in \mathcal{S}_{n+1}:=\operatorname{span}\left\{b, (I-\eta A)v_n\right\} \subseteq \mathcal{K}_{{n+1}}(A,b)
\end{equation*}
since the subspace $\operatorname{span}\left\{b, (I-\eta A)v_n\right\}$ includes the information of the increased set $\{\left( I-\eta A \right)^ib\}_{i=0}^{n}$. In summary, we can construct a sequence of two-dimensional subspaces $\{\mathcal{S}_{n}\}$ that implicitly filters information from the Krylov subspaces. The rationale for this procedure can be illustrated in~\cref{fig:Sn}.

\begin{wrapfigure}{R}{0.53\textwidth}
\ \ 
 \begin{minipage}{0.47\textwidth} 
\begin{figure}[H]
\vspace{-1.2cm}
    \centering   
    \includegraphics[width=1\linewidth]{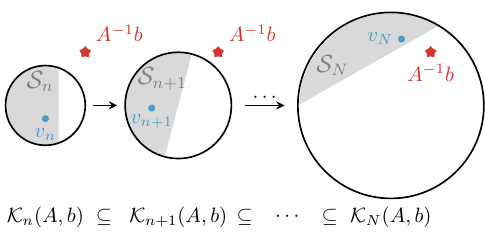}
    \caption{Illustration of approximating $A^{-1}b\in\mathcal{K}_{{N}}(A,b)$ by~$v_n$ in the two-dimensional subspace $\mathcal{S}_n\subseteq\mathcal{K}_{{n}}(A,b)$. \label{fig:Sn}}
\end{figure}
\end{minipage}
\vspace{0.cm}
\end{wrapfigure}

In the context of bilevel optimization, we seek the best solution $v_k$ to the subproblem \eqref{eq:v_quadra} in the subspace
\begin{equation}\label{eq:subbio_S}
    \mathcal{S}_{k}=\operatorname{span}\left\{ b_{{k}},(I-\eta A_{{k}})v_{k-1} \right\}.
\end{equation}
Repeating the procedure is capable of dynamically approximating the Hessian inverse vector product, i.e., $v_k$ approximates $A_k^{-1}b_k$. The Krylov Subspace-aided Bilevel Optimization algorithm (SubBiO) is listed in~\cref{alg:SubBiO}.

\begin{algorithm}[H]
    \caption{SubBiO}
    \label{alg:SubBiO}
    \begin{algorithmic}[1]
        \REQUIRE iteration threshold $K$, step sizes $\theta, \lambda, \eta$, initialization $x_1, y_1, v_{0}$
        \FOR{$k = 1, 2, \dots, K$}
            \STATE $A_k=\nabla _{yy}^{2}g\left( x_k,y_k \right)$, $b_k=\nabla _yf(x_k,y_k)$
            \STATE ${\mathcal{S}_k = \operatorname{span}\{b_k,\left( I-\eta A_k  \right)v_{k-1}\}}$ \label{alg:step1}
            \STATE $v_k =\argmin_{v\in {\mathcal{S}_k}}\frac{1}{2}v^\top A_k v-b_k^\top v$ \label{alg:step2}
            \STATE $x_{k+1} = x_k - \lambda \left(\nabla_xf(x_k,y_k)-\nabla^2_{xy}g(x_k,y_k)v_k\right)$
            \STATE $y_{k+1} = y_k - \theta \nabla_y g(x_{k+1},y_k)$
        \ENDFOR
        \ENSURE $(x_{K+1},y_{K+1})$
    \end{algorithmic}
\end{algorithm}

\subsection{Why~dynamic~Lanczos:~the~LancBiO~framework}\label{sec:lancbio}
Notice that the subproblem in SubBiO (\cref{alg:SubBiO}) can be equivalently reduced to 
\begin{equation*}
    \min_{z\in\mathbb{R}^2}\frac{1}{2}z^\top(S_k^\top A_kS_k)z - b_k^\top S_k z,
\end{equation*}
where $S_k:=[b_k\,\,\left( I-\eta A_k  \right)v_{k-1}]\in\mathbb{R}^{d_y\times 2}$. The solution $z^*\in\mathbb{R}^2$ results in $v_k=S_k z^*$. It is a two-dimensional subproblem, whereas computing the projection $S_k^\top A_kS_k$ requires two Hessian-vector products, which dominate the cost of the subproblem. Therefore, it is crucial to reduce or amortize the projection cost while preserving the advantages of the Krylov subspace. To this end, we find that the Lanczos process~(\cref{sec:krylov_lanczos}) provides an enlightening way~\citep{lanczos1950lanczos,saad2011numerical,golub2013matrix} since it allows for the construction of a Krylov subspace and maintaining a tridiagonal matrix as the projection matrix, which significantly reduces the computational cost.

In bilevel optimization, since the quadratic problem~\eqref{eq:v_quadra} evolves through outer iterations, it is difficult to leverage the Lanczos process to amortize the projection cost while updating variables as in SubBiO. Specifically, the Lanczos process is inherently unstable \citep{paige1980accuracy}, and thus the accumulative difference among $\{A_k\}$ and $\{b_k\}$ will make the Lanczos process invalid. 

In order to address the above difficulties, we propose a restart mechanism to guarantee the benign behavior of approximating the Krylov subspace and consider solving residual systems to employ the historical information. In summary, we propose a dynamic Lanczos-aided Bilevel Optimization framework, LancBiO, which is listed in~\cref{alg:LancBiO}. The only difference between LancBiO and SubBiO is solving the subproblem (line~3-4 in~\cref{alg:SubBiO}).
\begin{wrapfigure}{R}{0.6\textwidth}
\vspace{-0.7cm}
\begin{minipage}{0.6\textwidth} 
\begin{algorithm}[H]
    \caption{LancBiO}
    \label{alg:LancBiO}
    \begin{algorithmic}[1]
        \REQUIRE iteration threshold $K$, step sizes $\theta, \lambda$, initialization $x_1, y_1, v_1$, initial correction $\Delta v_{0}=0$, subspace dimension $m$, initial epoch $h=-1$
        \FOR{$k = 1, 2, \dots, K$}
            \STATE $A_k=\nabla^2_{yy}g(x_k,y_k),\ b_k=\nabla _yf(x_k,y_k)$
            \IF{$(k\ \texttt{mod}\ m)=1$}
                \STATE $h=h+1$
                \STATE $\bar{v}_h=v_k$
                \STATE $w_h=A_k\bar{v}_h$
                \STATE $Q_{k-1}=(b_k-w_h)/\norm{b_k-w_h}$
                \STATE $T_{k-1}=\texttt{Empty Matrix}$
                \STATE $\beta_k=0$
            \ENDIF
            \STATE
                    $
                (T_{k},Q_k,\beta_{k+1})\!=\!\texttt{DLanczos}(T_{k-1},Q_{k-1},A_k,\beta_k)
                    $
            \STATE $r_k=b_k-w_h$
            \STATE $\Delta v_k=Q _k(T_k)^{-1}Q_k^\top r_k$
            \STATE $v_k = \bar{v}_h + \Delta v_k$
            \STATE $x_{k+1} = x_k - \lambda \left(\nabla_xf(x_k,y_k)-\nabla^2_{xy}g(x_k,y_k)v_k\right)$
            \STATE $y_{k+1} = y_k - \theta \nabla_y g(x_{k+1},y_k)$
        \ENDFOR
        \ENSURE $(x_{K+1},\ y_{K+1})$
    \end{algorithmic}
\end{algorithm}
\end{minipage}
\vspace{-6cm}
\end{wrapfigure}

If we adapt the standard Lanczos process for tridiagonalizing~$A_k$, by starting from $q_1=b_1/\norm{b_1},\ q_0={\bf 0},\ \beta_1 = 0$ and using the dynamic matrices~$A_j$ for $j=1,2,\ldots,k$, the process is as follows,
\begin{align*}
    u_j&=A_jq_j-\beta _jq_{j-1}, 
    \\
    \alpha _j&=q_{j}^{\top}u_j,    
    \\
    \omega _j&=u_j-\alpha _jq_j,
    \\
    \beta _{j+1}&=\left\| \omega _j \right\| ,
    \\
    q_{j+1}&=\omega _j/\beta _{j+1},
\end{align*}
and~$T_k$ is~a tridiagonal matrix recursively computed from

\begin{minipage}{0.35\textwidth}
\vspace{-0.4cm}
\begin{center}
\begin{figure}[H]
    \includegraphics[width=1\linewidth]{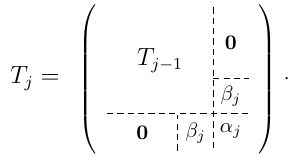}
\end{figure}
\end{center}
\vspace{-0.4cm}
\end{minipage}   

\textbf{Restart mechanism:} In contrast to SubBiO, we construct the subspace
\begin{equation*}
 \mathcal{S}_k=\operatorname{span}(Q_k)=\operatorname{span}([q_1,\ldots,q_k]).   
\end{equation*}
Consequently, $Q_k$~approximates the basis of the true Krylov subspace $\mathcal{K}_k(A_k,b_k)$, and~$T_k$ approximates the projection of~$A_k$ onto it, i.e., $T_k\approx Q_k^\top A_k Q_k$. However, $Q_k$ will lose the orthogonality due to the evolution of~$A_j$, and $T_k$ will deviate from the true projection. Based on this observation, we restart the subspace spanned by the matrix $Q$ for every $m$ outer iterations (line~3 to line~10 in~\cref{alg:LancBiO}). The restart mechanism allows us to mitigate the accumulation of the difference among $\{A_1,\dots, A_k\}$, and hence, we can maintain a more reliable basis matrix to approximate Krylov subspaces. The above dynamic Lanczos subroutine, \texttt{DLanczos}, is summarized in~\cref{sec:dLanczosSub}.

\textbf{Residual minimization:} The preceding discussion reveals that the dimension of the subspace $\mathcal{S}_k$ should be moderate to retain the reliability of the dynamic process. However, the limited dimension stagnates the approximation of the subproblem~\eqref{eq:v_quadra} to the full space problem. Therefore, instead of directly solving the quadratic subproblem \eqref{eq:v_quadra} in the subspace, we intend to find $v\in\mathcal{S}_k=\operatorname{span}(Q_k)$ such that $A_kv\approx b_k$. To this end, after going through~$m$ outer iterations, we denote the current approximation by $\bar{v}$. In the subsequent outer iterations, we concentrate on a linear system with a residual in the form
\begin{equation}\label{eq:residual}
    A_k\Delta v = b_k-A_k\bar{v},\quad \quad \Delta v\in\mathcal{S}_k.
\end{equation}
Specifically, we inexactly solve the minimal residual subproblem
\begin{equation}\label{eq:minimal_res}
    \min_{\Delta v\in\mathcal{S}_k} \norm{(b_k-A_k\bar{v})-A_k\Delta v}^2
\end{equation}
 and use the solution $\Delta v_k$ as~a correction to $\bar{v}$ (line~12 to line~14 in~\cref{alg:LancBiO}), $v_k = \bar{v} + \Delta v_k.$

Consequently, taking into account the two strategies above, we illuminate the framework LancBiO in~\cref{fig:LancBiOworkflow}. It is structured into epochs, with each epoch built by $m$~outer iterations. Notably, each epoch restarts by incrementally constructing from~a~one-dimensional space $Q_1$ to an~{$m$-dimensional} space $Q_m$, aiming to approximate the solution to the residual system within these subspaces. The solution~$\Delta v_j$  to the subproblem serves as a correction to enhance the hyper-gradient estimation, which facilitates the $(x, y)$ updating.

\begin{wrapfigure}{R}{0.49\textwidth}
\vspace{-1.cm}
 \begin{minipage}{0.4\textwidth} 
\begin{figure}[H]
    \centering       
    \,\,\,\includegraphics[scale=0.9]{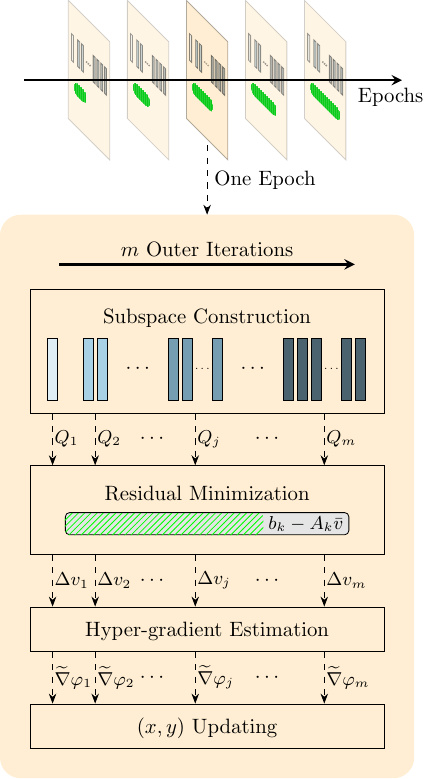}
    \captionsetup{margin={0.5cm,-0.5cm}}
    \caption{An overview of LancBiO.}
    \label{fig:LancBiOworkflow}
\end{figure}
\end{minipage}
\vspace{-2.cm}
\end{wrapfigure}

The combination of the two strategies, restart mechanism and residual minimization, not only controls the dimension of the subspace but also utilizes historical information to enhance the approximation accuracy. By considering a simplified scenario, we reduce the two strategies into solving a standard linear system problem $Ax=b$ with $A$ and $b$ fixed. Note that, from the perspective of Theorem~1 in \citep{carmon2018krylov}, the residual associated with solving a linear system in an $m$-dimensional Krylov subspace decays faster than a rate $\mathcal{O}\kh{{1}/{m^2}}$ after each restart. In other words, the estimation error of the Hessian inverse vector product experiences a decay rate of $\mathcal{O}\kh{{1}/{m^2}}$ after every restart, i.e.,
\begin{small}
\begin{equation*}
\begin{aligned}
     \frac{\left\| b-Av_{m\left( h+1 \right)} \right\|^2}{\left\| b-Av_{mh} \right\|^2}&=\frac{\left\| \left( b-A\bar{v}_h \right) -A\Delta _{m\left( h+1 \right)} \right\|^2}{\left\| b-A\bar{v}_h \right\|^2}
     \\
     &= \mathcal{O} \left( \frac{1}{m^2} \right).   
\end{aligned}
\end{equation*}    
\end{small}

\begin{remark}
The classic Lanczos process is known for its capability to solve indefinite linear systems~\citep{greenbaum1999lanczosindefinite}. In a similar fashion, the LancBiO framework can be adapted to the bilevel problems with a nonconvex lower-level problem. Interested readers are referred to~\cref{app:LLPL} for details.
\end{remark}

\subsection{Relation~to existing algorithms}\label{sec:relation}
The proposed SubBiO and LancBiO have intrinsic connections to the existing algorithms. Generally, in each outer iteration, methods such as BSA~\citep{ghadimi2018approximation} and~TTSA \citep{hong2023ttsa} truncate the Neumann series at~$N$, exploiting information from~an $N$-dimensional Krylov subspace. In contrast, both SubBiO and LancBiO implicitly gather knowledge from a higher-dimensional Krylov subspace with less effort.

SubBiO shares similarities with SOBA~\citep{dagreou2022soba} and FSLA~\citep{li2022fsla}. The update rule for the estimator~$v$ of the Hessian inverse vector product in SOBA and FSLA is
\begin{equation*}
    \begin{aligned}
        v_{k}&=v_{k-1}-\eta \left( A_{k}v_{k-1}-b_{k} \right)=(I-\eta A_k)v_{k-1} + \eta b_k,
    \end{aligned}
\end{equation*}
while the proposed SubBiO constructs~a two-dimensional subspace, $\mathcal{S}_{k}=\operatorname{span}\left\{ b_{{k}},(I-\eta A_{{k}})v_{k-1} \right\}$ defined in~\eqref{eq:subbio_S}. It~is worth noting that the updated~$v_k$ in SOBA and FSLA belongs to the subspace $\mathcal{S}_{k}$. Furthermore, in the sense of solving the two-dimensional subproblem~\eqref{eq:v_quadra}, SubBiO selects the optimal solution~$v$ in the subspace. 

In addition, if the subspace dimension~$m$ is set to one, LancBiO is simplified to~a scenario in which one conjugate gradient (CG) step with the warm start mechanism is performed in each outer iteration, which exactly recovers the algorithm AmIGO-CG~\citep{arbel2022amortized} with one inner iteration to the update of~$v$. Alternatively, if the step size~$\lambda$ in~\cref{alg:LancBiO} is set to~$0$ within each $m$-steps (i.e., only inner iterations are invoked), LancBiO reduces to the algorithm AmIGO-CG~\citep{arbel2022amortized} with~$m$ inner iterations.

\section{Theoretical Analysis}\label{sec:theory}
In this section, we provide~a non-asymptotic convergence analysis for LancBiO. Firstly, we introduce some appropriate assumptions. Subsequently, to address the principal theoretical challenges, we analyze the properties and dynamics of the subspaces constructed in~\cref{sec:framework}. Finally, we prove the global convergence of LancBiO and give the iteration complexity; the detailed proofs are provided in the appendices.

\begin{assumption}\label{assu:f}
    The upper-level function $f$ is twice continuously differentiable. The gradients $\nabla_x f(x,y)$ and $\nabla_y f(x,y)$ are $L_{fx}$-Lipschitz and $L_{fy}$-Lipschitz, and $\norm{\nabla_y f\kh{x,y^*(x)}}\leq C_{fy}$.
\end{assumption}
\begin{assumption}\label{assu:g}
    The lower-level function $g$ is twice continuously differentiable. $\nabla_x g(x,y)$ and $\nabla_yg(x,y)$ are $L_{gx}$-Lipschitz and $L_{gy}$-Lipschitz. The derivative $\nabla^2_{xy}g(x,y)$ and the Hessian matrix $\nabla^2_{yy}g(x,y)$ are $L_{gxy}$-Lipschitz and $L_{gyy}$-Lipschitz.
\end{assumption}
\begin{assumption}\label{assu:strongg}
    For any $x\in\mathbb{R}^{d_x}$, the lower-level function $g(x,\cdot)$ is $\mu_g$-strongly convex.
\end{assumption}
The Lipschitz properties of~$f,g$ and the strong convexity of the lower-level problem revealed by the above assumptions are standard in bilevel optimization \citep{ghadimi2018approximation,chen2021closing,ji2021stocbio,khanduri2021sustain,arbel2022amortized,chen2022stable,dagreou2022soba,li2022fsla,ji2022will,hong2023ttsa}. These assumptions ensure the smoothness of~$\varphi$ and~$y^*$; see the following results~\citep{ghadimi2018approximation}.
 \begin{lemma}\label{lem:ysmooth}
    Under the Assumptions \ref{assu:g} and \ref{assu:strongg}, $y^*(x)$ is ${L_{gx}}/{\mu_g}$ -Lipschitz continuous, i.e., for any $x_1,x_2\in\mathbb{R}^{d_x}$, $\norm{y^*(x_1)-y^*(x_2)}\le \frac{L_{gx}}{\mu_g}\norm{x_1-x_2}$.
\end{lemma}
\begin{lemma}\label{lem:phismooth}
    Under the Assumptions \ref{assu:f}, \ref{assu:g} and \ref{assu:strongg}, the hyper-gradient $\nabla\varphi(\cdot)$ is $L_{\varphi}$-Lipschitz continuous, i.e., for any $x_1,x_2\in\mathbb{R}^{d_x}$, $\norm{\nabla\varphi(x_1)-\nabla\varphi(x_2)}\leq L_\varphi\norm{x_1-x_2}$, where $L_\varphi>0$ is defined in~\cref{sec:app_smooth_phi}.
\end{lemma}
\begin{assumption}\label{assu:boundfx}
    There exists a constant $C_{fx}$ so that $\norm{\nabla_x f(x,y)}\leq C_{fx}$.
\end{assumption}
\cref{assu:boundfx}, commonly adopted in~\citep{ghadimi2018approximation,ji2021stocbio,liu2022bome,kwon2023f2sa}, is helpful in ensuring the stable behavior of the dynamic Lanczos process; see~\cref{sec:subspaceerror}.

\subsection{Subspace Properties in Dynamic Lanczos Process}\label{sec:subspaceerror}
In view of the inherent instability of the Lanczos process \citep{paige1980accuracy,meurant2006lanczos} and the evolution of the Hessian $\{A_k\}$ and the gradient~$\{b_k\}$ in~LancBiO, the analysis of the constructed subspaces is intricate. Based on the existing work~\citep{paige1976error,paige1980accuracy,greenbaum1997iterative}, this subsection sheds light on the analysis of the subspaces and the effectiveness of the subproblem in approximating the full space problem in LancBiO.

An \emph{epoch} is constituted of~a complete $m$-step dynamic Lanczos process between two restarts, namely, after~$h$ epochs, the number of outer iterations is~$mh$. Given the outer iterations $k=mh+j$ for $j=1,2,\ldots,m$, we denote 
\begin{equation*}
    \varepsilon _{st}^{h}:=\left( 1+\frac{L_{gx}}{\mu _g} \right) \left\| x_{mh+s}-x_{mh+t} \right\| +\left\| y_{mh+s}-y_{mh+s}^{*} \right\|
\end{equation*}
for $s,t=1,2,\dots,m$, and $\varepsilon _{j}^{h}:=\max_{1\le s,t\le j} \varepsilon _{st}^{h}$, serving as the accumulative difference. For brevity, we omit the superscript where there is no ambiguity, and we are slightly abusing of notation that at the current epoch, $\{A_{mh+j}\}$ and $\{b_{mh+j}\}$ are simplified by $\{A_{j}\}$ and $\{b_{j}\}$ for~$j=1,\dots,m$. In addition, the approximations in the residual system~\eqref{eq:residual} are simplified by $\bar{v}$ and $\bar{b}:=b_1-A_1\bar{v}$.

The following proposition demonstrates that the dynamic subspace constructed in~\cref{alg:LancBiO} within an epoch is indeed an approximate Krylov subspace. 

\begin{proposition}
    At the $j$-th step within an epoch ($j=1,2,\ldots,m-1$), the subspace spanned by the matrix $Q_{j+1}$ in~\cref{alg:LancBiO} satisfies
    \begin{equation*}
        \mathrm{span}(Q_{j+1})\subseteq\mathrm{span}\left\{ A_{1}^{a_1}A_{2}^{a_2}\cdots A_{j}^{a_j}\bar{b}\,\ \big |\ a_s=0\ \text{\emph{or}}\ 1,\ s=1,2,\ldots,j\right\} .
    \end{equation*}
    Specifically, when $A_1=A_2=\cdots=A_j=A$ and $Q_{j+1}$ is of full rank, $\mathrm{span}(Q_{j+1}) = \mathcal{K}_{j+1}\kh{A,\bar{b}}$.
\end{proposition}

Denote $A^*=\nabla_{yy}^2g(x,y^*)$ and $b^*=\nabla_yf(x,y^*)$. Notice that the dynamic Lanczos process in~\cref{alg:LancBiO} centers on $A_j$ instead of $A^*_j$. The subsequent lemma interprets the perturbation analysis for the dynamic Lanczos process in terms of $A^*_j$, which satisfies an approximate {three-term} recurrence with a perturbation term $\delta Q$.
\begin{lemma}
    Suppose Assumptions \ref{assu:f} to \ref{assu:strongg} hold. The dynamic Lanczos process in~\cref{alg:LancBiO} with normalized $q_1$ satisfies
    \begin{equation*}
        A_{j}^{*}Q_j=Q_jT_j+\beta _{j+1}q_{j+1}e_{j}^{\top}+\delta Q_j,\ \  \text{for}\ j=1,2,\dots,m, 
    \end{equation*}
    where $Q_j=\left[ q_1,q_2,\dots ,q_j \right]$, $\delta Q_j=\left[ \delta q_1,\delta q_2,\dots ,\delta q_j \right]$ with $\norm{\delta q_j} \le L_{gyy}\varepsilon_{j}$, and $T_j$ is~a~{$j\times j$} symmetric tridiagonal matrix with diagonal elements $\hkh{\alpha_1,\ldots,\alpha_j}$ and subdiagonal elements~$\hkh{\beta_2,\ldots,\beta_j}$.
\end{lemma}

\subsection{Convergence Analysis}\label{sec:convergence}
To guarantee the stable behavior of the dynamic process, we need the subsequent assumption.
\begin{assumption}\label{assu:y_initial}
    The initialization of $y_1$ in~\cref{alg:LancBiO} satisfies$\left\| y_1-y_{1}^{*} \right\| \le \frac{\sqrt{3}\mu _g}{8\left( m+1 \right) ^3L_{gyy}}$.
\end{assumption}
Similar initialization refinement is used in \citep{hao2023bilevel}, which can be achieved by implementing several gradient descent steps for the smooth and strongly convex lower-level problem. The following lemma reveals that the dynamic process yields an improved solution for the subproblem~\eqref{eq:minimal_res}.
\begin{lemma}\label{lem:descent_res}
 Suppose Assumptions \ref{assu:f}, \ref{assu:g}, \ref{assu:strongg}, \ref{assu:boundfx} and \ref{assu:y_initial} hold. Within each epoch, if we set the step size $\theta\sim\mathcal{O}({1}/{m})$ a constant for $y$, and the step size for $x$ as zero in the first $m_0\sim\Omega(1)$ steps and the others as an appropriate constant $\lambda\sim\mathcal{O}({1}/{m^4})$, we have the following inequality,
\begin{equation*}
\frac{\left\| \bar{r}_j \right\|}{\left\| \bar{r}_0 \right\|}\le 2\sqrt{\tilde{\kappa}\left( j \right)}\left( \frac{\sqrt{\tilde{\kappa}\left( j \right)}-1}{\sqrt{\tilde{\kappa}\left( j \right)}+1} \right) ^j+\sqrt{j}L_{gyy}\varepsilon _j\tilde{\kappa}\left( j \right),
\end{equation*}
where 
$
\tilde{\kappa}\left( j \right) :=\frac{L_{gy}+\frac{2\sqrt{3}}{3}\left( j+1 \right) ^3L_{gyy}\varepsilon _j}{\mu _g-\frac{2\sqrt{3}}{3}\left( j+1 \right) ^3L_{gyy}\varepsilon _j}.
$
\end{lemma}

\begin{theorem}\label{the:convergence}
    Suppose Assumptions \ref{assu:f}, \ref{assu:g}, \ref{assu:strongg}, \ref{assu:boundfx} and \ref{assu:y_initial} hold. Within each epoch, if we set the step size $\theta\sim\mathcal{O}({1}/{m})$ a constant for $y$, and the step size for $x$ as zero in the first $m_0$ steps and the others as an appropriate constant $\lambda\sim\mathcal{O}({1}/{m^4})$, the iterates $\{x_k\}$ generated by \cref{alg:LancBiO} satisfy
        \begin{equation*}
            \frac{m}{K\left( m-m_0 \right)} \hspace{-3mm}\sum_{\footnotesize \substack{k=0,\\
    	 (k\,\emph{\texttt{mod}}\,m)>m_0 }}^K \hspace{-3mm}{\left\| \nabla \varphi \left( x_k \right) \right\| ^2} =  \mathcal{O} \left( \frac{m\lambda ^{-1}}{K\left( m-m_0 \right)} \right),
        \end{equation*}
        where $m_0\sim\Omega(\log m)$ is a constant and $m$ is the subspace dimension.
\end{theorem}

In other words, we prove that the proposed LancBiO is globally convergent, and the average norm square of the hyper-gradient $\norm{\nabla\varphi (x_k)}^2$ achieves $\epsilon$~within $\mathcal{O}(\epsilon^{-1})$ outer iterations.

\section{Numerical Experiments}\label{sec:exp}
In this section, we conduct experiments in the deterministic setting to empirically validate the performance of the proposed algorithms. We test on~a synthetic problem and two deep learning tasks. The selection of parameters and more details of the experiments are deferred to \cref{sec:details_exp}. We have made the code available on \href{https://github.com/UCAS-YanYang/LancBiO}{https://github.com/UCAS-YanYang/LancBiO}.

\begin{figure*}[h]
\begin{minipage}{\textwidth}
    \centering
    \includegraphics[width=0.9\linewidth]{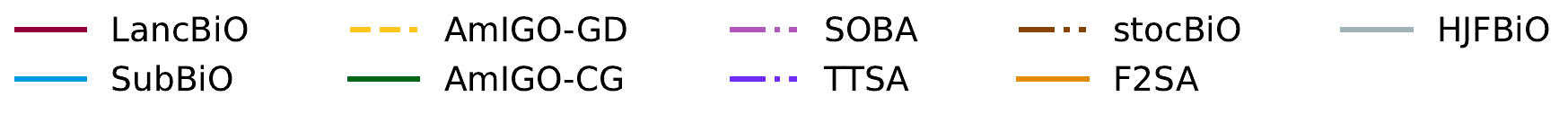}
\end{minipage}
\\
\begin{minipage}{0.33\textwidth}
    \centering
    \includegraphics[width=1\linewidth]{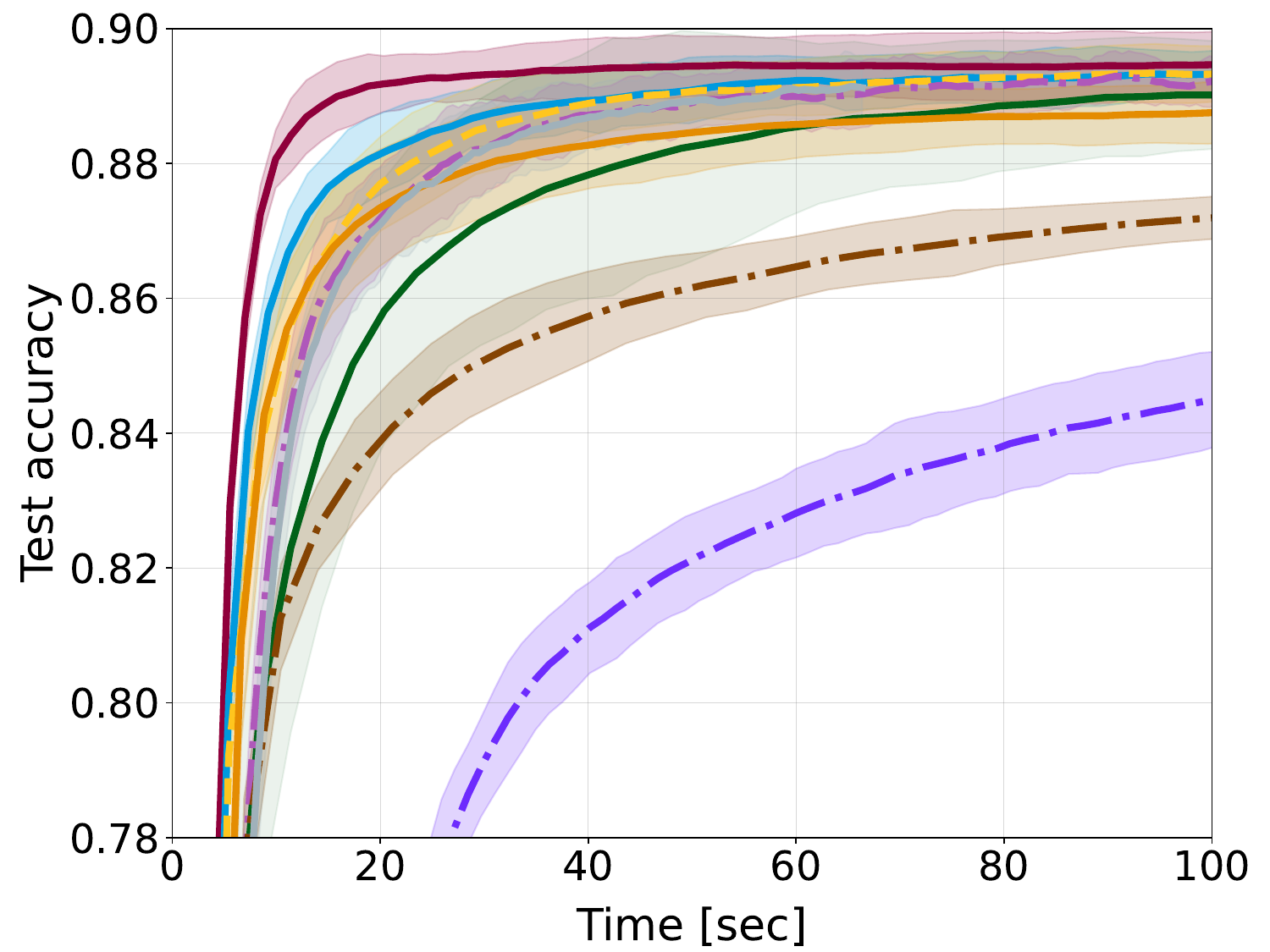}
\end{minipage}
\begin{minipage}{0.33\textwidth}
    \centering
    \includegraphics[width=1\linewidth]{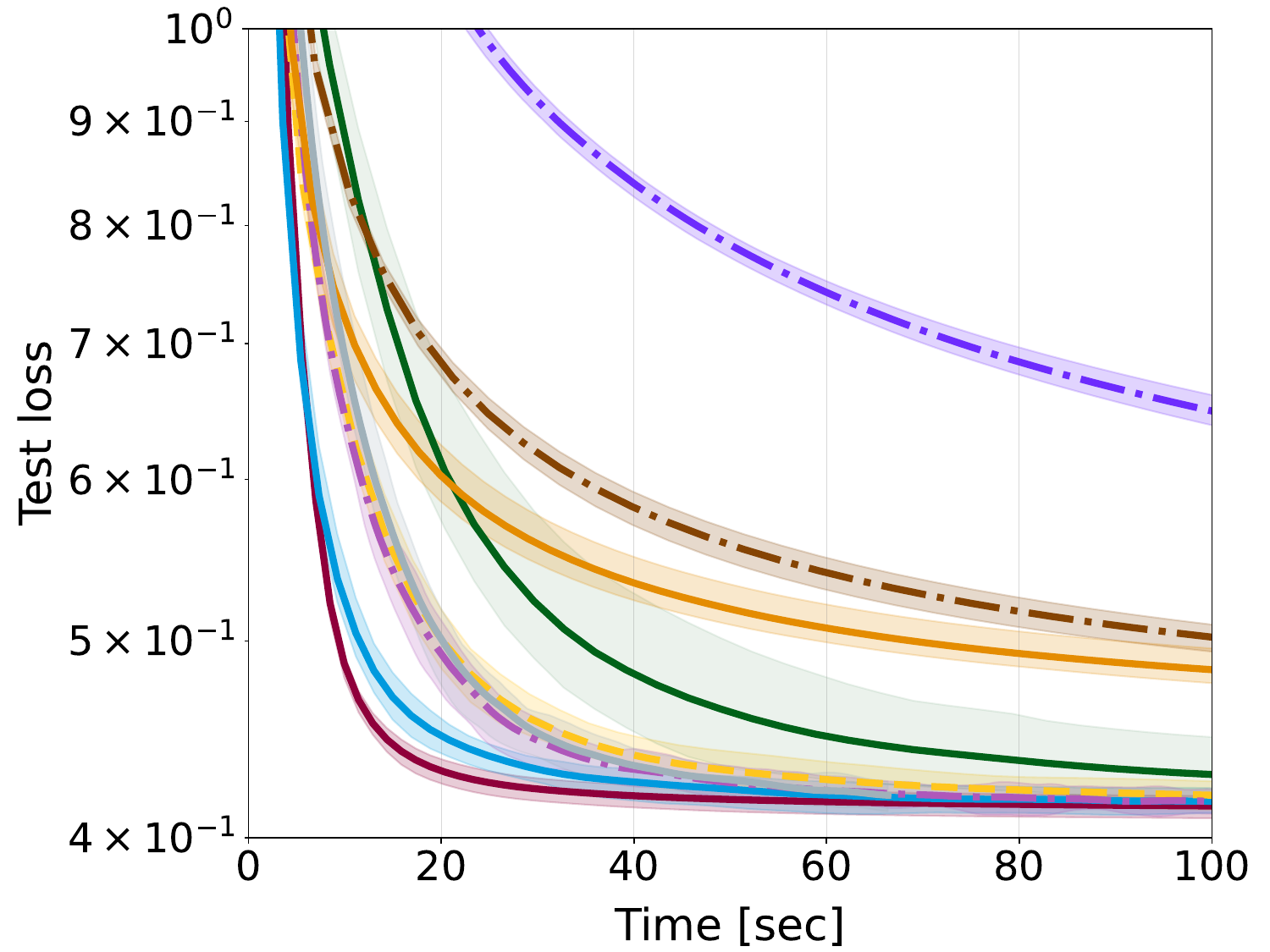}
\end{minipage}
\begin{minipage}{0.33\textwidth}
    \centering
    \includegraphics[width=1\linewidth]{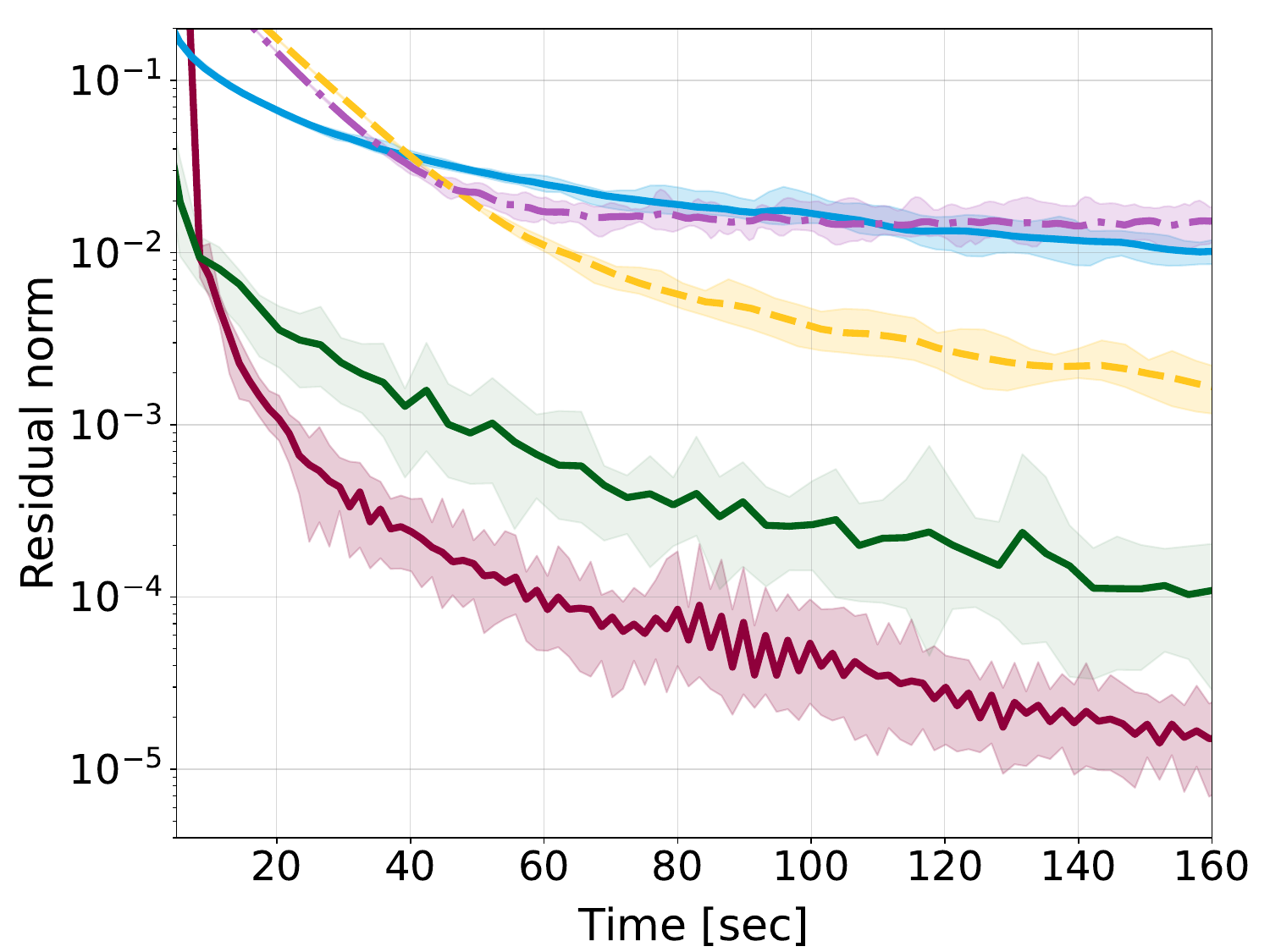}
\end{minipage}
\caption{Comparison of the bilevel algorithms on data hyper-cleaning task when~$p=0.8$. \boldt{Left:} test accuracy; \boldt{Center}: test loss; \boldt{Right}:~residual norm of the linear system, $\norm{A_kv_k-b_k}$.}
\label{fig:clean_detem_mean_compare}
\end{figure*}

\begin{table*}[htbp]
\caption{Comparison of the bilevel algorithms on data hyper-cleaning task across two corruption rates $p=0.5$ and $p=0.8$. The results are averaged over $10$ runs and $\pm$ is followed by the standard deviation. The results are conducted after $40$ and $60$ seconds running time.}
\begin{center}
\begin{scriptsize}
\setlength{\tabcolsep}{3pt} 
\begin{tabular}{lcccccc}
\toprule
\multicolumn{1}{l}{\multirow{2}{*}{Algorithm}} & \multicolumn{3}{c}{$p=0.5$} & \multicolumn{3}{c}{$p=0.8$}\\ \cmidrule(l{2pt}r{2pt}){2-4} \cmidrule(l{2pt}r{2pt}){5-7}
& \multicolumn{1}{c}{Test accuracy (\%)} & \multicolumn{1}{c}{Test loss} & \multicolumn{1}{c}{Residual} & \multicolumn{1}{c}{Test accuracy} (\%) & \multicolumn{1}{c}{Test loss} & \multicolumn{1}{c}{Residual} \\
\midrule
LancBiO   & ${90.35\pm0.1716}$ & ${0.36\pm 0.0028}$ & $\expnumber{1.20}{-4}\pm\expnumber{2.52}{-5}$&  ${89.45\pm0.2470}$ & ${0.42\pm0.0038}$ & ${\expnumber{9.18}{-5}\pm\expnumber{2.77}{-5}}$\\
SubBiO  & $90.21\pm0.2159$&$0.36\pm 0.0035$   &$\expnumber{2.22}{-2}\pm\expnumber{2.63}{-3}$&  $89.22\pm0.2587$& $0.42\pm 0.0050$&$\expnumber{2.49}{-2}\pm\expnumber{1.93}{-3}$\\
AmIGO-GD & $90.16\pm0.2114$&$0.37\pm 0.0044$ &$\expnumber{3.66}{-2}\pm\expnumber{4.00}{-3}$&  $89.14\pm 0.2722$& $0.43\pm0.0044$& $\expnumber{1.07}{-2}\pm\expnumber{8.25}{-4}$\\
AmIGO-CG & $90.06\pm0.2305$ & $0.38\pm 0.0053$  &$\expnumber{5.89}{-4}\pm\expnumber{2.52}{-4}$&  $88.57\pm0.5839$& $0.46\pm0.0176$& $\expnumber{6.32}{-4}\pm\expnumber{3.74}{-4}$\\
SOBA & $90.00\pm0.1811$&$0.37\pm 0.0051$      &$\expnumber{2.74}{-2}\pm\expnumber{8.52}{-3}$&  $88.99\pm0.2661$& $0.42\pm 0.0054$& $\expnumber{1.73}{-2}\pm\expnumber{1.70}{-3}$\\
TTSA &$89.35\pm0.2747$&$0.40\pm 0.0103$       &-&  $82.91\pm0.4516$& $0.74\pm 0.0072$& -\\
stocBiO & $89.20\pm0.1824$&$0.43\pm 0.0033$   &-&  $86.44\pm0.2907$& $0.54\pm0.0064$& -\\
\revise{F2SA} & \revise{$89.78\pm0.1969$} & \revise{$0.40\pm 0.0073$}   & \revise{-} &  \revise{$88.65\pm0.2828$} & \revise{$0.51\pm0.0055$} & \revise{-}\\
\revise{HJFBiO} & \revise{$90.21\pm0.2027$} & \revise{$0.37\pm 0.0048$}   & \revise{-} &  \revise{$89.30\pm0.3594$} & \revise{$0.43\pm0.0040$} & \revise{-}\\
\bottomrule
\end{tabular}
\end{scriptsize}
\end{center}
\label{tab:compare}
\end{table*}

\textbf{Data hyper-cleaning on \revise{three datasets}:} The data hyper-cleaning task~\citep{shaban2019truncated} aims to train a classifier in~a corruption scenario, where the labels of the training data are randomly altered to incorrect classification numbers at~a probability~$p$, referred to as the corruption rate. The results on the MNIST dataset are presented in~\cref{fig:clean_detem_mean_compare} and \cref{tab:compare}. Note that LancBiO is crafted for approximating the Hessian inverse vector product $v^*$, while the two solid methods, TTSA and stocBiO are not. Consequently, with respect to the residual norm of the linear system, i.e., $\norm{A_kv_k-b_k}$, we only compare the results with AmIGO-GD, AmIGO-CG, and SOBA. Observe that the proposed subspace-based LancBiO achieves the lowest residual norm and the best test accuracy, and subBiO is comparable to the other algorithms. Specifically, in~\cref{fig:clean_detem_mean_compare}, the efficiency of LancBiO stems from its accurate approximation of the linear system. Additionally, while AmIGO-CG is also adept~at approximating $v^*$, the results in \cref{tab:compare}~indicate that it tends to yield higher variance. \revise{Moreover, algorithms are also evaluated on the Fashion-MNIST and Kuzushiji-MNIST datasets; see \cref{fig:clean_data_Fashion} and \cref{fig:clean_data_KMNIST}, respectively. The proposed LancBiO performs better than other algorithms and showcases robustness across various datasets.}

\begin{figure*}[tbp]
\centering
\begin{minipage}{0.35\textwidth}
    \centering
    \includegraphics[width=1\linewidth]{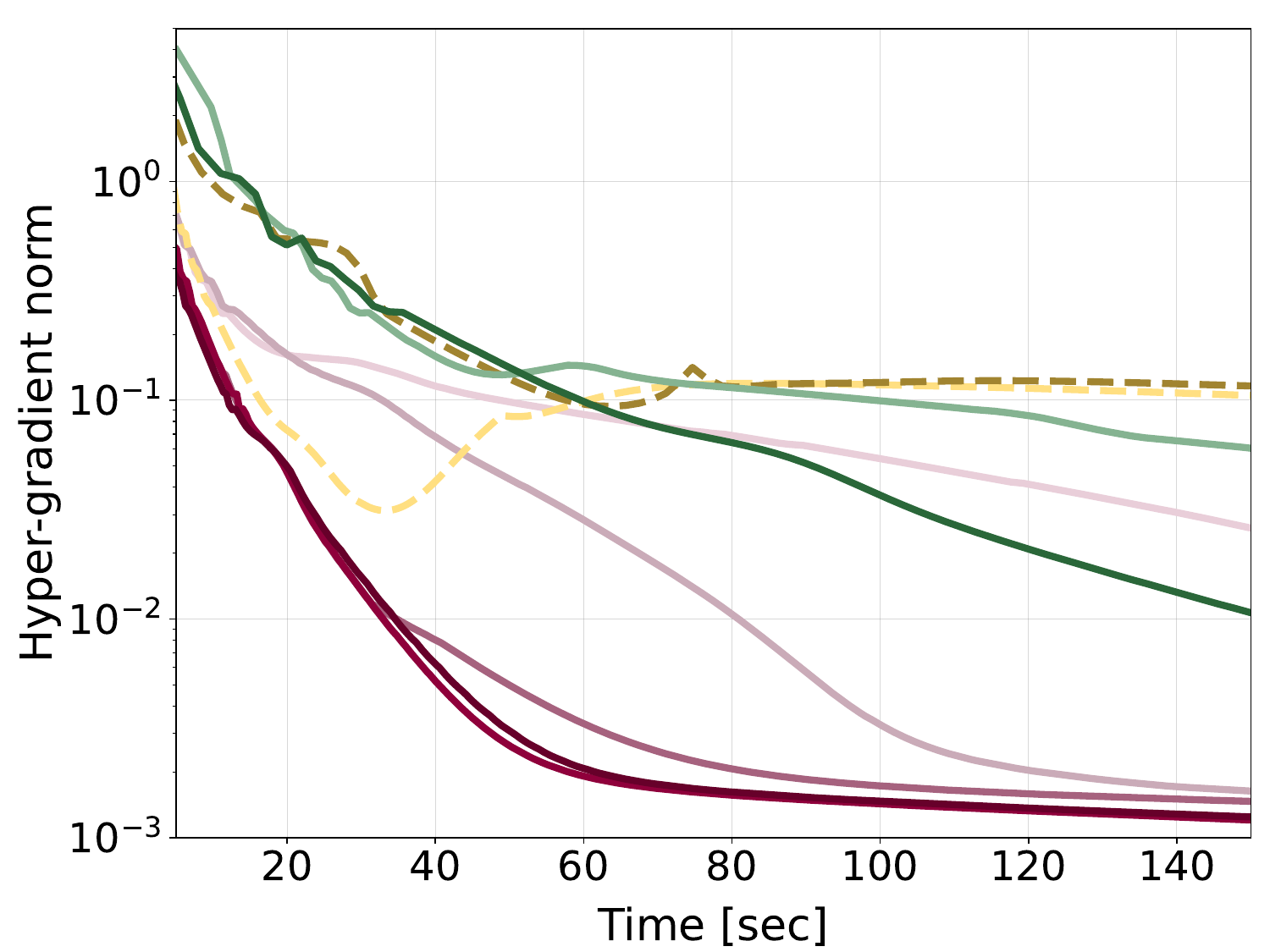}
\end{minipage}
\,
\begin{minipage}{0.35\textwidth}
    \centering
    \includegraphics[width=1\linewidth]{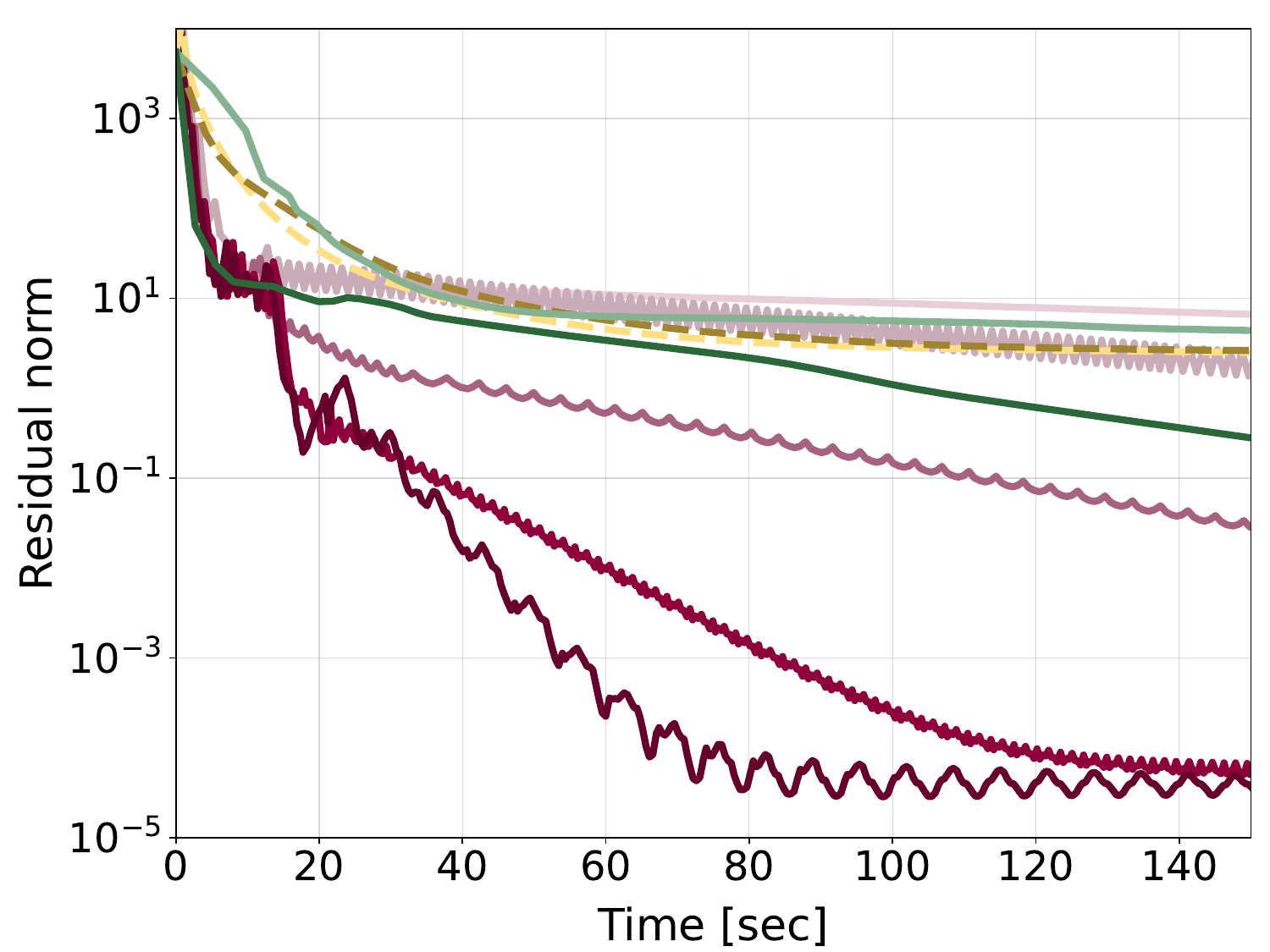}
\end{minipage}
\begin{minipage}{0.20\textwidth}
    \centering
    \includegraphics[width=1\linewidth]{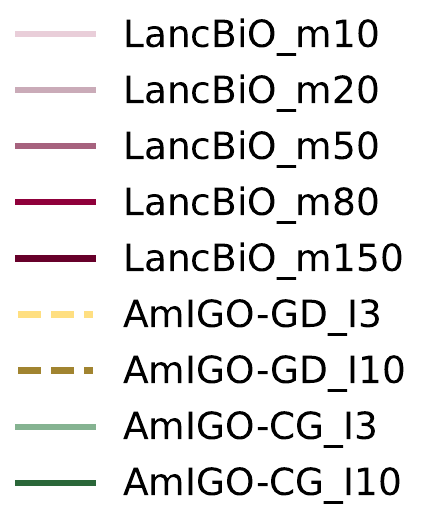}
\end{minipage}
\caption{Influence of the subspace dimension $m$ on LancBiO. {The post-fix of legend represents the subspace dimension $m$ or the inner iteration $I$.} \boldt{Left:} norm  of the hyper-gradient; \boldt{Right}:~residual norm of the linear system, $\norm{A_kv_k-b_k}$.}
\label{fig:sythe_seed4_compare}
\end{figure*}

\textbf{Synthetic problem:} We concentrate on a synthetic bilevel optimization \eqref{eq:standar_bio} with $d_x=d_y=d$ and
\begin{align*}
    f(x,y):=&~c_1\cos\kh{x^\top D_1 y} + \frac{1}{2} \norm{D_2x-y}^2,
    \\
    g(x,y):=&~c_2\sum_{i=1}^{d}{\sin(x_i+y_i)} + \log\kh{\sum_{i=1}^{d}{e^{x_iy_i}}} + \frac{1}{2}y^\top\kh{D_3+G}y.
\end{align*}
It can be seen from \cref{fig:sythe_mean_compare} that LancBiO achieves the final accuracy the fastest, which benefits from the more accurate $v^*$ estimation. \cref{fig:sythe_seed4_compare} illustrates how variations in $m$ and $I$ influence the performance of LancBiO and AmIGO, tested across a range from $10$ to $150$ for $m$, and from $2$ to $10$ for $I$. For clarity, we set the seed of the experiment at $4$, and present typical results to encapsulate the observed trends. It is observed that the increase of $m$ accelerates the decrease in the residual norm, thus achieving better convergence of the hyper-gradient, which aligns with the spirit of the classic Lanczos process. Under the same outer iterations, to attain a comparable convergence property, $I$ for AmIGO-CG should be set to~$10$. Furthermore, given that the number of Hessian-vector products averages at $(1+1/m)$ per outer iteration for LancBiO, whereas AmIGO involves~$I\ge 2$ calculations, it follows that LancBiO is more efficient. \revise{Moreover, to illustrate how the methods scale with increasing dimensions, we present the convergence time and the final upper-level value under different problem dimensions $d=10^{i}, i=1,2,3,4$ in~\cref{tab:scale}. The results demonstrate the proposed methods maintain decent performance across different problem dimensions.}

\begin{figure*}[htbp]
	\begin{minipage}{\textwidth}
		\centering
		\includegraphics[width=0.9\linewidth]{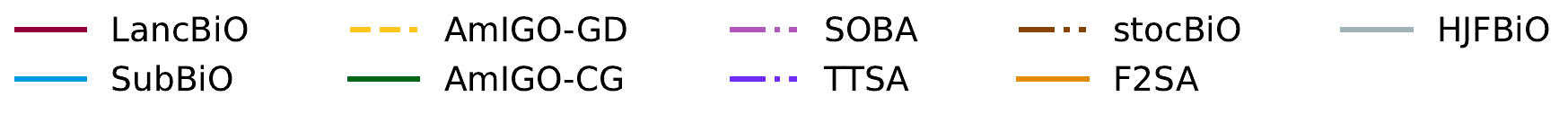}
	\end{minipage}
	\\
	\begin{minipage}{0.33\textwidth}
		\centering
		\includegraphics[width=1\linewidth]{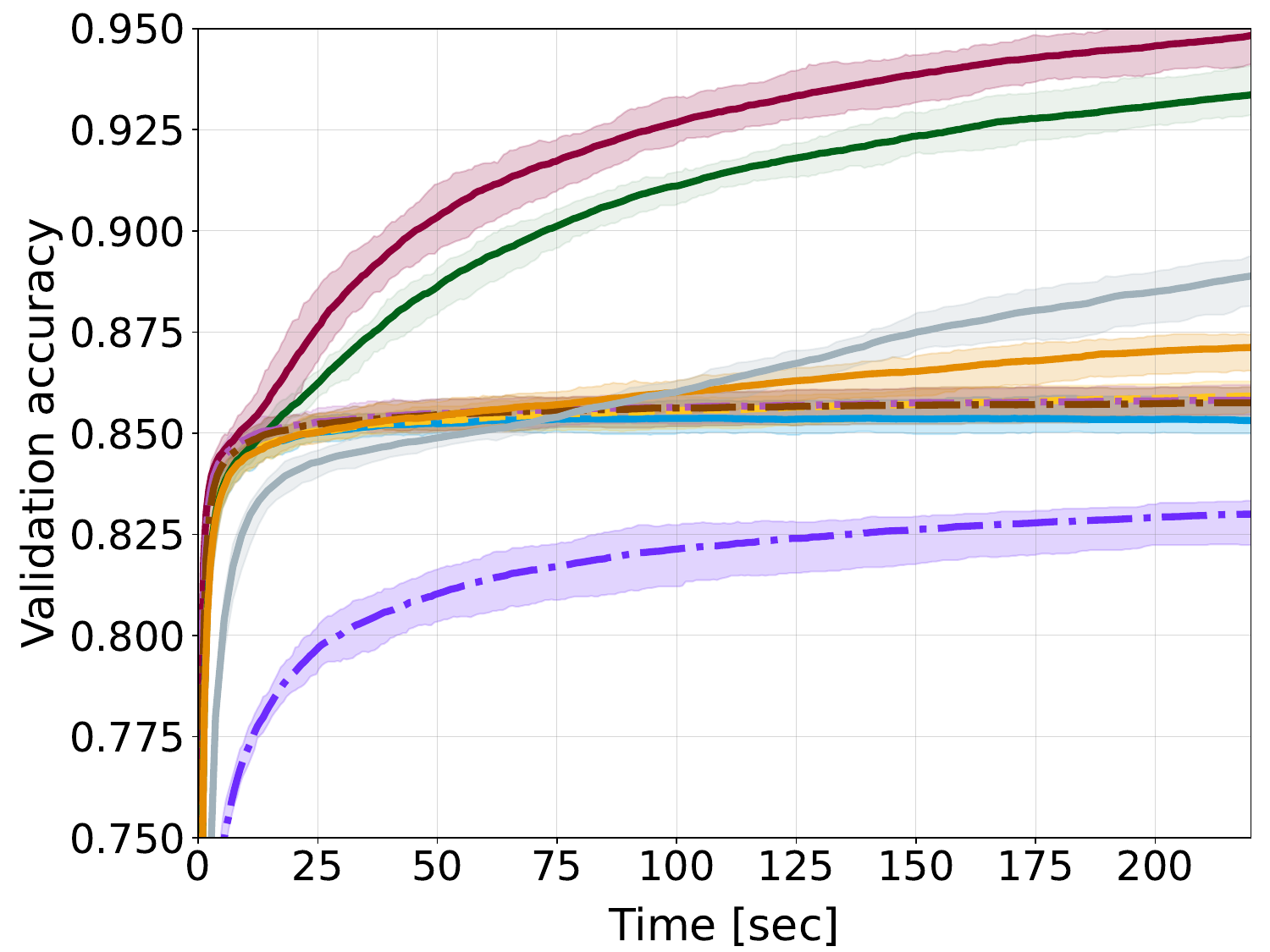}
	\end{minipage}
	\begin{minipage}{0.33\textwidth}
		\centering
		\includegraphics[width=1\linewidth]{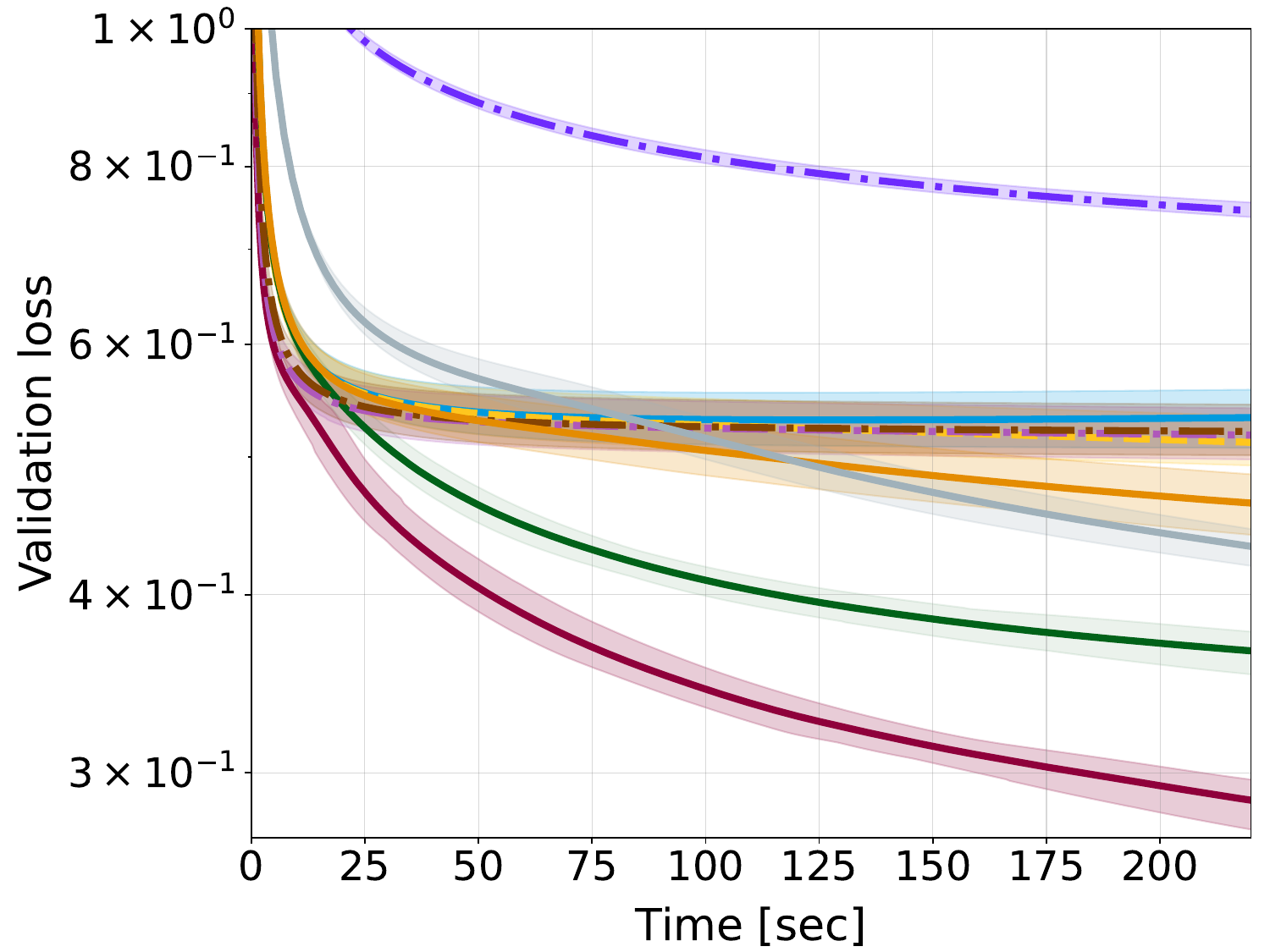}
	\end{minipage}
	\begin{minipage}{0.33\textwidth}
		\centering
		\includegraphics[width=1\linewidth]{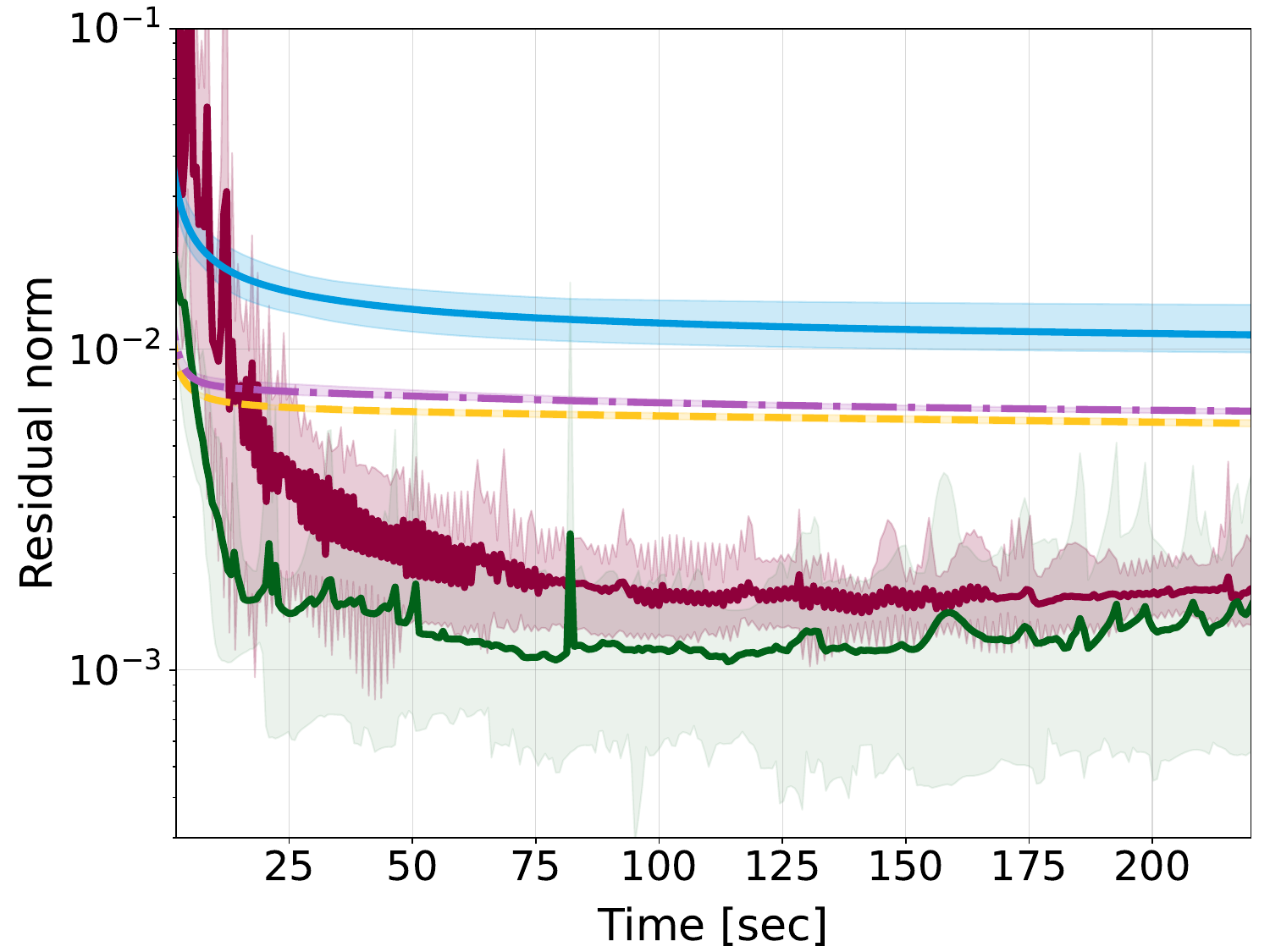}
	\end{minipage}
	\caption{Comparison of the bilevel algorithms on the hyper-parameters selection task. \boldt{Left:} validation accuracy; \boldt{Center}: validation loss; \boldt{Right}:~residual norm of the linear system, $\norm{A_kv_k-b_k}$.}
	\label{fig:hypasele_detem_mean_compare}
\end{figure*}

\textbf{Logistic regression on $20$Newsgroup:} Consider the hyper-parameter selection task on the $20$Newsgroups dataset~\citep{grazzi2020iteration}. The goal is to train a linear classifier $w$ and determine the optimal regularization parameter $\zeta$. As shown in~\cref{fig:hypasele_detem_mean_compare}, AmIGO-CG exhibits slightly better performance in reducing the residual norm. Nevertheless, under the same time, LancBiO implements more outer iterations to update $x$, which optimizes the hyper-function more efficiently.

Generally, to solve standard linear systems, the Lanczos process is recognized for its efficiency and versatility over gradient descent methods. LancBiO, in~a sense, reflects this principle in the context of bilevel optimization, underscoring the effectiveness of the dynamic Lanczos-aided approach.


\newpage
\subsubsection*{Acknowledgments}
\revise{Bin Gao was supported by the Young Elite Scientist Sponsorship Program by CAST. Ya-xiang Yuan was supported by the National Natural Science Foundation of China (grant No. 12288201). The authors are grateful to the Program and Area Chairs and Reviewers for their valuable comments and suggestions.}

\bibliographystyle{plainnat}
\bibliography{lib}

\clearpage
\appendix

{\Large \textbf{Appendix}}
\section{Related Work in Bilevel Optimization}\label{sec:bilevel}
A variety of bilevel optimization algorithms are based on reformulation \citep{liu2021value,yao2024Hessianfree}. These algorithms involve transforming the lower-level problem into~a set of constraints, such as the optimal conditions of the lower-level problem~\citep{dempe2012bilevel,li2023wolfe}, or the optimal value condition \citep{outrata1990numerical,ye1995optimality,ye2010new,dempe2013bilevel,lin2014solving,xu2014smoothing}. Furthermore, incorporating the constraints of the reformulated problem as the penalty function into the upper-level objective inspires~a series of algorithms \citep{liu2022bome,hu2023cdb,kwon2023f2sa,lu2023penalty,kwon2023penalty}. Another category of methods in bilevel optimization is the iterative differentiation (ITD) based method \citep{maclaurin2015gradient,franceschi2017forward,shaban2019truncated,grazzi2020iteration,liu2020bda, ji2021stocbio}, which takes advantage of the automatic differentiation technique. Central to this approach is the construction of~a computational graph during each outer iteration, achieved by solving the lower-level problem. This setup facilitates the approximation of the hyper-gradient through backpropagation, and it is noted that parts of these methods share a unified structure, characterized by recursive equations \citep{ji2021stocbio,li2022fsla,zhang2023introduction2Bi}. 
The approximate implicit differentiation (AID) treats the lower-level variable as~a function of the upper-level variable. It calculates the hyper-gradient to implement alternating gradient descent between the two levels \citep{ghadimi2018approximation,ji2021stocbio,chen2022stable,dagreou2022soba,li2022fsla,hong2023ttsa}. \revise{Moreover, extending the spirits of AID to distributed settings has garnered increasing interest in recent years\citep{kong2024decentralized,he2024distributed,zhu2024sparkle}.}

\section{Krylov Subspace and Lanczos Process}\label{sec:krylov_lanczos}
Krylov subspace~\citep{krylov1931numerical} is fundamental in numerical linear algebra~\citep{parlett1998symmetric,saad2011numerical,golub2013matrix} and nonlinear optimization~\citep{yuan2014review,liu2021subspace}, specifically in the context of solving large linear systems and eigenvalue problems. In this section, we will briefly introduce the Krylov subspace and the Lanczos process, and recap some important properties; readers are referred to~\citet{saad2011numerical,golub2013matrix} for more details.

An $N$-dimensional Krylov subspace generated by a matrix~$A$ and~a vector~$b$ is defined as follows,
\begin{equation*}
    \mathcal{K}_N(A, b):=\operatorname{span}\left\{b, A b, A^2 b, \ldots, A^{N-1} b\right\},
\end{equation*}
and the sequence of vectors $\left\{b, A b, A^2 b, \ldots, A^{N-1} b\right\}$ forms the basis for it. The Krylov subspace is widely acknowledged for its favorable properties in various aspects, including approximating eigenvalues~\citep{kuczynski1992eigenvec}, solving the regularized nonconvex quadratic problems~\citep{gould1999tr,zhang2017generalizedLanc,carmon2018krylov}, and reducing computation cost~\citep{brown1990hybridkrylov,bellavia2001ngmeres,liu2013limited,jiang2024krylov}.

The Lanczos process~\citep{lanczos1950lanczos} is~an algorithm that exploits the structure of the Krylov subspace when~$A$ is symmetric. Specifically, in the~$j$-th step of the Lanczos process, we can efficiently maintain~an orthogonal basis~$Q_j$ of $\mathcal{K}_j(A,b)$, so that $T_j=Q_j^\top AQ_j$ is tridiagonal, which means~a tridiagonal matrix~$T_j$ approximates~$A$ in the Krylov subspace. Consequently, it allows to solve the minimal residual problem or the eigenvalue problem efficiently within the Krylov subspace. There are several equivalent variants of the Lanczos process~\citep{paige1971thesis,paige1976error,meurant2006lanczos}, and we follow the update rule as shown in~\cref{alg:standard_lanczos}.
\begin{algorithm}[htbp]
    \caption{Lanczos process}
    \label{alg:standard_lanczos}
    \begin{algorithmic}[1]
        \REQUIRE dimension $m$, matrix $A\in\mathbb{R}^{n\times n}$, initial vector $b\in\mathbb{R}^n$
        \STATE {\bfseries Initialization:} $q_1=\frac{b}{\norm{b}}$, $q_0=\bf{0}$, $\beta_1=0$, $Q_0=T_0=\texttt{Empty Matrix}$
        \FOR{$j = 1, 2, \dots, m$}
            \STATE $u_j=Aq_j-\beta _jq_{j-1}$
            \STATE $\alpha _j=q_{j}^{\top}u_j$
            \STATE $\omega _j=u_j-\alpha _jq_j$
            \STATE $ \beta _{j+1}=\left\| \omega _j\right\|$
            \STATE $q_{j+1}=\omega _j/\beta _{j+1}$
            \STATE $Q_j = [Q_{j-1}\ q_j]$
            \STATE $T_j={\setlength{\arraycolsep}{1.pt} \left( \begin{array}{ccc:c}
            	&		&		&		\\
            	&		\phantom{{\small ss}}{\Large T_{j-1}}&		&		\\
            	&		&		&		\beta _j\vphantom{{\small \frac{2}{\frac{2}{3}}}}\\
                    \hdashline
            	&		&		\beta _j\phantom{{\small s}}&		\alpha _j\\
            \end{array} \right) }$
        \ENDFOR
        \ENSURE $T_m,Q_m,\norm{b}e_1$
    \end{algorithmic}
\end{algorithm}
We now present several key properties of the Krylov subspace and the Lanczos Process from~\citet{saad2011numerical}.
\begin{definition}\label{def:minimal}
    The minimal polynomial of a vector $v\in \mathbb{R}^n$ with respect to a matrix $A\in\mathbb{R}^{n\times n}$ is defined as the non-zero monic polynomial $p$ of the lowest degree such that $p(A)v = 0$, where a monic polynomial is a non-zero univariate polynomial with the coefficient of highest degree equal to~$1$.
\end{definition}

\begin{remark}
The degree of the minimal polynomial $p$ does not exceed $n$ because the set of $n\!+\!1$ vectors $\{A^{n}v,A^{n-1}v,\ldots,A^2v,Av,v\}$ is linearly dependent.
\end{remark}

\begin{remark}
Suppose the minimal polynomial of a vector $v$ with respect to a matrix $A$ is 
\begin{equation*}
    p(x) = x^m+c_{m-1}x^{m-1}+\cdots+c_2x^2+c_1x+c_0,
\end{equation*}
and has a degree of $m$. If $c_0\neq 0$ and $A$ is invertible, by \cref{def:minimal},
\begin{equation*}
   A^mv+c_{m-1}A^{m-1}v+\cdots+c_2A^2v+c_1Av+c_0v=0,
\end{equation*}
multiply both sides of the equation by $A^{-1}$ and rearrange the equation,
\begin{equation*}
   A^{-1}v=-\frac{1}{c_0}\left( A^{m-1}v+c_{m-1}A^{m-2}v+\cdots+c_2Av+c_1v \right).
\end{equation*}
In other words, $A^{-1}v$ belongs to the Krylov subspace $\mathcal{K}_m(A,v)$.
\end{remark}

\begin{proposition}\label{pro:Lanc_process}
    Denote the $n \times j$ matrix with column vectors $q_1, \ldots, q_j$ by $Q_j$  and the $j \times j$ tridiagonal matrix by $T_j$, all of which are generated by \cref{alg:standard_lanczos}. Then the following {three-term} recurrence holds.
    \begin{equation*}
    \begin{aligned}
    A Q_j & =Q_j T_j+ \beta_{j+1}q_{j+1} e_j^\top, \\
    Q_j^\top A Q_j & =T_j .
    \end{aligned}
    \end{equation*}
\end{proposition}
Based on \cref{pro:Lanc_process}, the Lanczos process is illustrated in \cref{fig:three_recurrence}.
\begin{figure}[htbp]
    \centering
    \includegraphics[width=.85\linewidth]{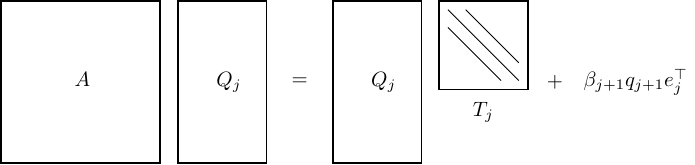}
    \caption{Classic three-term recurrence}
    \label{fig:three_recurrence}
\end{figure}

\section{Dynamic Lanczos Subroutine}\label{sec:dLanczosSub}
This section lists the \texttt{DLanczos} subroutine~(\cref{alg:dlanczos}) invoked in the LancBiO framework~(\cref{alg:LancBiO}). One of the main differences between \cref{alg:standard_lanczos} and \cref{alg:dlanczos} is that \cref{alg:standard_lanczos} represents the entire $m$-step Lanczos process, while
\cref{alg:dlanczos} serves as a one-step subroutine. Specifically, LancBiO invokes \texttt{DLanczos} once in each outer iteration (line~11 in \cref{alg:LancBiO}), expanding both $T$ and $Q$ by one dimension. Consequently, the inputs of \cref{alg:dlanczos} are not indexed to avoid confusion, with their corresponding variables $(T_{k-1},Q_{k-1},A_k,\beta_k)$ in \cref{alg:LancBiO} evolving across outer iterations indexed by~$k$. Another difference lies in the dynamic property of the \texttt{DLanczos} subroutine, i.e., the matrix~$A_k$ passed during each invocation varies in \cref{alg:LancBiO}, while the classic Lanczos process (\cref{alg:standard_lanczos}) employs~a~static~$A$.
\begin{algorithm}[htbp]
    \caption{Dynamic Lanczos subroutine for LancBiO (\texttt{DLanczos})}
    \label{alg:dlanczos}
    \label{alg:LancBiOSubroutine}
    \begin{algorithmic}[1]
        \REQUIRE tridiagonal matrix $T$, basis matrix $Q$ with $j$ columns, Hessian matrix $A$, and $\beta$
        \IF{ $j=1$}
            \STATE $[q_1]=Q,\ q_{-1}={\bf{0}}$	
        \ELSE
            \STATE $[q_1,q_2,\ldots,q_j] = Q$
        \ENDIF
        \STATE $u_j=Aq_j-\beta q_{j-1}$
        \STATE $\alpha_j = q_j^\top u_j$
        \STATE $\omega_j=u_j-\alpha_j q_j$
        \STATE $\beta_{j+1} = \norm{\omega_j}$
        \STATE $q_{j+1} = \omega_{j}/{\beta_{j+1}}$
        \STATE $Q_{j+1}=[Q\ q_{j+1}]$
        \STATE $T_{j+1}={\setlength{\arraycolsep}{1pt} \left( \begin{array}{ccc:c}
            &		&		&		\\
            &		\phantom{{\small ss}}{\Large T}&		&		\\
            &		&		&		\beta \vphantom{{\small \frac{2}{\frac{2}{3}}}}\\
                \hdashline
            &		&		\beta \phantom{{\small s}}&		\phantom{{\small s}}\alpha_j\\
        \end{array} \right) }$
        \ENSURE $T_{j+1},Q_{j+1},\beta_{j+1}$
    \end{algorithmic}
\end{algorithm}

\section{{Extending LancBiO to Non-convex Lower-level Problem}}\label{app:LLPL}
The Lanczos process is known for its efficiency of constructing Krylov subspaces and is capable of solving indefinite linear systems \citep{greenbaum1999lanczosindefinite}. In this section, we will briefly demonstrate that the dynamic Lanczos-aided bilevel optimization framework, LancBiO, can also handle lower-level problems with the indefinite Hessian.

Suppose $A$ is invertible, and consider solving a standard linear system
\begin{equation*}
    Ax=b,
\end{equation*}
with initial ponit $x_0$, initial residual $r_0=b-Ax_0$ and initial error $e_0=A^{-1}b-x_0$. If the matrix $A$ is positive-definite, the classic Lanczos algorithm is equivalent to the Conjugate Gradient (CG) algorithm \citep{hestenes1952cg}, both of which minimize the $A$-norm of the error in an affine space~\citep{greenbaum1997iterative,meurant2006lanczos}, i.e., at the $m$-th step,
\begin{equation*}
    x_m=\!\!\!\argmin_{x\in x_0+\mathcal{K}_m(A,b)}\norm{A^{-1}b-x}_A.
\end{equation*}
If the matrix $A$ is not positive-definite, MINRES \citep{paige1975minres} is the algorithm recognized to minimize the 2-norm of the residual in an affine space~\citep{greenbaum1997iterative,meurant2006lanczos}, i.e., at the $m$-th step,
\begin{equation}\label{eq:minres_subproblem}
    x_m=\!\!\!\argmin_{x\in x_0+\mathcal{K}_m(A,b)}\norm{b-Ax}.
\end{equation}
Additionally, based on $Q_m$ as the basis of the Krylov subspace $\mathcal{K}_m(A,b)$, $T_m$ as the projection of~$A$ onto $\mathcal{K}_m(A,b)$, and the {three-term} recurrence
\begin{equation*}
    A Q_m = Q_m T_m+ \beta_{m+1}q_{m+1} e_m^\top,
\end{equation*}
we can rewrite \eqref{eq:minres_subproblem} as
\begin{equation*}
    x_m = x_0 + Q_mc_m,
\end{equation*}
with
\begin{align*}
    c_m = &\argmin_c\norm{r_0-AQ_mc}
    \\
    = &\argmin_c\norm{r_0-Q_{m+1}T_{m+1,m}c}
    \\
    = &\argmin_c\norm{Q_{m+1}\left(\norm{r_0}e_1-T_{m+1,m}c\right)}
    \\
    = &\argmin_c\norm{\norm{r_0}e_1-T_{m+1,m}c},
\end{align*}
where
\begin{equation*}
    T_{m+1,m}:=\left[ \begin{array}{c}
	T_m\\
	\beta _{m+1}e_{m}^{\top}\\
    \end{array} \right]. 
\end{equation*}

In the spirit of MINRES, to address the bilevel problem where the lower-level problem exhibits an indefinite Hessian, the framework LancBiO~(\cref{alg:LancBiO}) requires only~a minor modification. Specifically, line~13 in~\cref{alg:LancBiO}, , which solves a small-size tridiagonal linear system, will be replaced by solving~a low-dimensional least squares problem
\begin{equation*}
    c_k=\argmin_c \norm{\norm{r_k}e_1-T_{k+1,k}c}^2,
\end{equation*}
and computing the correction
\begin{equation*}
    \Delta v_k=Q_kc_k,
\end{equation*}
where
\begin{equation*}
    T_{k+1,k}:=\left[ \begin{array}{c}
	T_k\\
	\beta _{k+1}e_{k}^{\top}\\
    \end{array} \right]. 
\end{equation*}

\section{Proof of Smoothness of $y^*$ and $\varphi$}\label{sec:app_smooth_phi}
To ensure completeness, we provide detailed proofs for the preliminary lemmas that characterize the smoothness of the lower level solution $y^*$ and the hyper-objective $\varphi$.

\begin{lemma}
    Under the Assumptions \ref{assu:g} and \ref{assu:strongg}, $y^*(x)$ is $\frac{L_{gx}}{\mu_g}$-Lipschitz continuous, i.e., for any $x_1,x_2\in\mathbb{R}^{d_x}$,
    \begin{equation*}
        \norm{y^*(x_1)-y^*(x_2)}\le \frac{L_{gx}}{\mu_g}\norm{x_1-x_2}.
    \end{equation*}
\end{lemma}
\begin{proof}
    The assunption that $\nabla_xg(x,y)$ is $L_{gx}$-Lipschitz reveals $\norm{\nabla^2_{xy}g(x,y)}\le L_{gx}$. Then 
    \begin{equation*}
        \left\| \nabla y^*\left( x \right) \right\| =\left\| \nabla _{xy}^{2}g\left( x,y \right) \left[ \nabla _{yy}^{2}g\left( x,y \right) \right] ^{-1} \right\| \le \left\| \nabla _{xy}^{2}g\left( x,y \right) \right\| \left\| \left[ \nabla _{yy}^{2}g\left( x,y \right) \right] ^{-1} \right\| \le \frac{L_{gx}}{\mu _g},
    \end{equation*}
    since $g(x,\cdot)$ is $\mu_g$-strongly convex.
\end{proof}

\begin{lemma}\label{lem:hypersmooth}
    Under the Assumptions \ref{assu:f}, \ref{assu:g} and \ref{assu:strongg}, the hyper-gradient $\nabla\varphi(x)$ is $L_{\varphi}$-Lipschitz continuous, i.e., for any $x_1,x_2\in\mathbb{R}^{d_x}$,
    \begin{equation*}
        \norm{\nabla\varphi(x_1)-\nabla\varphi(x_2)}\leq L_\varphi\norm{x_1-x_2},
    \end{equation*}
    where $L_\varphi=\left( 1+\frac{L_{gx}}{\mu _g} \right) \left( L_{fx}+\frac{L_{gx}L_{fy}+L_{gxy}C_{fy}}{\mu _g}+\frac{L_{gx}C_{fy}L_{gyy}}{\mu _{g}^{2}} \right) $.
\end{lemma}
\begin{proof}
By combining 
\begin{equation*}
    \begin{aligned}
        &\left( A_{1}^{*} \right) ^{-1}b_{1}^{*}-\left( A_{2}^{*} \right) ^{-1}b_{2}^{*}
        \\
        =&\left( A_{1}^{*} \right) ^{-1}b_{1}^{*}-\left( A_{1}^{*} \right) ^{-1}b_{2}^{*}+\left( A_{1}^{*} \right) ^{-1}b_{2}^{*}-\left( A_{2}^{*} \right) ^{-1}b_{2}^{*}
        \\
        =&\left( A_{1}^{*} \right) ^{-1}\left( b_{1}^{*}-b_{2}^{*} \right) +\left( A_{1}^{*} \right) ^{-1}\left( A_{2}^{*}-A_{1}^{*} \right) \left( A_{2}^{*} \right) ^{-1}b_{2}^{*}
    \end{aligned}
\end{equation*}
with the properties revealed by Assumptions \ref{assu:f} \ref{assu:g} and \ref{assu:strongg}, we can derive
\begin{equation*}
    \left\| \left( A_{1}^{*} \right) ^{-1}b_1-\left( A_{2}^{*} \right) ^{-1}b_2 \right\| \le \frac{L_{fy}}{\mu _g}\left( 1+\frac{L_{gx}}{\mu _g} \right) \left\| x_1-x_2 \right\| +\frac{C_{fy}L_{gyy}}{\mu _{g}^{2}}\left( 1+\frac{L_{gx}}{\mu _g} \right) \left\| x_1-x_2 \right\|.
\end{equation*}
In a similar way, the subsequent decomposition holds,
\begin{equation*}
    \begin{aligned}
        \nabla \varphi \left( x_1 \right) -\nabla \varphi \left( x_2 \right) =&\left( \nabla _xf\left( x_1,y^*\left( x_1 \right) \right) -\nabla _xf\left( x_2,y^*\left( x_2 \right) \right) \right) 
        \\
        &-\nabla^2 _{xy}g\left( x_1,y^*\left( x_1 \right) \right) \left( A_{1}^{*} \right) ^{-1}b_{1}^{*}+\nabla^2 _{xy}g\left( x_1,y^*\left( x_1 \right) \right) \left( A_{2}^{*} \right) ^{-1}b_{2}^{*}
        \\
        &+\nabla^2 _{xy}g\left( x_2,y^*\left( x_2 \right) \right) \left( A_{2}^{*} \right) ^{-1}b_{2}^{*}-\nabla ^2_{xy}g\left( x_1,y^*\left( x_1 \right) \right) \left( A_{2}^{*} \right) ^{-1}b_{2}^{*}.
    \end{aligned}
\end{equation*}
It follows that
\begin{equation*}
    \begin{aligned}
        \left\| \nabla \varphi \left( x_1 \right) -\nabla \varphi \left( x_2 \right) \right\| \le&~L_{fx}\left( 1+\frac{L_{gx}}{\mu _g} \right) \left\| x_1-x_2 \right\| 
        \\
        &+L_{gx}\left( 1+\frac{L_{gx}}{\mu _g} \right) \left( \frac{L_{fy}}{\mu _g}+\frac{C_{fy}L_{gyy}}{\mu _{g}^{2}} \right) \left\| x_1-x_2 \right\| 
        \\
        &+L_{gxy}\frac{C_{fy}}{\mu _g}\left( 1+\frac{L_{gx}}{\mu _g} \right) \left\| x_1-x_2 \right\| 
        \\
        =&~L_{\varphi}\left\| x_1-x_2 \right\|,
    \end{aligned}
\end{equation*}
where $L_{\varphi}:=\left( 1+\frac{L_{gx}}{\mu _g} \right) \left( L_{fx}+\frac{L_{gx}L_{fy}+L_{gxy}C_{fy}}{\mu _g}+\frac{L_{gx}C_{fy}L_{gyy}}{\mu _{g}^{2}} \right) $.
\end{proof}

\section{Properties of Dynamic Subspace in \cref{sec:subspaceerror}}\label{sec:proof_lanc}
In this section, we focus on the properties of the basis matrix $Q$ and the tridiagonal matrix $T$ constructed within each epoch of the dynamic Lanczos process. 
Denote
\begin{equation*}
    A_{k}^{*}=\nabla _{yy}^{2}g\left( x_k,y^*_k \right)\text{ and }b_k^*=\nabla _yf\left( x_k,y^*_k\right).
\end{equation*}
An \emph{epoch} is constituted of~a complete $m$-step dynamic Lanczos process between two restarts, namely, after~$h$ epochs, the number of outer iterations is~$mh$. Given the outer iterations $k=mh+j$ for $j=1,2,\ldots,m$, we denote 
\begin{equation*}
 \varepsilon _{st}^{h}:=\left( 1+\frac{L_{gx}}{\mu _g} \right) \left\| x_{mh+s}-x_{mh+t} \right\| +\left\| y_{mh+s}-y_{mh+s}^{*} \right\|   
\end{equation*}
for $s,t=1,2,\dots,m$ and 
\begin{equation*}
    \varepsilon _{j}^{h}:=\max_{1\le s,t\le j} \varepsilon _{st}^{h},
\end{equation*}
serving as the accumulative difference. For brevity, we omit the superscript where there is no ambiguity, and we are slightly abusing of notation that at the current epoch, $\{A_{mh+j}\}$ and $\{b_{mh+j}\}$ are simplified by $\{A_{j}\}$ and $\{b_{j}\}$ for~$j=1,\dots,m$. In addition, the approximations in the residual system~\eqref{eq:residual} are simplified by~$\bar{v}$ and $\bar{b}:=b_1-A_1\bar{v}$.

We rewrite the dynamic update rule from \cref{sec:lancbio}
\begin{align}
    u_j&=A_jq_j-\beta _jq_{j-1}, \label{eq:lanc_u}
    \\
    \alpha _j&=q_{j}^{\top}u_j,    \label{eq:lanc_alpha}
    \\
    \omega _j&=u_j-\alpha _jq_j,\label{eq:lanc_omega}
    \\
    \beta _{j+1}&=\left\| \omega _j \right\| ,\label{eq:beta}
    \\
    q_{j+1}&=\omega _j/\beta _{j+1},\label{eq:lanc_q}
\end{align}
for $j=1,2,\ldots,m$ with $q_0=\mathbf{0}\ ,\beta _1=0$ and $Q_1=q_1=\bar{b}/\norm{\bar{b}}$.
The following proposition characterizes that the dynamic subspace constructed in~\cref{alg:LancBiO} within an epoch is indeed an approximate Krylov subspace. 
\begin{proposition}
    At the $j$-th step within an epoch ($j=1,2,\ldots,m-1$), the subspace spanned by the matrix $Q_{j+1}$ in~\cref{alg:LancBiO} satisfies
    \begin{equation}\label{eq:Q_characterize}
        \mathrm{span}(Q_{j+1})\subseteq\mathrm{span}\left\{ A_{1}^{a_1}A_{2}^{a_2}\cdots A_{j}^{a_j}\bar{b}\,\left|\! \begin{array}{c}
	a_s=0\ \text{\emph{or}}\ 1\\
	\forall s=1,2,\ldots,j\\
        \end{array}\!\right.\!\right\} .
    \end{equation}
    Specifically, when $A_1=A_2=\cdots=A_j=A$ and $Q_{j+1}$ is of full rank,
    \begin{equation*}
        \mathrm{span}(Q_{j+1}) = \mathcal{K}_{j+1}\kh{A,\bar{b}}
    \end{equation*}
\end{proposition}
\begin{proof}
    Note that $Q_1=\mathrm{span}\{q_1\}$ with $q_1 = \frac{\bar{b}}{\norm{\bar{b}}}$ satisfies \eqref{eq:Q_characterize}. We will give a proof by induction. Suppose for $i=1,2,\ldots,j$, it holds that
    \begin{equation*}
        \mathrm{span}(Q_{i+1})\subseteq\mathrm{span}\left\{ A_{1}^{a_1}A_{2}^{a_2}\cdots A_{i}^{a_i}\bar{b}\,\left|\! \begin{array}{c}
	a_s=0\ \text{or}\ 1\\
	\forall s=1,2,\ldots,i\\
        \end{array}\!\right.\!\right\} .
    \end{equation*}
    By the dynamic Lanczos process, it yields
    \begin{equation}\label{eq:updateq}
        q_{j+2}=\frac{1}{\beta _{j+2}}\left( A_{j+1}q_{j+1}-\beta _{j+1}q_{j}-\alpha _{j+1}q_{j+1} \right).
    \end{equation}
    Since 
    \begin{equation*}
            q_{j+1}\in \mathrm{span}\left\{ A_{1}^{a_1}A_{2}^{a_2}\cdots A_{j}^{a_j}\bar{b}\,\,\left| \begin{array}{c}
    	a_s=0\ \text{or}\,\,a_s=1\\
    	\forall s=1,2,\ldots,j\\
    \end{array} \right. \right\},
    \end{equation*}
    then we have
    \begin{equation*}
            A_{j+1}q_{j+1} \in \mathrm{span}\left\{ A_{1}^{a_1}A_{2}^{a_2}\cdots A_{j+1}^{a_{j+1}}\bar{b}\,\,\left| \begin{array}{c}
    	a_s=0\ \text{or}\,\,a_s=1\\
    	\forall s=1,2,\ldots,j+1\\
    \end{array} \right. \right\}.
    \end{equation*}
    It follows from \eqref{eq:updateq} that
    \begin{equation*}
            q_{j+2} \in \mathrm{span}\left\{ A_{1}^{a_1}A_{2}^{a_2}\cdots A_{j+1}^{a_{j+1}}\bar{b}\,\,\left| \begin{array}{c}
    	a_s=0\ \text{or}\,\,a_s=1\\
    	\forall s=1,2,\ldots,j+1\\
    \end{array} \right. \right\}.
    \end{equation*}
    By induction, we complete the proof.
\end{proof}

Although we can estimate the difference between the basis of the above two subspaces, it is noted that the Krylov subspaces can be very sensitive to small perturbation \citep{meurant2006lanczos,paige1976error,paige1980accuracy,greenbaum1997iterative}. The next lemma interprets the perturbation analysis for the dynamic Lanczos process in terms of $A^*_j$, which satisfies an approximate {three-term} recurrence with a perturbation term $\delta Q$.

\begin{lemma}\label{lem:dyn_lanc_decom}
        Suppose Assumptions \ref{assu:f} to \ref{assu:strongg} hold. The dynamic Lanczos process in~\cref{alg:LancBiO} with normalized $q_1$ and $\alpha_j, \beta_{j}, q_{j}$ satisfies
    \begin{equation}\label{eq:lanc_recur_compact}
        A_{j}^{*}Q_j=Q_jT_j+\beta _{j+1}q_{j+1}e_{j}^{\top}+\delta Q_j
    \end{equation}
    for $j=1,2,\dots,m$, where $Q_j=\left[ q_1,q_2,\dots ,q_j \right]$, $\delta Q_j=\left[ \delta q_1,\delta q_2,\dots ,\delta q_j \right]$,
    \begin{equation*}
        T_j =
        \begin{pmatrix}
        \alpha_1 & \beta_2 & & & & \\
        \beta_2 & \alpha_2 & \beta_3 & &  &\\
         & \beta_3&\ddots  & \ddots &\\
         & & \ddots & \ddots & \beta_j &\\
         & & & \beta_j& \alpha_j& \\
        \end{pmatrix}.
    \end{equation*}
    The columns of the perturbation $\delta Q_j$ satisfy
    \begin{equation*}
        \norm{\delta q_i} \le L_{gyy}\varepsilon_{j},\text{ for }i=1,2,\dots,j.
    \end{equation*}
    Additionally, if we decompose $Q_j$ as
    \begin{equation}\label{eq:orth_R}
        Q^\top _jQ_j = R_j^\top  + R_j,
    \end{equation}
    with $R_j$ as a strictly upper triangular matrix, then
    \begin{equation}\label{eq:TR_RT}
        T_jR_j-R_jT_j=\beta _{j+1}Q_{j}^{\top}q_{j+1}e_{j}^{\top}+\delta R_j,
    \end{equation}
    where $\delta R_j$ is strictly upper triangular with elements $\abs{\zeta_{st}}\le 2L_{gyy}\varepsilon_j$, for $1\le s<t\le j$.
\end{lemma}
\begin{proof}
    From
    \begin{equation*}
            \alpha _j=q_{j}^{\top}u_j=q_{j}^{\top}A_jq_j-\beta _jq_{j}^{\top}q_{j-1}
    \end{equation*}
    and 
    \begin{equation*}
        q_{j+1}^{\top}q_j=\frac{1}{\beta _{j+1}}\omega _{j}^{\top}q_j=\frac{1}{\beta _{j+1}}\left( u_j-\alpha _jq_j \right) ^\top q_j=\frac{1}{\beta _{j+1}}\left( A_jq_j-\alpha _jq_j-\beta _jq_{j-1} \right) ^\top q_j,
    \end{equation*}
    we can derive 
    \begin{equation}\label{eq:orth_q}
        q^\top _{j+1}q_j=0
    \end{equation}
    by induction. Then, we combine equations \eqref{eq:lanc_u}, \eqref{eq:lanc_alpha}, \eqref{eq:lanc_omega} and \eqref{eq:lanc_q}, and rewrite them in the perturbed form:
    \begin{equation}\label{eq:lanc_recur}
        \beta _{i+1}q_{i+1}=A_iq_i-\beta _iq_{i-1}-\alpha _iq_i=A_{j}^{*}q_i-\beta _iq_{i-1}-\alpha _iq_i+\delta q_i,\ \ for\ i=1,2,\dots,j,
    \end{equation}
    where $\left\| \delta q_i \right\| \le L_{gyy}\varepsilon _j$ due to Assumpstions \ref{assu:g} and \ref{assu:strongg}. Specifically,  \eqref{eq:lanc_recur} can be rewritten in a compact form:
    \begin{equation*}
        A_{j}^{*}Q_j=Q_jT_j+\beta _{j+1}q_{j+1}e_{j}^{\top}+\delta Q_j.
    \end{equation*}
    Then, we consider the orthogonality of matrix $Q_j$, which is reflected by $R_j$ in \eqref{eq:orth_R}.
    Multiply on both sides of \eqref{eq:lanc_recur_compact} by $Q^\top _j$,
    \begin{equation*}
        Q_{j}^{\top}A_{j}^{*}Q_j=Q_{j}^{\top}Q_jT_j+\beta _{j+1}Q_{j}^{\top}q_{j+1}e_{j}^{\top}+Q_{j}^{\top}\delta Q_j.
    \end{equation*}
    Combining its symmetry with the decomposition \eqref{eq:orth_R}, we obtain
    \begin{equation}\label{eq:M_MT}
        T_j\left( R_{j}^{\top}+R_j \right) -\left( R_{j}^{\top}+R_j \right) T_j=\beta _{j+1}\left( Q_{j}^{\top}q_{j+1}e_{j}^{\top}-e_jq_{j+1}^{\top}Q_j \right) +Q_{j}^{\top}\delta Q_j-\delta Q_{j}^{\top}Q_j.
    \end{equation}
    Denote $M_j=T_jR_j-R_jT_j$ which is upper triangular. Since the consecutive $q_i$ is orthogonal as revealed by \eqref{eq:orth_q}, we conclude that the diagonal elements of $M_j$ are $0$. Furthermore, by extracting the upper triangular part of the right hand side of \eqref{eq:M_MT}, we can get
    \begin{equation*}
        M_j=T_jR_j-R_jT_j=\beta _{j+1}Q_{j}^{\top}q_{j+1}e_{j}^{\top}+\delta R_j,
    \end{equation*}
    where $\delta R_j$ is strictly upper triangular with elements $\zeta_{st}$ satisfying: for $t=2,3,\ldots,j$,
    \begin{equation*}
        \left\{ \begin{array}{cl}
        	\zeta _{t-1,t}&=q_{t-1}^{\top}\delta q_t-\delta q_{t-1}^{\top}q_t\\
        	\zeta _{st}&=q_{s}^{\top}\delta q_t-\delta q_{s}^{\top}q_t,\ \ \ \ \quad\ \  s=1,2,\dots ,t-2.\\
        \end{array} \right. 
    \end{equation*}
    From the boundedness of $\norm{\delta q_j}$, it follows that for $1\le s<t\le j$, $\abs{\zeta_{st}}\le 2L_{gyy}\varepsilon_j$.
\end{proof}

\cref{lem:dyn_lanc_decom} illustrates the influence of the dynamics in~\cref{alg:LancBiO} imposed on the three-term Lanczos recurrence, and as \eqref{eq:orth_R} reveals, $R$~reflects the loss of orthogonality of the basis $Q$. However, the following lemmas demonstrate that the range of eigenvalues of the approximate projection matrix $T$ is indeed controllable. 

To proceed, we establish the Ritz pairs of $T_j$ as $\left( \mu _{i}^{\left( j \right)},y_{i}^{\left( j \right)} \right)$ for $i=1,2,\dots,j$, such that
\begin{equation*}
    T_j Y^{(j)}=Y^{(j)} \operatorname{diag}\left(\mu_1^{(j)},\mu_2^{(j)},\ldots,\mu_j^{(j)}\right).
\end{equation*}
where the normalized $\{y^{(j)}_i\}_{i=1}^j$ form the orthogonal matrix $Y^{(j)}$ with the elements $\varsigma^{(j)}_{st}$ for $1\le s,t\le j$, and we arrange the Ritz values in~a specific order,
\begin{equation*}
    \mu_1^{(j)}>\mu_2^{(j)}>\cdots>\mu_j^{(j)}.
\end{equation*}
We define the $j$-th approximate eigenvector matrix
\begin{equation*}
    Z^{\left( j \right)}:=\left[ z_{1}^{\left( j \right)},z_{2}^{\left( j \right)},\dots ,z_{j}^{\left( j \right)} \right] :=Q_jY^{\left( j \right)}.
\end{equation*}
and the corresponding Rayleigh quotients of $A_{j}^{*}$
\begin{equation*}
    \nu^{(j)}_i:=\frac{\left( z_{i}^{\left( j \right)} \right) ^\top A_{j}^{*}z_{i}^{\left( j \right)}}{\left( z_{i}^{\left( j \right)} \right) ^\top z_{i}^{\left( j \right)}},\text{ for}\ i=1,2,\dots,j.
\end{equation*}
The subsequent lemma describes the difference of eigenvalues between $T_j$ and some $T_n$ constructed in preceding steps.
\begin{lemma}\label{lem:mu_diff}
Suppose Assumptions \ref{assu:f} to \ref{assu:strongg} hold. For any eigenpair $\left( \mu _{i}^{\left( j \right)},y_{i}^{\left( j \right)} \right) $ of $T_j$, there exists an integer pair $(s,n)$ where $1\le s\le n<j$, such that
\begin{equation}\label{eq:diff_mu}
    \abs{\mu _{i}^{\left( j \right)}-\mu _{s}^{\left( n \right)}} \le \frac{2j^2L_{gyy}\varepsilon _j}{\sqrt{3}\left| \left( y_{i}^{\left( j \right)} \right) ^\top R_jy_{i}^{\left( j \right)} \right|}.
\end{equation}
\end{lemma}
\begin{proof}
Multiply the extended eigenvector $y_{r}^{\left( i \right)}$ of $T_i$ by $T_j$, where $i<j$,
\begin{equation}\label{eq:mu_T_diff}
    T_j\left[ \begin{array}{c}
    	y_{r}^{\left( i \right)}\\
    	\mathbf{0}\\
    \end{array} \right] =\left[ \begin{array}{c}
    	\mu _{r}^{\left( i \right)}y_{r}^{\left( i \right)}\\
    	\beta _{i+1}\varsigma _{ir}^{\left( i \right)}e_1\\
    \end{array} \right].
\end{equation}
Then multiply $\left( y_{s}^{\left( j \right)} \right) ^\top $ on the both sides of \eqref{eq:mu_T_diff},
\begin{equation}\label{eq:mus_mur}
    \left( \mu _{s}^{\left( j \right)}-\mu _{r}^{\left( i \right)} \right) \left( y_{s}^{\left( j \right)} \right) ^\top \left[ \begin{array}{c}
	y_{r}^{\left( i \right)}\\
	\mathbf{0}\\
\end{array} \right] =\beta _{i+1}\varsigma _{ir}^{\left( i \right)}\varsigma _{i+1,s}^{\left( j \right)},
\end{equation}
and multiply the eigenvectors $\left( y_{s}^{\left( j \right)} \right) ^\top$ and $y_{t}^{\left( j \right)}$ on the left and right of equation \eqref{eq:TR_RT}, respectively,
\begin{equation}\label{eq:mus_mut}
    \left( \mu _{s}^{\left( j \right)}-\mu _{t}^{\left( j \right)} \right) \left( y_{s}^{\left( j \right)} \right) ^\top R_jy_{t}^{\left( j \right)}=\varsigma _{jt}^{\left( j \right)}\beta _{j+1}\left( z_{s}^{\left( j \right)} \right) ^\top q_{j+1}+\epsilon _{st}^{\left( j \right)},
\end{equation}
where we define
\begin{equation}\label{eq:epsilon}
    \epsilon _{st}^{\left( j \right)}:=\left( y_{s}^{\left( j \right)} \right) ^\top \delta R_jy_{t}^{\left( j \right)}.
\end{equation}
Note that $\left| \epsilon _{st}^{\left( j \right)} \right|\le 2jL_{gyy}\varepsilon _j$ because of \cref{lem:dyn_lanc_decom}. Specifically, taking $s=t$ in \eqref{eq:mus_mut},
\begin{equation}\label{eq:zTq}
    \left( z_{s}^{\left( j \right)} \right) ^\top q_{j+1}=-\frac{\epsilon _{ss}^{\left( j \right)}}{\beta _{j+1}\varsigma _{js}^{\left( j \right)}},
\end{equation}
we can rewrite \eqref{eq:zTq} in the matrix form
\begin{equation}\label{eq:represent_R}
    Q_{r}^{\top}q_{r+1}=Y^{\left( r \right)}c_r,
\end{equation}
where for $r=1,\ldots,s$,
\begin{equation*}
    e_{s}^{\top}c_r:=-\frac{\epsilon _{ss}^{\left( r \right)}}{\beta _{r+1}\varsigma _{rs}^{\left( j \right)}}.
\end{equation*}
By observing that $Q_{r}^{\top}q_{r+1}=Y^{\left( r \right)}c_r$ is the $(r+1)$-th column of $R_j$, we can derive
\begin{align}
    \left( y_{i}^{\left( j \right)} \right) ^\top R_jy_{i}^{\left( j \right)}&=-\sum_{r=1}^{j-1}{\varsigma _{r+1,i}^{\left( j \right)}}\sum_{t=1}^r{\frac{\epsilon _{tt}^{\left( r \right)}}{\beta _{r+1}\varsigma _{rt}^{\left( r \right)}}\left( y_{i}^{\left( j \right)} \right) ^\top \left[ \begin{array}{c}
	y_{t}^{\left( r \right)}\\
	\mathbf{0}\\
    \end{array} \right]}    \label{eq:yRy_1}
    \\
    &=-\sum_{r=1}^{j-1}{\left( \varsigma _{r+1,i}^{\left( j \right)} \right) ^2}\sum_{t=1}^r{\frac{\epsilon _{tt}^{\left( r \right)}}{\mu _{i}^{\left( j \right)}-\mu _{t}^{\left( r \right)}}},     \label{eq:yRy_2}
\end{align}
where \eqref{eq:yRy_1} and \eqref{eq:yRy_2} follow from \eqref{eq:represent_R} and \eqref{eq:mus_mur} respectively. Consequently, the definition \eqref{eq:epsilon} reveals
\begin{equation}\label{eq:sum_eps}
    \sum_{s,t=1}^j{\left( \epsilon _{st}^{\left( j \right)} \right) ^2}=\left\| \delta R_j \right\| _{F}^{2}\le 2j^2L_{gyy}\varepsilon _j.
\end{equation}
Based on \eqref{eq:yRy_2} and the orthogonality of $Y^{(j)}$,
\begin{align*}
    \left| \left( y_{i}^{\left( j \right)} \right) ^\top R_jy_{i}^{\left( j \right)} \right|&\le \left| \sum_{r=1}^{j-1}{\varsigma _{r+1,i}^{\left( j \right)}}\sum_{t=1}^r{\frac{\epsilon _{tt}^{\left( r \right)}}{\min\limits_{1\le d\le l <j}\abs{\mu _{i}^{\left( j \right)}-\mu _{d}^{\left( l \right)}}}} \right|
    \\
    &\le \frac{1}{\min\limits_{1\le d\le l <j}\abs{\mu _{i}^{\left( j \right)}-\mu _{d}^{\left( l \right)}}}\left| \sum_{r=1}^{j-1}\sum_{t=1}^r{\varsigma _{r+1,i}^{\left( j \right)}\epsilon _{tt}^{\left( r\right)}} \right|,
\end{align*}
it follows from the Cauchy--Schwarz inequality and \eqref{eq:sum_eps} that
\begin{equation*}
    \min_{1\le d\le l <j}\abs{\mu _{i}^{\left( j \right)}-\mu _{d}^{\left( l \right)}} \le \frac{\sqrt{j}\sqrt{\sum_{r=1}^{j-1}{\sum_{t=1}^r{\left( \epsilon _{tt}^{\left( r \right)} \right) ^2}}}}{\left| \left( y_{i}^{\left( j \right)} \right) ^\top R_jy_{i}^{\left( j \right)} \right|}\,\le \frac{2j^2L_{gyy}\varepsilon _j}{\sqrt{3}\left| \left( y_{i}^{\left( j \right)} \right) ^\top R_jy_{i}^{\left( j \right)} \right|}.
\end{equation*}
\end{proof}

\begin{lemma}\label{lem:T_bound}
    Suppose Assumptions \ref{assu:f} to \ref{assu:strongg} hold. $Q_j$ and $T_j$ are the basis matrix and the approximate tridiagonal matrix in the $j$-th step and $R_j$ is the strictly upper triangular matrix defined in \eqref{eq:orth_R} characterizing the orthogonality of $Q_j$.  
    Given $\mu _{i}^{\left( j \right)}$ the $i$-th eigenvalue of $T_j$, then for $i=1,\ldots,j$,
    \begin{equation}\label{eq:T_bound}
        \mu _g-\frac{2\sqrt{3}}{3}\left( j+1 \right) ^3L_{gyy}\varepsilon _j\le \mu^{(j)}_i\le L_{gy}+\frac{2\sqrt{3}}{3}\left( j+1 \right) ^3L_{gyy}\varepsilon _j.
    \end{equation}
\end{lemma}
\begin{proof}
By conducting the left multiplication with $\left( y_{i}^{\left( j \right)}Q_{j}\right)^{\top}$ and the right multiplication with $y_{i}^{\left( j \right)}$ on both sides of equation~\eqref{eq:lanc_recur_compact}, we obtain the following equation,
\begin{equation*}
    \left( z_{i}^{\left( j \right)} \right) ^\top A_{j}^{*}z_{i}^{\left( j \right)}-\mu _{i}^{\left( j \right)}\left( z_{i}^{\left( j \right)} \right) ^\top z_{i}^{\left( j \right)}=-\epsilon _{ii}^{\left( j \right)}+\left( z_{i}^{\left( j \right)} \right) ^\top \delta Q_jy_{i}^{\left( j \right)}.
\end{equation*}
Dividing it by $\left( z_{i}^{\left( j \right)} \right) ^\top z_{i}^{\left( j \right)}$, we have
\begin{align}
    \left| \frac{\left( z_{i}^{\left( j \right)} \right) ^\top A_{j}^{*}z_{i}^{\left( j \right)}-\mu _{i}^{\left( j \right)}\left( z_{i}^{\left( j \right)} \right) ^\top z_{i}^{\left( j \right)}}{\left( z_{i}^{\left( j \right)} \right) ^\top z_{i}^{\left( j \right)}} \right|&\le \frac{\left| \epsilon _{ii}^{\left( j \right)} \right|+\left| \left( z_{i}^{\left( j \right)} \right) ^\top \delta Q_jy_{i}^{\left( j \right)} \right|}{\left| 1+2\left( y_{i}^{\left( j \right)} \right) ^\top R_jy_{i}^{\left( j \right)} \right|}    \nonumber
    \\
    &\le \frac{\left| \epsilon _{ii}^{\left( j \right)} \right|+L_{gyy}\varepsilon _j\sqrt{j\left( \left| 1+2\left( y_{i}^{\left( j \right)} \right) ^\top R_jy_{i}^{\left( j \right)} \right| \right)}}{\left| 1+2\left( y_{i}^{\left( j \right)} \right) ^\top R_jy_{i}^{\left( j \right)} \right|},    \label{eq:mu_quo_diff}
\end{align}
where the inequalities come from
\begin{equation}\label{eq:z_yRy}
    \left\| z_{i}^{\left( j \right)} \right\| ^2=1+2\left( y_{i}^{\left( j \right)} \right) ^\top R_jy_{i}^{\left( j \right)}.
\end{equation}

{\bfseries Case I:} If
\begin{equation}\label{eq:yRy_condition}
    \abs{\left( y_{i}^{\left( j \right)} \right) ^\top R_jy_{i}^{\left( j \right)}}<\frac{3}{8}-\frac{\varepsilon _j}{2}
\end{equation}
holds, then from \eqref{eq:z_yRy}, 
\begin{equation*}
    \left\| z_{i}^{\left( j \right)} \right\| \ge \frac{1}{2}.
\end{equation*}
Furthermore, \eqref{eq:mu_quo_diff} reveals that there exists~a Rayleigh quotient $\nu^{(j)}$ of $A^*_j$ that satisfies
\begin{equation*}
    \left| \nu ^{\left( j \right)}-\mu _{i}^{\left( j \right)} \right|\le 2\left| \epsilon _{ii}^{\left( j \right)} \right|+\sqrt{2j}L_{gyy}\varepsilon _j\le \left( 4j+\sqrt{2j} \right) L_{gyy}\varepsilon _j.
\end{equation*}

{\bfseries Case II:} If the condition \eqref{eq:yRy_condition} does not hold, by applying \cref{lem:mu_diff}, we can find an integer pair $(s_1,n_1)$ with $1\le s_1\le n_1<j$ such that 
\begin{equation*}
\left| \mu _{i}^{\left( j \right)}-\mu _{s_1}^{\left( n_1 \right)} \right|\le \frac{2j^2L_{gyy}\varepsilon _j}{\sqrt{3}\left| \left( y_{i}^{\left( j \right)} \right) ^\top R_jy_{i}^{\left( j \right)} \right|}\le 2\sqrt{3}j^2L_{gyy}\varepsilon _j.
\end{equation*}
By observing that $\left( \mu _{s_1}^{\left( n_1 \right)},y_{s_1}^{\left( n_1 \right)} \right) $ and $\left( y_{s_1}^{\left( n_1 \right)} \right) ^\top R_{n_1}y_{s_1}^{\left( n_1 \right)}$ can also be categorized into {\bfseries Case I} or {\bfseries Case II}, we can repeat this process and construct a sequence $\{(s_t,n_t)\}_{t=0}^{l+1}$ with $1\le n_{l+1}<n_l<\cdots <n_1<n_0=j$ until
\begin{equation*}
    \abs{\left( y_{s_{l+1}}^{\left( n_{l+1} \right)} \right) ^\top R_jy_{s_{l+1}}^{\left( n_{l+1} \right)}}<\frac{3}{8}-\frac{\varepsilon _{n_{l+1}}}{2}.
\end{equation*}
Inequality~\eqref{eq:yRy_condition} holds when the superscript $j=1$ since~$T_1=[\alpha_1],\ y_1^{(1)}=1$~and $z^{(1)}_1=q_1$. Therefore, we can obtain $\{(s_t,n_t)\}_{t=0}^{l+1}$ in finite steps, resulting in the following estimate.
\begin{equation*}
    \begin{aligned}
        \left| \bar{\nu}^{\left( j \right)}-\mu _{i}^{\left( j \right)} \right|&\le \left| \bar{\nu}^{\left( j \right)}-\mu _{s_{l+1}}^{\left( n_{l+1} \right)} \right|+\sum_{t=0}^l{\left| \mu _{s_{t+1}}^{\left( n_{t+1} \right)}-\mu _{s_t}^{\left( n_l \right)} \right|}
        \\
        &\le \left( 4j+\sqrt{2j} \right) L_{gyy}\varepsilon _j+\sum_{t=0}^l{2\sqrt{3}n_{t}^{2}L_{gyy}\varepsilon _j}
        \\
        &\le \frac{2\sqrt{3}}{3}\left( j+1 \right) ^3L_{gyy}\varepsilon _j,
    \end{aligned}
\end{equation*}
for some $\bar{\nu}^{\left( j \right)}$.

Any Rayleigh quotient of $A_j^*$ is bounded by its eigenvalues \citep{parlett1998symmetric}, i.e., for any $\nu^{(j)}$, 
\begin{equation*}
    \lambda_{min}^{(j)}\le \nu^{(j)}\le \lambda_{max}^{(j)},
\end{equation*}
where $\lambda_{min}^{(j)}$ and $\lambda_{max}^{(j)}$ are the minimum and maximal eigenvalue of $A_j^*$, respectively. Based on \cref{assu:g} and \cref{assu:strongg}, we complete the proof.
\end{proof}

\section{Proof of \cref{lem:descent_res}}\label{sec:proof312}
We detail the proof for \cref{lem:descent_res} in this section, starting with a proof sketch.
\subsection{Proof sketch}
The proof of \cref{lem:descent_res} is structured by four steps.
\paragraph{Step1: extending $\varepsilon_j$ to $\tilde{\varepsilon}_j$ in the lemmas given in  \cref{sec:proof_lanc}.}
 \hspace*{\fill} \\
In  \cref{sec:proof_lanc}, we adopt $A^*_j$ and $b^*_j$ as reference values with each epoch for the analysis, i.e., for $j=1,2,\ldots,m$,
\begin{align*}
    A_j^*:=\nabla_{y y}^2 g\left(x_j, y_j^*\right),\ b_j^*:=\nabla_y f\left(x_j, y_j^*\right),\  \bar{b}:=b_1-A_1 \bar{v},\ \bar{r}_0:=\bar{b},\ \bar{r}_j:=\bar{b}-A_j^* \Delta v_j.
\end{align*}
and
\begin{align*}
    \varepsilon_{j}:=\max_{1\le s,t\le j}\left(1+\frac{L_{g x}}{\mu_g}\right)\left\|x_{m h+s}-x_{m h+t}\right\|+\left\|y_{m h+s}-y_{m h+s}^*\right\|.
\end{align*}
In parallel, we can also view the $A_1,b_1$ as reference values. In this way, denote the similar quantities
\begin{align}\label{eq:residual_prime}
    A_j:=\nabla_{y y}^2 g\left(x_j, y_j\right),\ b_j:=\nabla_y f\left(x_j, y_j\right),\  \bar{b}^\prime:=b_1-A_1 \bar{v},\ \bar{r}^\prime_0:=\bar{b}^\prime,\ \bar{r}^\prime_j:=\bar{b}^\prime-A_1 \Delta v_j.
\end{align}
and
\begin{align}\label{eq:tildevarepsilon}
    \tilde{\varepsilon}_{j}:=\max_{1\le s,t\le j}\left\|x_{m h+s}-x_{m h+t}\right\|+\left\|y_{m h+s}-y_{m h+t}\right\|.
\end{align}
Consequently, we extend the lemmas in \cref{sec:proof_lanc}. The results listed in \cref{sec:extendF} are the recipe of Step2.
\paragraph{Step2: upper-bounding the residual $\norm{\bar{r}_j^\prime}$.}
 \hspace*{\fill} \\
In \cref{sec:proofstep2}, we then demonstrate that if the value of $\tilde{\varepsilon}_j$ is not too large, $\norm{\bar{r}_j^\prime}$ can be bounded, which is an important lemma for Step3.
\paragraph{Step3: controlling $\tilde{\varepsilon}_j$ and ${\varepsilon}_j$ by induction.}
\hspace*{\fill} \\
Since the expression of $\tilde{\varepsilon}_j$ does not involve $y^*$, its magnitude can be controlled by adjusting the step size, implied by \eqref{eq:tildevarepsilon}. With the help of the stability of $\tilde{\varepsilon}_j$, we can prove
\begin{equation}\label{eq:bound_eps}
\mu _g-\frac{2\sqrt{3}}{3}\left( j+1 \right) ^3L_{gyy}{\varepsilon} _j>0.
\end{equation}
\paragraph{Step4: proof of \cref{lem:descent_res}}
\hspace*{\fill} \\
Based on the conclusion \eqref{eq:bound_eps} revealing the benign property of the dynamic process, we achieve the proof of \cref{lem:descent_res}.

\subsection{Extended lemmas from \cref{sec:proof_lanc}}\label{sec:extendF}
In this part, we view the $A_1,b_1$ as reference values. In this way, we can institute $A_j^*$ with $A_1$ and $\varepsilon_j$ with $\tilde{\varepsilon}_j$ in the lemmas from  \cref{sec:proof_lanc}, resulting in the extended version of lemmas.
\begin{lemma}\label{lem:prime_dyn_lanc_decom}
        (Extended version of \cref{lem:dyn_lanc_decom}) Suppose Assumptions \ref{assu:f} to \ref{assu:strongg} hold. The dynamic Lanczos process in \cref{alg:LancBiO} with normalized $q_1$ satisfies
    \begin{equation}\label{eq:lanc_recur_compact_ext}
        A_{1}Q_j=Q_jT_j+\beta _{j+1}q_{j+1}e_{j}^{\top}+\delta Q^\prime_j
    \end{equation}
    for $j=1,2,\dots,m$, where $Q_j=\left[ q_1,q_2,\dots ,q_j \right]$, $\delta Q^\prime_j=\left[ \delta q^\prime_1,\delta q^\prime_2,\dots ,\delta q^\prime_j \right]$,
    \begin{equation*}
        T_j =
        \begin{pmatrix}
        \alpha_1 & \beta_2 & & & & \\
        \beta_2 & \alpha_2 & \beta_3 & &  &\\
         & \beta_3&\ddots  & \ddots &\\
         & & \ddots & \ddots & \beta_j &\\
         & & & \beta_j& \alpha_j& \\
        \end{pmatrix}.
    \end{equation*}
    The columns of the perturbation $\delta Q^\prime_j$ satisfy
    \begin{equation}\label{eq:bound_deltaq_ext}
        \norm{\delta q^\prime_i} \le L_{gyy}\tilde{\varepsilon}_{j},\text{ for }i=1,2,\dots,j.
    \end{equation}
\end{lemma}

\begin{lemma}\label{lem:prime_T_bound}
    (Extended version of \cref{lem:T_bound}) Suppose Assumptions \ref{assu:f} to \ref{assu:strongg} hold. $Q_j$ and $T_j$ are the basis matrix and the approximate tridiagonal matrix in the $j$-th step. Take $\mu _{i}^{\left( j \right)}$ as the $i$-th eigenvalue of $T_j$, then for $i=1,\ldots,j$,
    \begin{equation}\label{eq:T_bound_ext}
        \mu _g-\frac{2\sqrt{3}}{3}\left( j+1 \right) ^3L_{gyy}\tilde{\varepsilon} _j\le \mu^{(j)}_i\le L_{gy}+\frac{2\sqrt{3}}{3}\left( j+1 \right) ^3L_{gyy}\tilde{\varepsilon} _j.
    \end{equation}
\end{lemma}

\subsection{Proof of Step2}\label{sec:proofstep2}
The following lemma demonstrates that if the value of $\tilde{\varepsilon}_j$ is not too large, $\norm{\bar{r}_j^\prime}$ can be bounded.
\begin{lemma}\label{lem:vdescent}
Suppose Assumpsions \ref{assu:f} to \ref{assu:strongg} and
\begin{equation*}
\mu _g-\frac{2\sqrt{3}}{3}\left( j+1 \right) ^3L_{gyy}\tilde{\varepsilon} _j>0
\end{equation*}
are satisfied  within an epoch. Then, it holds that
\begin{equation*}
\frac{\left\| \bar{r}^\prime_j \right\|}{\left\| \bar{r}^\prime_0 \right\|}\le 2\sqrt{\tilde{\kappa}^\prime\left( j \right)}\left( \frac{\sqrt{\tilde{\kappa}^\prime\left( j \right)}-1}{\sqrt{\tilde{\kappa}^\prime\left( j \right)}+1} \right) ^j+\sqrt{j}L_{gyy}\varepsilon _j\tilde{\kappa}^\prime\left( j \right),
\end{equation*}
where 
$
    \tilde{\kappa}^\prime\left( j \right) :=\frac{L_{gy}+\frac{2\sqrt{3}}{3}\left( j+1 \right) ^3L_{gyy}\tilde{\varepsilon} _j}{\mu _g-\frac{2\sqrt{3}}{3}\left( j+1 \right) ^3L_{gyy}\tilde{\varepsilon}_j}.
$
\end{lemma}

\begin{proof}
Denote the solution in the dynamic subspace in the $j$-th step by
\begin{equation}\label{eq:subspace_sol}
    \Delta \xi _j=\left( T_j \right) ^{-1}\bar{b} =\norm{\bar{r}^\prime_0}\left( T_j \right) ^{-1}e_1,
\end{equation}
where $T_j$ is nonsingular because of \cref{lem:prime_T_bound}. By~\eqref{eq:residual_prime}, \eqref{eq:lanc_recur_compact_ext}, and~\eqref{eq:subspace_sol}, we have
\begin{equation*}
    \bar{r}^\prime_j=\bar{b}-A_{1}Q_j\Delta \xi _j=-\beta _{j+1}q_{j+1}e_{j}^{\top}\Delta \xi _j-\delta Q^\prime_j\Delta \xi _j.
\end{equation*}
It follows that
\begin{equation}\label{rj_r0_1}
    \frac{\left\| \bar{r}^\prime_j \right\|}{\left\| \bar{r}^\prime_0 \right\|}\le \left\| \delta Q^\prime_j \right\| \left\| \left( T_j \right) ^{-1} \right\| + \left| \beta _{j+1}e_{j}^{\top}\left( T_j \right) ^{-1}e_1 \right|.
\end{equation}
The first term on the right side of \eqref{rj_r0_1} can be bounded by \eqref{eq:bound_deltaq_ext} and \eqref{eq:T_bound_ext}:
\begin{equation}
    \left\| \delta Q^\prime_j \right\| \left\| \left( T_j \right) ^{-1} \right\| \le \sqrt{j}L_{gyy}\tilde{\varepsilon} _j\frac{L_{gy}}{\mu _g-\frac{2\sqrt{3}}{3}\left( j+1 \right) ^3L_{gyy}\tilde{\varepsilon} _j}\le \sqrt{j}L_{gyy}\tilde{\varepsilon} _j\tilde{\kappa}^\prime\left( j \right).
\end{equation}
Recall that
\begin{equation}\label{eq:T_eigen}
    T_{j+1,j}:=\left[ \begin{array}{c}
	T_j\\
	\beta _{j+1}e_{j}^{\top}\\
    \end{array} \right],
\end{equation}
and that for any symmetric tridiagonal matrix $T$, where the upper left $(j+1)\times j$ block is $T_{j+1,j}$, the application of the classic Lanczos algorithm to $T$, starting with the initial vector $e_1$ will result in the matrix $T_{j+1,j}$ at the $j$-th step. To construct a suitable $(j+1)\times(j+1)$ symmetric tridiagonal matrix T, we consider a \emph{virtual step} with ${\lambda^\diamond}={\theta^\diamond}=0$, which leads to $(x^\diamond_{j+1},y^\diamond_{j+1})=\kh{x_j,y_j}$, $\tilde{\varepsilon}^\diamond_{j+1}=\tilde{\varepsilon}_j$, and 
\begin{equation*}
    T={\setlength{\arraycolsep}{1.pt} \left( \begin{array}{ccc:c}
            	&		&		&		\\
            	&		\vphantom{{\small \frac{2^{2^{2^2}}}{\frac{2}{3}}}}{\Large T_{j+1,j}}\phantom{s}&		&		\\
            	&		&		&		\beta _{j+1}\vphantom{{\small \frac{2}{\frac{2}{3}}}}\\
            	&		&		&		\tilde{\alpha}^\diamond _{j+1}\\
            \end{array} \right) }
\end{equation*}
By \cref{lem:prime_T_bound}, given any  eigenvalue $\mu$ of $T$
\begin{equation*}
    \mu _g-\frac{2\sqrt{3}}{3}\left( j+1 \right) ^3L_{gyy}\tilde{\varepsilon} _j\le \mu \le L_{gy}+\frac{2\sqrt{3}}{3}\left( j+1 \right) ^3L_{gyy}\tilde{\varepsilon} _j.
\end{equation*}
In this way, $\left| \beta _{j+1}e_{j}^{\top}\left( T_j \right) ^{-1}e_1 \right|$ can be seen as the residual in the $j$-th step of the classic Lanczos process with the positive-definite matrix $T$ and the iniitial vector $e_1$. Since the eigenvalues of $T$ satisfy~\eqref{eq:T_eigen}, it~follows from the standard convergence property of the Lanczos process~\citep{greenbaum1997iterative} that
\begin{equation*}
    \left| \beta _{j+1}e_{j}^{\top}\left( T_j \right) ^{-1}e_1 \right|\le 2\sqrt{\tilde{\kappa}^\prime\left( j \right)}\left( \frac{\sqrt{\tilde{\kappa}^\prime\left( j \right)}-1}{\sqrt{\tilde{\kappa}^\prime\left( j \right)}+1} \right) ^j,
\end{equation*}
which completes the proof.
\end{proof}

\subsection{Proof of $\mu _g-\frac{2\sqrt{3}}{3}\left( j+1 \right) ^3L_{gyy}{\varepsilon} _j>0$}
In this part, we will give the detailed proof of $\mu _g-\frac{2\sqrt{3}}{3}\left( j+1 \right) ^3L_{gyy}{\varepsilon} _j>0$. At the same time, we demonstrate that, with appropriate step sizes, the auxiliary variable~$v_k$ is bounded, thereby showing that the hyper-gradient estimator~\eqref{eq:hypergradient_esti} remains bounded. This ensures the stable behavior of the dynamic Lanczos process,
\begin{lemma}\label{lem:bound_v}
    Suppose Assumptions \ref{assu:f}, \ref{assu:g}, \ref{assu:strongg}, \ref{assu:boundfx} and \ref{assu:y_initial} hold. If within each epoch, we set the step size $\theta\sim\mathcal{O}({1}/{m})$ a constant for $y$, and the step size for $x$ as zero in the first $m_0\sim\mathcal{O}(1)$ steps and the others as an appropriate constant $\lambda\sim\mathcal{O}({1}/{m^4})$, then for any epoch, 
    \begin{equation*}
        \mu _g-\frac{2\sqrt{3}}{3}\left( j+1 \right) ^3L_{gyy}\varepsilon _j>0,\ \text{ for } j=1,2,\ldots,m+1,
    \end{equation*}
    and there exists a constant $C_v>0$ so that $\norm{v_k}\le C_v$ for ${v_k}$ generated by \cref{alg:LancBiO}.
\end{lemma}
\begin{proof}
    Consider the iterates within one epoch and the constants
    \begin{equation*}
            0<\tilde{\varepsilon}< \mu _g \text{ and } \tilde{\kappa}:=\frac{L_{gy}+\tilde{\varepsilon}}{\mu _g-\tilde{\varepsilon}}.
    \end{equation*}
   It follows from \cref{lem:vdescent} that if $\tilde{\varepsilon}_j\le\frac{\sqrt{3}\tilde{\varepsilon}}{2(m+1)^3L_{gyy}}$, then
    \begin{equation}\label{eq:bound_r}
        \frac{\left\| b_1-A_1v_j \right\|}{\left\| b_1-A_1\bar{v} \right\|}\le 2\sqrt{\tilde{\kappa}}\left( \frac{\sqrt{\tilde{\kappa}}-1}{\sqrt{\tilde{\kappa}}+1} \right) ^j+\frac{\sqrt{3}\tilde{\varepsilon}}{2(m+1)^2}\tilde{\kappa}\le 3\sqrt{\tilde{\kappa}}.
    \end{equation}
   
    Then, we give a proof by induction. At the beginning of the algorithm, \cref{assu:y_initial} reveals
    \begin{equation*}
        \varepsilon_1 = \left\| y_1-y_{1}^{*} \right\| \le \frac{\sqrt{3}\mu _g}{8\left( m+1 \right) ^3L_{gyy}}.
    \end{equation*}
    Combining it with $\left\| b_1-A_1\bar{v}_0 \right\| \le C_r$ constructs the start of induction within an epoch and induction between epochs. Within an epoch, suppose the following statements hold for $i=1,2,\ldots,j$,
    \begin{align*}
            \left\| b_1-A_1v_i \right\| &\le 3\sqrt{\tilde{\kappa}}C_r, 
            \\
            \left\| v_i \right\| &\le \frac{1}{\mu _g}\left( 3\sqrt{\tilde{\kappa}}C_r+C_{fx} \right), 
            \\
            \left\| \widetilde{\nabla }\varphi _i \right\| &\le C_{fx}+\frac{1}{\mu _{g}^{2}}\left( 3\sqrt{\tilde{\kappa}}C_r+C_{fx} \right), 
            \\
            \tilde{\varepsilon}_i&\le\frac{\sqrt{3}\tilde{\varepsilon}}{2(m+1)^3L_{gyy}}, \nonumber
            \\ 
            \varepsilon_i &\le \frac{\sqrt{3}\mu _g}{4\left( m+1 \right) ^3L_{gyy}}. \nonumber
    \end{align*}
    Then by setting the stepsizes
    \begin{equation}
        \begin{aligned}
            &\theta \le \frac{\tilde{\varepsilon}}{\mu_g L_{gy}m},
            \\
            &\lambda \le \frac{\sqrt{3}\tilde{\varepsilon}}{4(m+1)^4L_{gyy}}\kh{1+\theta\frac{L_{gy}L_{gx}}{\mu_g}}^{-1}\left( C_{fx}+\frac{1}{\mu _{g}^{2}}\left( 3\sqrt{\tilde{\kappa}}C_r+C_{fx} \right) \right) ^{-1},
        \end{aligned}
    \end{equation}
    and noticing that
    \begin{align*}
        \norm{\nabla_y g(x_{i+1},y_i)} &= \norm{\nabla_y g(x_{i+1},y_i) - \nabla_y g(x_{i+1},y^*_{i+1})} 
        \\
        &\le L_{gy}\norm{y_i-y^*_{i+1}}
        \\
        &\le L_{gy}\kh{\norm{y_i-y^*_i}+\norm{y_i^*-y^*_{i+1}}}
        \\
        &\le L_{gy}\norm{y_i-y^*_i} + \frac{L_{gy}L_{gx}}{\mu_g}\norm{x_i-x_{i+1}},
    \end{align*}
    we can get
    \begin{align}
        \tilde{\varepsilon}_{j+1}\le&\ \lambda \sum_{i=1}^{i=j}{\left\| \widetilde{\nabla }\varphi _i \right\|}+\theta \sum_{i=1}^{i=j}{\left\| \nabla _yg\left( x_{i+1},y_i \right) \right\|}   \nonumber
        \\
        \le& \kh{1+\theta\frac{L_{gy}L_{gx}}{\mu_g}}\lambda \sum_{i=1}^{i=j}{\left\| \widetilde{\nabla }\varphi _i \right\|} + \theta L_{gy}\sum_{i=1}^{i=j}\norm{y_i-y^*_i}    \nonumber
        \\
        \le& \kh{1+\theta\frac{L_{gy}L_{gx}}{\mu_g}}\lambda \sum_{i=1}^{i=j}{\left\| \widetilde{\nabla }\varphi _i \right\|} + \theta L_{gy}\sum_{i=1}^{i=j}\varepsilon_i   \nonumber
        \\
        \le& \kh{1+\theta\frac{L_{gy}L_{gx}}{\mu_g}}\lambda \sum_{i=1}^{i=j}{\left\| \widetilde{\nabla }\varphi _i \right\|} + \theta L_{gy}\frac{\sqrt{3}\mu _gm}{4\left( m+1 \right) ^3L_{gyy}}   \nonumber
        \\
        \le&\ \frac{\sqrt{3}\tilde{\varepsilon}}{2(m+1)^3L_{gyy}}.      \label{eq:induction_with_epoch1}
    \end{align}
    It follows from \eqref{eq:bound_r} that
    \begin{align}
        \left\| b_1-A_1v_{j+1} \right\| &\le 3\sqrt{\tilde{\kappa}}C_r, \label{eq:induction_with_epoch2}
        \\
        \left\| v_{j+1} \right\| &=\left\| A_{1}^{-1}\left( b_1-A_1v_{j+1}-b_1 \right) \right\| \le \frac{1}{\mu _g}\left( 3\sqrt{\tilde{\kappa}}C_r+C_{fx} \right) , \label{eq:induction_with_epoch3}
        \\
        \left\| \widetilde{\nabla }\varphi _{j+1} \right\| &\le C_{fx}+\frac{1}{\mu _{g}^{2}}\left( 3\sqrt{\tilde{\kappa}}C_r+C_{fx} \right). \label{eq:induction_with_epoch4}
    \end{align}

    Additionally, the descent property of $\norm{y_s-y^*_s}$ and the Lipschitz continuity of $y^*$ reveal that
    \begin{align}
        \left\| y_s-y_{s}^{*} \right\| \le & \left( 1-\theta \mu _g \right) ^{\frac{1}{2}}\norm{y_{s-1}-y^*_s}
        \\
        \le& \left( 1-\theta \mu _g \right) ^{\frac{1}{2}}\left\| y_{s-1}-y_{s-1}^{*} \right\| +\left( 1-\theta \mu _g \right) ^{\frac{1}{2}}\left( \frac{L_{gx}}{\mu _g} \right) \left\| x_s-x_{s-1} \right\| \nonumber
        \\
        \le& \left( 1-\theta \mu _g \right) ^{\frac{s-1}{2}}\left\| y_1-y_{1}^{*} \right\| +\left( 1-\theta \mu _g \right) ^{\frac{1}{2}}\left( \frac{L_{gx}}{\mu _g} \right) \lambda \sum_{t=1}^{t=s-1}{\left\| \widetilde{\nabla }\varphi _t \right\|} \nonumber
        \\
        \le & \left\| y_1-y_{1}^{*} \right\| +\left( 1-\theta \mu _g \right) ^{\frac{1}{2}}\left( \frac{L_{gx}}{\mu _g} \right) \lambda \sum_{t=1}^{t=s-1}{\left\| \widetilde{\nabla }\varphi _t \right\|}. \label{eq:y_decent_step}
    \end{align}
    Setting
    \begin{align*}
        \lambda \le \frac{\sqrt{3}\mu _g}{8(m+1)^4L_{gyy}}\left( C_{fx}+\frac{1}{\mu _{g}^{2}}\left( 3\sqrt{\tilde{\kappa}}C_r+C_{fx} \right) \right) ^{-1}\left( 1+\frac{L_{gx}}{\mu _g}+\left( 1-\theta \mu _g \right) ^{\frac{1}{2}}\left( \frac{L_{gx}}{\mu _g} \right) \right) ^{-1}
    \end{align*}
    yields from \cref{assu:y_initial}, \eqref{eq:induction_with_epoch3}, \eqref{eq:y_decent_step} that
    \begin{align}
        \varepsilon _{j+1}\le& \left( 1+\frac{L_{gx}}{\mu _g} \right) \max _{1\le s,t\le j+1}\left\| x_s-x_t \right\| +\max _{1\le s\le j+1}\left\| y_s-y_{s}^{*} \right\|  \nonumber
        \\
        \le& \left\| y_1-y_{1}^{*} \right\| +\left( 1+\frac{L_{gx}}{\mu _g}+\left( 1-\theta \mu _g \right) ^{\frac{1}{2}}\left( \frac{L_{gx}}{\mu _g} \right) \right) \lambda \sum_{i=1}^{i=j}{\left\| \widetilde{\nabla }\varphi _i \right\|} \nonumber
        \\
        \le &\left\| y_1-y_{1}^{*} \right\| +\left( 1+\frac{L_{gx}}{\mu _g}+\left( 1-\theta \mu _g \right) ^{\frac{1}{2}}\left( \frac{L_{gx}}{\mu _g} \right) \right) \lambda \sum_{i=1}^{i=j}{\left\| \widetilde{\nabla }\varphi _i \right\|} \nonumber
        \\
        \le&\ \frac{\sqrt{3}\mu _g}{4\left( m+1 \right) ^3L_{gyy}}. \label{eq:induction_with_epoch5}
    \end{align}
    As for the next epoch, denoting $C_v = \frac{1}{\mu _g}\left( 3\sqrt{\tilde{\kappa}}C_r+C_{fx}\right )$, we have
    \begin{align}
            \left\| b_{m+1}-A_{m+1}v_m \right\| &\le \left( L_{fx}+L_{gyy}C_{{v}} \right) \tilde{\varepsilon}_{m+1}+\left\| b_1-A_1v_m \right\|  \nonumber
            \\
            &\le \left( L_{fx}+L_{gyy}C_{{v}} \right) \tilde{\varepsilon}_{m+1}+\left( 2\sqrt{\tilde{\kappa}}\left( \frac{\sqrt{\tilde{\kappa}}-1}{\sqrt{\tilde{\kappa}}+1} \right) ^m+\sqrt{m}L_{gyy}\tilde{\varepsilon}_{m+1}\tilde{\kappa} \right) C_r \nonumber
            \\
            &\le C_r\left( \left( \frac{L_{fx}+L_{gyy}C_{{v}}}{C_r} \right) \tilde{\varepsilon}_{m+1}+2\sqrt{\tilde{\kappa}}\left( \frac{\sqrt{\tilde{\kappa}}-1}{\sqrt{\tilde{\kappa}}+1} \right) ^m+\sqrt{m}L_{gyy}\tilde{\varepsilon}_{m+1}\tilde{\kappa} \right)  \nonumber
            \\
            &\le C_r,       \label{eq:for_next_induc1}
        \end{align}
    by choosing $m,\tilde{\varepsilon}$ such that
    \begin{equation*}
        \left( \frac{L_{fx}+L_{gyy}C_{{v}}}{C_r} \right) \tilde{\varepsilon}_{m+1}+2\sqrt{\tilde{\kappa}}\left( \frac{\sqrt{\tilde{\kappa}}-1}{\sqrt{\tilde{\kappa}}+1} \right) ^m+\sqrt{m}L_{gyy}\tilde{\varepsilon}_{m+1}\tilde{\kappa}\le 1.
    \end{equation*}
    Moreover, since the step size for $x$ is set as zero at the first $m_0$ steps in the next epoch, we obtain for $i=1,2,\ldots,m_0$,
    \begin{align}\label{eq:for_next_induc2}
        \varepsilon _{m+i}=\left\| y_{m+i}-y_{m+i}^{*} \right\| \le \left( 1-\theta \mu _g \right) ^{\frac{i}{2}}\left\| y_m-y_{m}^{*} \right\| \le \varepsilon _m\le \frac{\sqrt{3}\mu _g}{4\left( m+1 \right) ^3L_{gyy}}.
    \end{align}
    Specifically,
    \begin{align}\label{eq:for_next_induc3}
        \left\| y_{m+m_0}-y_{m+m_0}^{*} \right\| \le \left( 1-\theta \mu _g \right) ^{\frac{m_0}{2}}\left\| y_m-y_{m}^{*} \right\| \le \frac{\sqrt{3}\mu _g}{8\left( m+1 \right) ^3L_{gyy}}
    \end{align}
    if we choose $m_0$ so that $\left( 1-\theta \mu _g \right) ^{\frac{m_0}{2}}\le \frac{1}{2}$. 
    Therefore, by induction within an epoch \eqref{eq:induction_with_epoch1}, \eqref{eq:induction_with_epoch2}, \eqref{eq:induction_with_epoch3}, \eqref{eq:induction_with_epoch4}, \eqref{eq:induction_with_epoch5} and induction between epochs \eqref{eq:for_next_induc1}, \eqref{eq:for_next_induc2}, \eqref{eq:for_next_induc3}, we conclude the lemma.
\end{proof}

\subsection{Proof of \cref{lem:descent_res}}
\begin{lemma}\label{lem:descent_res_ap}
Suppose Assumptions \ref{assu:f}, \ref{assu:g}, \ref{assu:strongg}, \ref{assu:boundfx} and \ref{assu:y_initial} hold. If within each epoch, we set the step size $\theta\sim\mathcal{O}({1}/{m})$ a constant for $y$ and the step size for $x$ as zero in the first $m_0\sim\mathcal{O}(1)$ steps, and the others as an appropriate constant $\lambda\sim\mathcal{O}({1}/{m^4})$, we have the following inequality,
\begin{equation*}
\frac{\left\| \bar{r}_j \right\|}{\left\| \bar{r}_0 \right\|}\le 2\sqrt{\tilde{\kappa}\left( j \right)}\left( \frac{\sqrt{\tilde{\kappa}\left( j \right)}-1}{\sqrt{\tilde{\kappa}\left( j \right)}+1} \right) ^j+\sqrt{j}L_{gyy}\varepsilon _j\tilde{\kappa}\left( j \right),
\end{equation*}
where 
$
\tilde{\kappa}\left( j \right) :=\frac{L_{gy}+\frac{2\sqrt{3}}{3}\left( j+1 \right) ^3L_{gyy}\varepsilon _j}{\mu _g-\frac{2\sqrt{3}}{3}\left( j+1 \right) ^3L_{gyy}\varepsilon _j}.
$
\end{lemma}
\begin{proof}
    \cref{lem:bound_v} guarantees the condition 
    \begin{equation*}
        \mu _g-\frac{2\sqrt{3}}{3}\left( j+1 \right) ^3L_{gyy}{\varepsilon} _j>0
    \end{equation*}
    is satisfied. The remaining proof can be directly adapted from \cref{lem:vdescent}.
\end{proof}

\section{Proof of The Main Theorem}\label{sec:proof_main_theory}
In this section, we provide proof of the main theorem presented in~\cref{sec:convergence}. Let $A_k=\nabla^2_{yy}g(x_k,y_k)$ and $b_k=\nabla_yf(x_k,y_k)$, and let the reference values be~$A_{k}^{*}=\nabla _{yy}^{2}g\left( x_k,y^*_k \right)$, $b^*_k=\nabla _yf\left( x_k,y^*_k \right)$ and $v_k^*=(A_k^*)^{-1}b_k^*$. To begin with, a short proof sketch is provided for guidance, which is structured in four main steps.
\paragraph{Step1: upper-bounding the residual $\norm{v_{k}-v_{k}^*}$}
 \hspace*{\fill} \\
\cref{sec:proof_lanc} and \ref{sec:proof312} lay a foundation to bound the residual term $\norm{v_{k}-v_{k}^*}$ (as a corollary of \cref{lem:descent_res_ap}).
\paragraph{Step2: studying the descent property of $\norm{y_k-y_k^*}$}
 \hspace*{\fill} \\
\cref{lem:y_descent} reveals the descent property of the estimation error for $y^*$ as follows,
\begin{equation*}
    \left\| y_{k+1}-y_{k+1}^{*} \right\| ^2\le \left( 1+\sigma \right) \left( 1-\theta \mu _g \right) \left\| y_k-y_{k}^{*} \right\| ^2+\left( 1+\frac{1}{\sigma} \right) \left( 1-\theta \mu _g \right) \left( \frac{L_{gx}}{\mu _g} \right) ^2\left\| x_{k+1}-x_k \right\| ^2.
\end{equation*}

\paragraph{Step3: Controlling the hyper-gradient estimation error $\|\widetilde{\nabla}\varphi(x_k)-\nabla \varphi(x_k)\|$}
 \hspace*{\fill} \\
Defining
\begin{equation*}
    \delta _k:=\left( \frac{L_{fx}^{2}\mu _{g}^{2}+L_{gxy}^{2}C_{fg}^{2}}{L_{gx}^{2}\mu _{g}^{2}} \right) \left\| y_k-y_{k}^{*} \right\| ^2+\left\| v_k-v_{k}^{*} \right\| ^2,
\end{equation*}
and incorporating the results from the last two steps, then in \cref{lem:hypergrad_esit_remove} we can establish the upper bound for $\|\widetilde{\nabla}\varphi(x_k)-\nabla \varphi(x_k)\|$ and $\delta_k$ recursively.

\paragraph{Step4: Assembling the estimations above and achieving the conclusion}
 \hspace*{\fill} \\
Consider the descent property of $\mathcal{L}_k:=\varphi_k+\delta_k$. Substituting the inequalities developed in step1 to step3, and telescoping the index from $0$ to $K$ gives the convergence results.

The following lemma displays the descent property of the iterates $\{y_k\}$.

\begin{lemma}\label{lem:y_descent}
Suppose Assumptions \ref{assu:g} and \ref{assu:strongg} hold. Setting $0<\theta\le\frac{2}{\mu_g+L_{gy}}$, we have
    \begin{equation*}
        \left\| y_{k+1}-y_{k+1}^{*} \right\| ^2\le \left( 1+\sigma \right) \left( 1-\theta \mu _g \right) \left\| y_k-y_{k}^{*} \right\| ^2+\left( 1+\frac{1}{\sigma} \right) \left( 1-\theta \mu _g \right) \left( \frac{L_{gx}}{\mu _g} \right) ^2\left\| x_{k+1}-x_k \right\| ^2,
    \end{equation*}
    for any $\sigma>0$.
\end{lemma}
\begin{proof}
\cref{alg:LancBiO} executes a single-step gradient descent on the strongly convex function $g(x_{k+1},\cdot)$ during the outer iteration. Leveraging the established convergence properties of strongly convex functions~\citep{nesterov2018lectures}, we are thus able to derive the following.
\begin{equation*}
    \left\| y_{k+1}-y_{k+1}^{*} \right\| ^2\le \left( 1-\theta \mu _g \right) \left\| y_k-y_{k+1}^{*} \right\| ^2.
\end{equation*}
By Young's inequality that $\abs{a+b}^2\le (1+\sigma )\abs{a} ^2+\left( 1+\frac{1}{\sigma} \right) \left\| b \right\| ^2$ with any $\sigma>0$,
\begin{equation*}
    \left\| y_{k+1}-y_{k+1}^{*} \right\| ^2\le \left( 1+\sigma \right) \left( 1-\theta \mu _g \right) \left\| y_k-y_{k}^{*} \right\| ^2+\left( 1+\frac{1}{\sigma} \right) \left( 1-\theta \mu _g \right) \left( \frac{L_{gx}}{\mu _g} \right) ^2\left\| x_{k+1}-x_k \right\| ^2.
\end{equation*}
\end{proof}

In the context of bilevel optimization, we define the initial residual in the $(h+1)$-th epoch as
\begin{equation*}
    r_{h+1}:=b_{mh+1}-A_{mh+1}\bar{v}_h,
\end{equation*}
and the residual in $k$-th step
\begin{equation*}
    r_k:=\left( b_{mh+1}-A_{mh+1}\bar{v}_h \right) -A_{k}^{*}\Delta v_k.
\end{equation*}

Based on the boundness of $v_k$ in \cref{lem:bound_v} we can estimate
\begin{align}\label{eq:esti_v}
        \mu _g\left\| v_k-v_{k}^{*} \right\| \le& \left\| \left( b_{k}^{*}-A_{k}^{*}\bar{v}_h \right) -A_{k}^{*}\Delta v_k \right\| \nonumber
        \\
        =&\left\| \left( b_{k}^{*}-A_{k}^{*}\bar{v}_h \right) -A_{k}^{*}\Delta v_k-r_k+r_k \right\| \nonumber
        \\
        \le& \left\| b_{k}^{*}-b_{mh+1} \right\| +\left\| A_{k}^{*}-A_{mh+1} \right\| \left\| \bar{v}_h \right\| +\left\| r_k \right\| \nonumber
        \\
        \le&~L_{fy}\left\| \left( x_k,y_{k}^{*} \right) -\left( x_{mh+1},y_{mh+1} \right) \right\| \nonumber
        \\
        &+L_{gyy}\left\| \left( x_k,y_{k}^{*} \right) -\left( x_{mh+1},y_{mh+1} \right) \right\| \left\| \bar{v}_h \right\| +\left\| r_k \right\| \nonumber 
        \\
       = &\left( L_{fy}+L_{gyy}\left\| v_{mh} \right\| \right) \varepsilon _{j}^{h}+\left\| r_k \right\|  \nonumber
        \\
        \le &\left( L_{fy}+L_{gyy}C_v \right) \varepsilon _{j}^{h}+\left\| r_k \right\|,
\end{align}
which comes from $\norm{(A_j^*)^{-1}}\le \frac{1}{\mu_g}$, $v_k=\bar{v}_h+\Delta v_k$, and $\norm{v_k}\le C_v$.

\begin{lemma}\label{lem:hypergrad_esit_remove}
         Suppose Assumptions \ref{assu:f}, \ref{assu:g}, \ref{assu:strongg}, \ref{assu:boundfx} and \ref{assu:y_initial} hold. Within each epoch, we set the step size $\theta\sim\mathcal{O}({1}/{m})$ a constant for $y$ and the step size for $x$ as zero in the first $m_0$ steps, and the others as an appropriate constant $\lambda\sim\mathcal{O}({1}/{m^4})$, then the iterates
         \begin{equation*}
             \{x_k\}\ \text{ for }\ k=mh+j,\ h=0,1,2,\ldots,\ \text{ and }\ j=m_0+1,m_0+2,\ldots,m,
         \end{equation*}
         generated by \cref{alg:LancBiO} satisfy
    \begin{equation}\label{eq:hg_estimate_error}
    \begin{aligned}
        \left\| \widetilde{\nabla }\varphi \left( x_k \right) -\nabla \varphi \left( x_k \right) \right\| ^2 \le&~3L_{gx}^{2}\delta_k,
    \end{aligned}
    \end{equation}
    and 
    \begin{align}
    \delta _k\le&\ \iota ^{2\left( m-m_0 \right)}\left( \delta _{m\left( h-1 \right)} \right) +\iota ^{2m}\delta _{m\left( h-1 \right)}+\iota ^{2m}\delta _{m\left( h-2 \right)}    \nonumber
    \\
    &+12m^2\lambda ^2\omega _{\varphi}\Big ( \| \widetilde{\nabla }\varphi (x_{mh}) \| ^2 +\| \widetilde{\nabla }\varphi (x_{m( h-1 )}) \| ^2 +\| \widetilde{\nabla }\varphi (x_{m( h-2 )}) \| ^2   \nonumber
    \\
    &+\sum_{t=m_0}^{j-1}{\| \widetilde{\nabla }\varphi ( x_{mh+t} ) \| ^2} + \sum_{t=m_0}^{m-1}{\| \widetilde{\nabla }\varphi ( x_{m( h-1 ) +t} ) \| ^2}+\sum_{t=m_0}^{m-1}{\| \widetilde{\nabla }\varphi ( x_{m( h-2 ) +t} ) \| ^2} \Big), \label{eq:new_estimnate_delta}
    \end{align}
    where $0<\iota<1$, $m_0\sim\Omega(\log m)$ are constants, $m$ is the subspace dimension,
    \begin{equation*}
        \delta _k=\left( \frac{L_{fx}^{2}\mu _{g}^{2}+L_{gxy}^{2}C_{fg}^{2}}{L_{gx}^{2}\mu _{g}^{2}} \right) \left\| y_k-y_{k}^{*} \right\| ^2+\left\| v_k-v_{k}^{*} \right\| ^2.   
    \end{equation*}
\end{lemma}
\begin{proof}
By \eqref{eq:esti_v}, conclusion from \cref{lem:vdescent} and the Young's inequality, it holds that
\begin{equation*}
    \begin{aligned}
        \left\| v_k-v_{k}^{*} \right\| ^2&\le 2\left( \frac{L_{fy}+L_{gyy}C_v}{\mu _g} \right) ^2\left( \varepsilon _{j}^{h} \right) ^2+2\frac{1}{\mu _{g}^{2}}\left\| r_k \right\| ^2
        \\
        &\le 2\left( \frac{L_{fy}+L_{gyy}C_v}{\mu _g} \right) ^2\left( \varepsilon _{j}^{h} \right) ^2+2\frac{1}{\mu _{g}^{2}}\left( 2\sqrt{\tilde{\kappa}\left( j \right)}\left( \frac{\sqrt{\tilde{\kappa}\left( j \right)}-1}{\sqrt{\tilde{\kappa}\left( j \right)}+1} \right) ^j+\sqrt{j}L_{gyy}\varepsilon _{j}^{h}\tilde{\kappa}\left( j \right) \right) ^2\left\| r_{k0} \right\| ^2
        \\
        &\le 2\left( \frac{L_{fy}+L_{gyy}C_v}{\mu _g} \right) ^2\left( \varepsilon _{j}^{h} \right) ^2+4\frac{1}{\mu _{g}^{2}}\left( 4\tilde{\kappa}\left( j \right) \left( \frac{\sqrt{\tilde{\kappa}\left( j \right)}-1}{\sqrt{\tilde{\kappa}\left( j \right)}+1} \right) ^{2j}+\left( \sqrt{j}\tilde{\kappa}\left( j \right) L_{gyy} \right) ^2\left( \varepsilon _{j}^{h} \right) ^2 \right) \left\| r_{k0} \right\| ^2.
    \end{aligned}
\end{equation*}
An estimate can be made for $\norm{r_{h+1}}$,
\begin{equation*}
    \begin{aligned}
   \norm{r_{h+1}} =&\left\| b_{mh+1}-A_{mh+1}\bar{v}_h \right\| 
    \\
    =&\left\| b_{mh+1}-b_{mh}^{*}+A_{mh}^{*}v_{mh}^{*}-A_{mh+1}v_{mh} \right\| 
    \\
    \le&~L_{fy}\left\| \left( x_{mh+1},y_{mh+1} \right) -\left( x_{mh},y_{mh}^{*} \right) \right\| 
    \\
    &+\frac{C_{fy}}{\mu _g}L_{gyy}\left\| \left( x_{mh+1},y_{mh+1} \right) -\left( x_{mh},y_{mh}^{*} \right) \right\| +L_{gy}\left\| v_{mh}^{*}-v_{mh} \right\| 
    \\
    =&~L_{gy}\left\| v_{mh}-v_{mh}^{*} \right\| +\frac{\mu _gL_{fy}+C_{fy}L_{gyy}}{\mu _g}\left( 1+\frac{L_{gx}}{\mu _g} \right) \left\| x_{mh+1}-x_{mh} \right\|, 
    \end{aligned}
\end{equation*}
and thus
\begin{align}\label{eq:vk_vsk}
        \left\| v_k-v_{k}^{*} \right\| ^2 \le&~\omega _{\varepsilon}\left( \varepsilon _{j}^{h} \right) ^2+\frac{L_{gy}^{2}}{\mu _{g}^{2}}\tilde{\kappa}\left( j \right) \left( \frac{\sqrt{\tilde{\kappa}\left( j \right)}-1}{\sqrt{\tilde{\kappa}\left( j \right)}+1} \right) ^{2j}\left\| v_{mh}-v_{mh}^{*} \right\| ^2 \nonumber
        \\
        &+\frac{\tilde{\kappa}\left( j \right)}{\mu _{g}^{2}}\left( \frac{\mu _gL_{fy}+C_{fy}L_{gyy}}{\mu _g}\left( 1+\frac{L_{gx}}{\mu _g} \right) \right) ^2\left( \frac{\sqrt{\tilde{\kappa}\left( j \right)}-1}{\sqrt{\tilde{\kappa}\left( j \right)}+1} \right) ^{2j}\left\| x_{mh+1}-x_{mh} \right\| ^2,
\end{align}
where
\begin{equation*}
    \omega _{\varepsilon}=2\left( \frac{L_{fy}+L_{gyy}C_v}{\mu _g} \right) ^2+\frac{4}{\mu _{g}^{2}}\left( \sqrt{m}\tilde{\kappa}\left( m \right) L_{gyy} \right) ^2\left( L_{gy}^{2}C_{v}^{2}+\left( \frac{\mu _gL_{fy}+C_{fy}L_{gyy}}{\mu _g}\left( 1+\frac{L_{gx}}{\mu _g} \right) \right) ^2C_{s}^{2} \right) \sim \mathcal{O} \left( m \right).
\end{equation*}
By definition of $\varepsilon_j^h$, the update rule of $x$, and Young's inequality, it holds that
\begin{equation*}
    \left( \varepsilon _{j}^{h} \right) ^2\le 2m\lambda ^2\left( 1+\frac{L_{gx}}{\mu _g} \right) ^2\sum_{i=1}^{j-1}{\left\| \widetilde{\nabla }\varphi \left( x_{mh+i} \right) \right\| ^2}+2\left\| y_{mh+i\left( j \right)}-y_{mh+i\left( j \right)}^{*} \right\| ^2
\end{equation*}
for some $m_0+1\le i(j)\le j$. Then we apply the descent property of $\|y_{k}-y_{k}^*\|$ (\cref{lem:y_descent}) recursively to derive
\begin{align}
    \left( \varepsilon _{j}^{h} \right) ^2\le&\ 2m\lambda ^2\left( 1+\frac{L_{gx}}{\mu _g} \right) ^2\sum_{i=1}^{j-1}{\left\| \widetilde{\nabla }\varphi \left( x_{mh+i} \right) \right\| ^2}+2\left( 1+\sigma \right) ^{i\left( j \right)}\left( 1-\theta \mu _g \right) ^{i\left( j \right)}\left\| y_{mh}-y_{mh}^{*} \right\| ^2    \nonumber
    \\
    &+2\left( 1+\frac{1}{\sigma} \right) \left( 1-\theta \mu _g \right) \left( \frac{L_{gx}}{\mu _g} \right) ^2\sum_{r=0}^{i\left( j \right)}{\left( 1+\sigma \right) ^r\left( 1-\theta \mu _g \right) ^r\left\| x_{k-r}-x_{k-r-1} \right\| ^2}    \nonumber
    \\
    \le&\ 2\lambda ^2\left( 1+\frac{L_{gx}}{\mu _g} \right) ^2\left( m+\left( 1+\frac{1}{\sigma} \right) \left( 1-\theta \mu _g \right) \right) \sum_{i=1}^{j-1}{\left\| \widetilde{\nabla }\varphi \left( x_{mh+i} \right) \right\| ^2}+2\left\| y_{mh}-y_{mh}^{*} \right\| ^2    \nonumber
    \\
    \le&\ 2\lambda ^2\left( 1+\frac{L_{gx}}{\mu _g} \right) ^2\left( m+\left( 1+\frac{1}{\sigma} \right) \left( 1-\theta \mu _g \right) \right) \sum_{i=1}^{j-1}{\left\| \widetilde{\nabla }\varphi \left( x_{mh+i} \right) \right\| ^2}    \nonumber
    \\
    &+2\left( 1+\sigma \right) ^{m_0}\left( 1-\theta \mu _g \right) ^{m_0}\left\| y_{mh}-y_{mh}^{*} \right\| ^2. \label{eq:new_estimate_eps}
\end{align}
Since \cref{lem:bound_v} reveals that under the appropriate step-size setting $\lambda\sim\mathcal{O}(\frac{1}{m^4})$ and $\theta\sim\mathcal{O}(\frac{1}{m})$, $\tilde{\kappa}(j)\le \tilde{\kappa}$ for $\tilde{\kappa}\sim\Omega(\frac{L_{gy}}{\mu_{gy}})$, we also choose $\sigma>0$ such that
\begin{equation}\label{eq:iota_remove}
    \iota :=\max \left\{ \frac{\sqrt{\tilde{\kappa}}-1}{\sqrt{\tilde{\kappa}}+1},\sqrt{\left( 1+\sigma \right) \left( 1-\theta \mu _g \right)} \right\} <1
\end{equation}
Additionally, set the warm-up steps $m_0$ to satisfy
\begin{equation*}
\begin{aligned}
    \iota ^{-2m_0}\ge \max \Bigg\{ \left( \frac{L_{fx}^{2}\mu _{g}^{2}+L_{gxy}^{2}C_{fg}^{2}}{L_{gx}^{2}\mu _{g}^{2}} \right)^{-1} \left( \frac{L_{fx}^{2}\mu _{g}^{2}+L_{gxy}^{2}C_{fg}^{2}}{L_{gx}^{2}\mu _{g}^{2}}+m \right) ,
    \\
    \frac{L_{gy}^{2}}{\mu _{g}^{2}}\tilde{\kappa} \left( \frac{\mu _gL_{fy}+C_{fy}L_{gyy}}{\mu _g}\left( 1+\frac{L_{gx}}{\mu _g} \right) \right) ^2\!\!,\,2\omega_\varphi \ \Bigg \} ,
\end{aligned}
\end{equation*}
which means $m_0\sim \Omega \left(\log  m  \right)$. In this manner, adding $\left( \frac{L_{fx}^{2}\mu _{g}^{2}+L_{gxy}^{2}C_{fg}^{2}}{L_{gx}^{2}\mu _{g}^{2}} \right) \left\| y_k-y_{k}^{*} \right\| ^2$ on both sides of \eqref{eq:vk_vsk} and incorporating the estimation for $\epsilon_j^h$~\eqref{eq:new_estimate_eps} yield that
\begin{align}
    \delta _k\le&\ \iota ^{2\left( j-m_0 \right)}\left( \delta _{mh} \right) +\left\| y_{mh}-y_{mh}^{*} \right\| ^2+6m^2\lambda ^2\omega _{\varphi}\left( \sum_{t=m_0}^{j-1}{\left\| \widetilde{\nabla }\varphi \left( x_{mh+t} \right) \right\| ^2}+\left\| \widetilde{\nabla }\varphi (x_{mh}) \right\| ^2 \right) \nonumber
    \\
    \le&\ \iota ^{2\left( j-m_0 \right)}\left( \delta _{mh} \right) +\iota ^{2m}\| y_{m\left( h-1 \right)}-y_{m\left( h-1 \right)}^{*} \| ^2 \nonumber
    \\
    &+6m^2\lambda ^2\omega _{\varphi} \Big( \sum_{t=m_0}^{j-1}{\| \widetilde{\nabla }\varphi ( x_{mh+t} ) \| ^2}+\| \widetilde{\nabla }\varphi (x_{mh}) \| ^2 \nonumber
    \\
    &\quad\quad\quad\quad\quad + \sum_{t=m_0}^{m-1}{\| \widetilde{\nabla }\varphi ( x_{m( h-1 ) +t} ) \| ^2}+\| \widetilde{\nabla }\varphi (x_{m( h-1 )}) \| ^2  \Big) \label{eq:estimate_delta_remove_1},
\end{align}
where we apply \cref{lem:y_descent} recursively to obtain the second inequality, and adopt the notation
\begin{equation*}
    \begin{aligned}
       \omega _{\varphi}:=\max &\left\{ \frac{\omega _{\varepsilon}}{m}\left( 1+\frac{L_{gx}}{\mu _g} \right) ^2,\frac{\tilde{\kappa} \iota }{m^2\mu _{g}^{2}}\left( \frac{\mu _gL_{fy}+C_{fy}L_{gyy}}{\mu _g}\left( 1+\frac{L_{gx}}{\mu _g} \right) \right)^2, \right.
       \\
        &\left. \frac{1}{m^2}\left( 1+\frac{1}{\sigma} \right) \left( 1-\theta \mu _g \right) \left( \frac{L_{fx}^{2}\mu _{g}^{2}+L_{gxy}^{2}C_{fg}^{2}}{L_{gx}^{2}\mu _{g}^{2}}+m \right) \left( \frac{L_{gx}}{\mu _g} \right) ^2 \right\} \sim \mathcal{O} \left( 1 \right).     
    \end{aligned}
\end{equation*}
Consequently, expanding $\delta_{mh}$ in \eqref{eq:estimate_delta_remove_1} in the same way derives \eqref{eq:new_estimnate_delta}. Regrading the upper-bound of the hyper-gradient estimation error, we have
\begin{align*}
        \left\| \widetilde{\nabla }\varphi \left( x_k \right) -\nabla \varphi \left( x_k \right) \right\| ^2 \le&~3\left\| \nabla _xf\left( x_k,y_k \right) -\nabla _xf\left( x_k,y_{k}^{*} \right) \right\| ^2+3\left\| \nabla _{xy}^{2}g\left( x_k,y_k \right) \right\| ^2\left\| v_k-v_{k}^{*} \right\| ^2 \nonumber
        \\
        &+3\left\| \nabla _{xy}^{2}g\left( x_k,y_k \right) -\nabla _{xy}^{2}g\left( x_k,y_{k}^{*} \right) \right\| ^2\left\| v_{k}^{*} \right\| ^2 \nonumber
        \\
        \le&~3L_{fx}^{2}\left\| y_k-y_{k}^{*} \right\| ^2+3L_{gx}^{2}\left\| v_k-v_{k}^{*} \right\| ^2+3L_{gxy}^{2}\left( \frac{C_{fy}}{\mu _g} \right) ^2\left\| y_k-y_{k}^{*} \right\| ^2 \nonumber
        \\
       = &~3L_{gx}^{2}\delta _k.
\end{align*}
\end{proof}

\begin{theorem}
    Suppose Assumptions \ref{assu:f}, \ref{assu:g}, \ref{assu:strongg}, \ref{assu:boundfx} and \ref{assu:y_initial} hold. Within each epoch, if we set the step size $\theta\sim\mathcal{O}({1}/{m})$ a constant for $y$, and the step size for $x$ as zero in the first $m_0$ steps and the others as an appropriate constant $\lambda\sim\mathcal{O}({1}/{m^4})$, the iterates $\{x_k\}$ generated by \cref{alg:LancBiO} satisfy
    \begin{equation*}
        \frac{m}{K\left( m-m_0 \right)} \hspace{-3mm}\sum_{\footnotesize \substack{k=0,\\
	 (k\,\emph{\texttt{mod}}\,m)>m_0 }}^K \hspace{-3mm}{\left\| \nabla \varphi \left( x_k \right) \right\| ^2} =  \mathcal{O} \left( \frac{m\lambda ^{-1}}{K\left( m-m_0 \right)} \right),
    \end{equation*}
    where $m_0\sim\Omega(\log m)$ is a constant and $m$ is the subspace dimension.
\end{theorem}
\begin{proof}
Consider the Lyapunov function $\mathcal{L}_k:=\varphi(x_k)+\delta_k$. According to \cref{lem:hypersmooth}, a gradient descent step leads to the decrease in the hyper-function:
\begin{equation}\label{eq:varphi_descent}
    \begin{aligned}
        \varphi \left( x_{k+1} \right)-\varphi \left( x_k \right) \le&~\left< \nabla \varphi \left( x_k \right) ,x_{k+1}-x_k \right> +\frac{L_{\varphi}}{2}\left\| x_{k+1}-x_k \right\| ^2.
    \end{aligned}
\end{equation}
Then, telescoping $\mathcal{L}_{k+1}-\mathcal{L}_k$ over the index set $\mathcal{I}:=\{k:\,0\le k\le K,\,(k\,{\texttt{mod}}\,m)>m_0 \}$,
\begin{align}
        &\ \sum_{k\in\mathcal{I}}{\mathcal{L} _{k+1}-\mathcal{L} _k}   \nonumber
        \\
        =&\ \sum_{k\in\mathcal{I}}{\varphi ( x_{k+1} ) -\varphi ( x_k )}+\!\!\!\!\!\!\!\!\!\sum_{{\footnotesize \substack{k=1,\\
	 ((k-1)\,{\texttt{mod}}\,m)>m_0 }}}^{K+1}{\!\!\!\!\!\!\!\!\!\delta _k}\ -\sum_{k\in\mathcal{I}}{\delta _k}   \nonumber
        \\
        \overset{(i)}{\le}&\ \sum_{k\in\mathcal{I}}{\left< \nabla \varphi ( x_k ) ,x_{k+1}-x_k \right> +( \frac{L_{\varphi}}{2}+36m^3\omega _{\varphi} ) \| x_{k+1}-x_k \| ^2}+\sum_{e=0}^{h-1}{2m\iota ^{2( m-m_0 )}\delta _{me}}-\sum_{k\in\mathcal{I}}{\delta _k}     \nonumber
        \\
        \overset{(ii)}{\le}&\ \sum_{k\in\mathcal{I}}-\kh{ \frac{\lambda}{2}-\lambda ^2( L_{\varphi}+72m^3\omega _{\varphi} ) } \| \nabla \varphi ( x_k ) \| ^2 + \kh{\frac{\lambda}{2}+\lambda ^2( L_{\varphi}+72m^3\omega _{\varphi} )} \| \nabla \varphi ( x_k ) -\widetilde{\nabla }\varphi ( x_k ) \| ^2   \nonumber
        \\
        &+\sum_{e=0}^{h-1}{2m\iota ^{2( m-m_0 )}\delta _{me}}-\sum_{k\in\mathcal{I}}{\delta _k}    \nonumber
        \\
        \overset{(iii)}{\le}&\ \sum_{k\in\mathcal{I}}{- \kh{\frac{\lambda}{2}-\lambda ^2( L_{\varphi}+72m^3\omega _{\varphi} )} \| \nabla \varphi ( x_k ) \| ^2}  \nonumber
        \\
        &+3L_{gx}^{2}\kh{ \frac{\lambda}{2}+\lambda ^2( L_{\varphi}+72m^3\omega _{\varphi} )} \sum_{k\in\mathcal{I}}{\delta _k}+\sum_{e=0}^{h-1}{2m\iota ^{2( m-m_0 )}\delta _{me}}-\sum_{k\in\mathcal{I}}{\delta _k},      \label{eq:LkLk}
\end{align}
where $(i)$ follows from inequalities \eqref{eq:new_estimnate_delta} and \eqref{eq:varphi_descent}; $(ii)$ results from the update rule $x_{k+1}=x_k-\lambda\widetilde{\nabla}\varphi(x_k)$ and Young's inequality; $(iii)$ comes from \eqref{eq:hg_estimate_error}. Taking the coefficients of $\delta_k$ into account, we set the dimension parameters $(m,m_0)$ satisfying $m\iota ^{2\left( m-m_0 \right)}<1/4$ and the step size $\lambda$ such that
\begin{equation}
\lambda \le \min \left\{ \frac{1}{6L_{gx}^{2}},\ \frac{1}{\left( 12\left( L_{\varphi}+72m^3\omega _{\varphi} \right) L_{gx}^{2} \right) ^{1/2}},\ \frac{1}{4\left( L_{\varphi}+72m^3\omega _{\varphi} \right)} \right\}.
\end{equation}
In this way, we obtain the following result from \eqref{eq:LkLk},
\begin{equation*}
    \mathcal{L} _{K+1}-\mathcal{L} _0=\sum_{{\footnotesize \substack{k=0,\\
	 (k\,{\texttt{mod}}\,m)>m_0 }}}^K{\mathcal{L} _{k+1}-\mathcal{L} _k}\le \sum_{{\footnotesize \substack{k=0,\\
	 (k\,{\texttt{mod}}\,m)>m_0 }}}^K{-\left( \frac{\lambda}{2}-\lambda ^2\left( L_{\varphi}+72m^3\omega _{\varphi} \right) \right) \left\| \nabla \varphi \left( x_k \right) \right\| ^2}.
\end{equation*}
Rearrange the above inequality and denote $\varphi^*:= \min_{x\in\mathbb{R}^{d_x}}\varphi(x)$,
\begin{equation*}
     \sum_{{\footnotesize \substack{k=0,\\
	 (k\,{\texttt{mod}}\,m)>m_0 }}}^K{\left\| \nabla \varphi \left( x_k \right) \right\| ^2}\le \frac{4\left( \varphi \left( x_0 \right) -\varphi ^* \right)}{\lambda}+\frac{4\delta _0}{\lambda},
\end{equation*}
which completes the proof by dividing both sides by $\frac{m-m_0}{m}K$.
\end{proof}

\section{Details on Experiments}\label{sec:details_exp}

\subsection{General settings}
We conduct experiments to empirically validate the performance of the proposed algorithms. We test on~a synthetic problem, a hyper-parameters selection task, and a data hyper-cleaning task. We compare the proposed SubBiO and LancBiO with the existing algorithms in bilevel optimization: stocBiO \citep{ji2021stocbio}, AmIGO-GD and AmIGO-CG \citep{arbel2022amortized}, SOBA \citep{dagreou2022soba} and TTSA \citep{hong2023ttsa}, \revise{F2SA \citep{kwon2023f2sa} and HJFBiO \citep{huang2024HJFBIO}}. The experiments are produced on a workstation that consists of two Intel® Xeon® Gold 6330 CPUs (total 2$\times$28 cores), 512GB RAM, and one NVIDIA A800 (80GB memory) GPU. The synthetic problem and the deep learning experiments are carried out on the CPUs and the GPU, respectively. For wider accessibility and application, we have made the code available on \href{https://github.com/UCAS-YanYang/LancBiO}{https://github.com/UCAS-YanYang/LancBiO}.

For the proposed LancBiO, we initiate the subspace dimension at $1$, and gradually increase it to $m=10$ for the deep learning experiments and to $m=80$ for the synthetic problem. For all the compared algorithms, we employ a grid search strategy to optimize the parameters. The optimal parameters yield the lowest loss. The experiment results are averaged over~$10$ runs. {Note that \cref{assu:y_initial} for initialization is not used in practice, for which we treat it as a theoretical assumption rather than incorporating it into \cref{alg:LancBiO} in this paper.}

In this paper, we consider the algorithms SubBiO and LancBiO in the deterministic scenario, so we initially compare them against the baseline algorithms with a full batch (i.e., deterministic gradient). In this setting, LancBiO yields favorable numerical results. Moreover, in the data hyper-cleaning task, to facilitate a more effective comparison with algorithms designed for stochastic applications, we implement all compared methods with a small batch size, finding that the proposed methods show competitive performance.

\subsection{Data hyper-cleaning}
The data hyper-cleaning task~\citep{shaban2019truncated}, conducted on the MNIST dataset~\citep{lecun1998mnist}, aims to train a classifier in~a corruption scenario, where the labels of the training data are randomly altered to incorrect classification numbers at~a certain probability~$p$, referred to as the corruption rate. The task is formulated as follows,
\begin{equation*}
\begin{array}{cl}
         \min\limits_\lambda& \mathcal{L}_{val}(\lambda, w^*) := \frac{1}{\left|\mathcal{D}_{\text{val}}\right|} \sum_{(x_i, y_i) \in \mathcal{D}_{\text{val}}} L(w^* x_i, y_i) 
         \\[5mm]
         \mathrm{s.\,t.}&
        \begin{aligned}[t]
            w^* &= \argmin_w\ \mathcal{L}_{tr}(w, \lambda) \\
            :&= \frac{1}{\left|\mathcal{D}_{\text{tr}}\right|} \sum_{(x_i, y_i) \in \mathcal{D}_{\text{tr}}} \sigma(\lambda_i) L(w x_i, y_i) + C_r\|w\|^2,
        \end{aligned}
\end{array}
\end{equation*}
where $L(\cdot)$ is the cross-entropy loss, $\sigma(\cdot)$ is the sigmoid function, and $C_r$ is~a regularization parameter. In addition, $w$ serves as~a~linear classifier and $\sigma\kh{\lambda_i}$ can be viewed as the confidence of each data.

In the deterministic setting, where we implement all compared methods with full-batch, the training set, the validation set and the test set contain $5000$, $5000$ and $10000$ samples, respectively. For algorithms that incorporate inner iterations to approximate $y^*$ or $v^*$, we select the inner iteration number from the set $\left\{5 i \mid i=1,2,3,4\right\}$. The step size of inner iteration is selected from the set $\{0.01,0.1,1,10\}$ and the step size of outer iteration is chosen from $\left\{5\times10^i \mid i=-3,-2,-1,0,1,2,3\right\}$. Regarding the Hessian/Jacobian-free algorithm HJFBiO, we set the step size $\delta=1\times 10^{-5}$ to implement finite difference methods. The results are presented in \cref{fig:clean_detem_mean_compare}. Note that LancBiO is crafted for approximating the Hessian inverse vector product $v^*$, while the solid methods stocBiO, TTSA, F2SA, and HJFBiO are not. Consequently, concerning the residual norm of the linear system, i.e., $\norm{A_kv_k-b_k}$, we only compare the results with AmIGO-GD, AmIGO-CG and SOBA. Observe that the proposed subspace-based LancBiO achieves the lowest residual norm and the best test accuracy, and SubBiO is comparable to the other algorithms. Specifically, in~\cref{fig:clean_detem_mean_compare}, the efficiency of LancBiO stems from its accurate approximation of the linear system. Furthermore, we implement the solvers designed for the stochastic setting using mini-batch to enable a broader comparison in \cref{fig:clean_detem_stoc_mean_compare}. It is shown that the stochastic algorithm SOBA tends to converge faster initially, but algorithms employing a full-batch approach achieve higher accuracy.

\begin{figure*}[htbp]
	\centering
	\begin{minipage}{0.36\textwidth}
		\centering
		\includegraphics[width=1\linewidth]{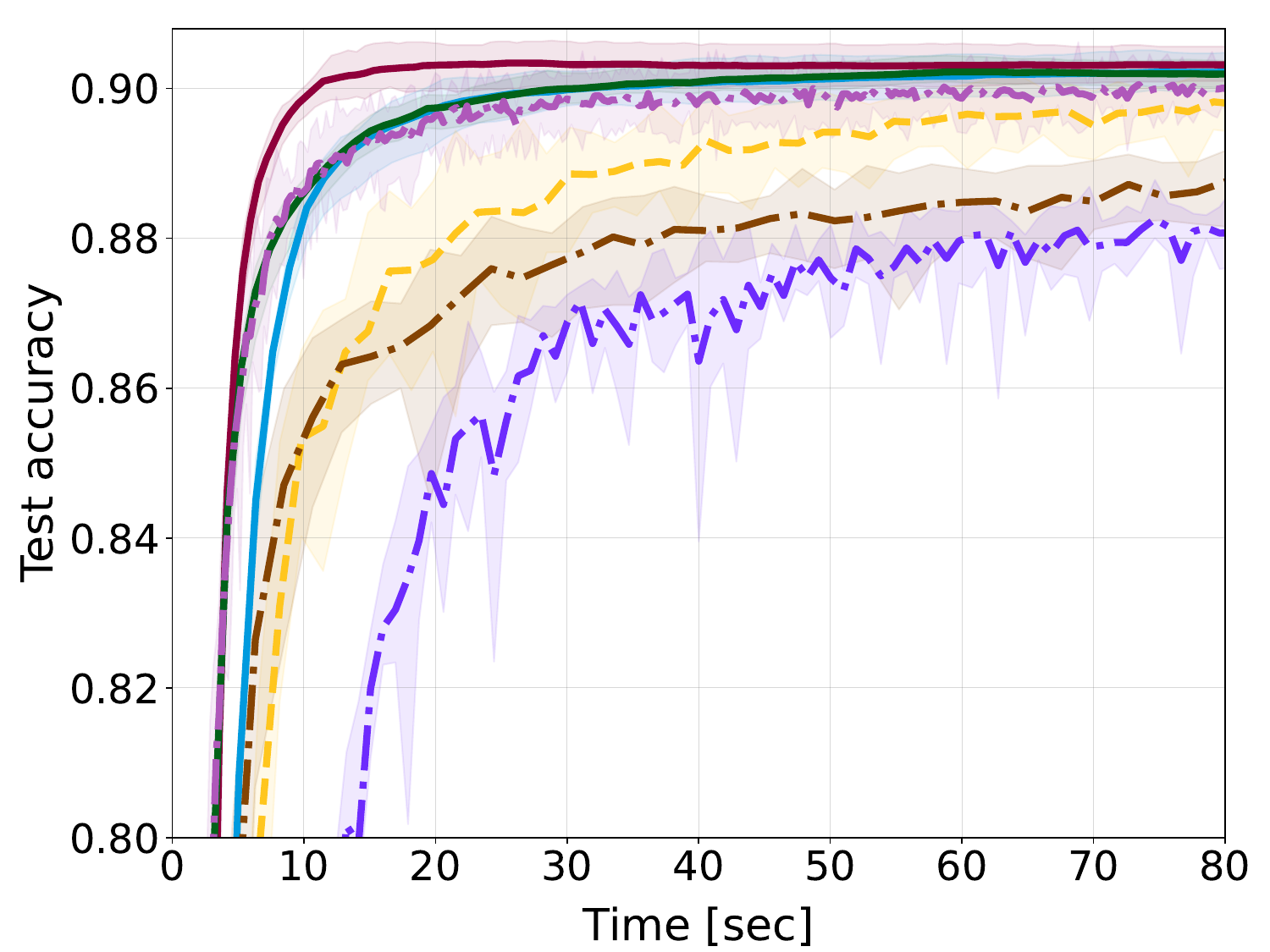}
	\end{minipage}
	\begin{minipage}{0.36\textwidth}
		\centering
		\includegraphics[width=1\linewidth]{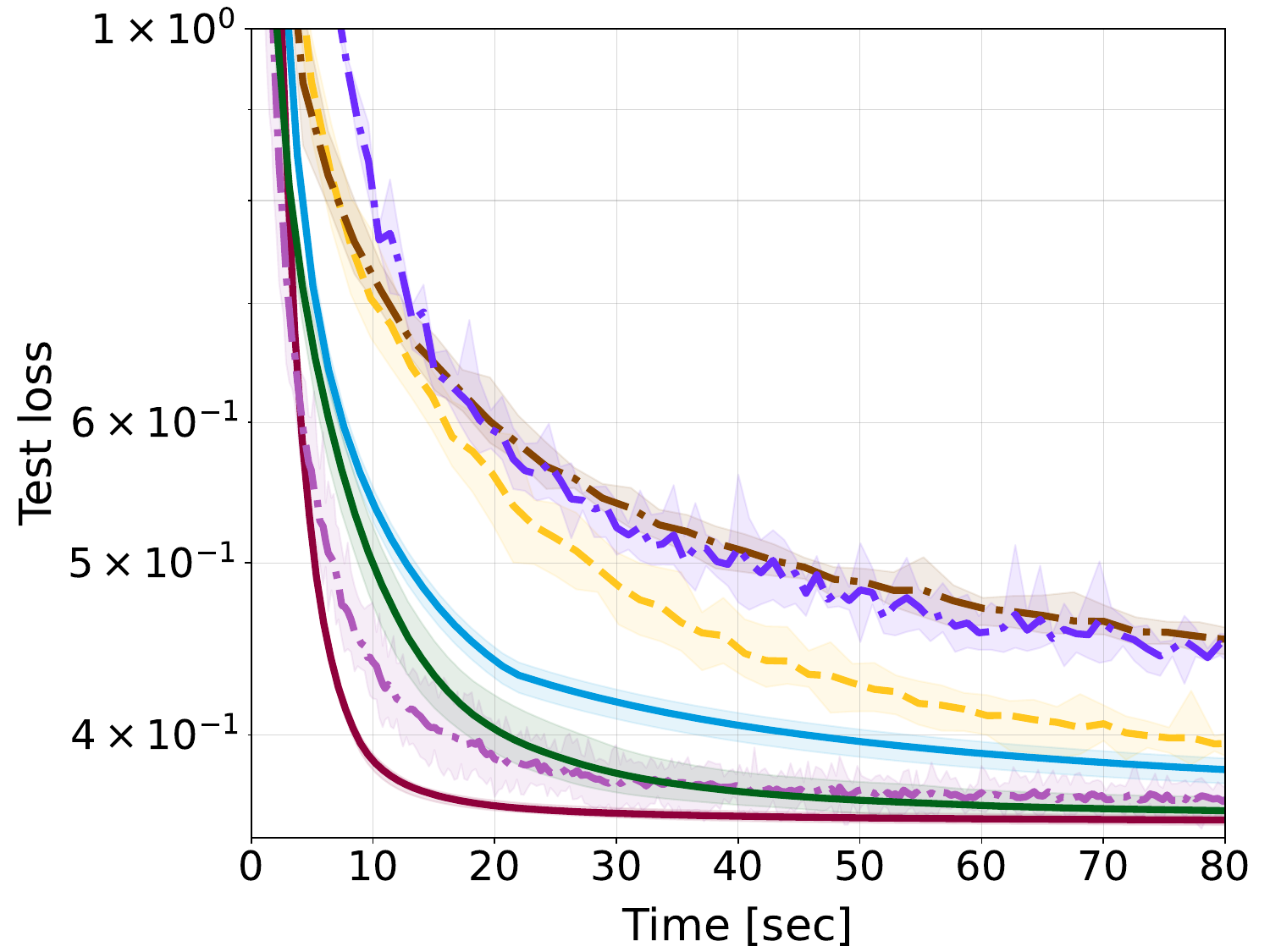}
	\end{minipage}
	\begin{minipage}{0.20\textwidth}
		\centering
		\includegraphics[width=1\linewidth]{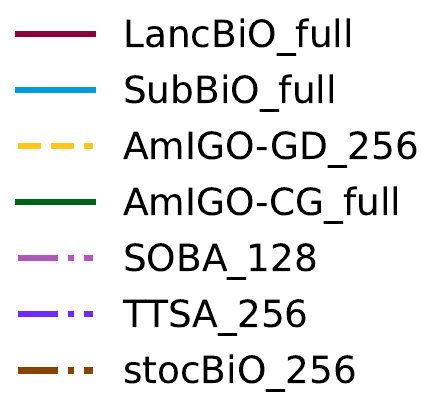}
	\end{minipage}
	\caption{Comparison of the bilevel algorithms on data hyper-cleaning task with mini-batch when~$p=0.5$. The training set, the validation set and the test set contain $5000$, $5000$ and $10000$ samples, respectively. {The post-fix of legend represents the batch size.} \boldt{Left:} test accuracy; \boldt{Right}:test loss.}
	\label{fig:clean_detem_stoc_mean_compare}
\end{figure*}

To explore the potential for extending our proposed methods to a stochastic setting, we also conduct an experiment with stochastic gradients. In this setting, where we implement all compared methods with mini-batch, the training set, the validation set and the test set contain $20000$, $5000$ and $10000$ samples, respectively. For algorithms that incorporate inner iterations to approximate $y^*$ or $v^*$, we select the inner iteration number from the set $\left\{3 i \mid i=1,2,3,4\right\}$. The step size of inner iteration is selected from the set $\{0.01,0.1,1,10\}$, the step size of outer iteration is chosen from $\left\{1\times10^i \mid i=-3,-2,-1,0,1,2,3\right\}$ and the batch size is picked from $\left\{32\times 2^i \mid i=0,1,2,3\right\}$. AmIGO-CG is not presented since it fails in this experiment in our setting. The results in~\cref{fig:clean_stoc_mean_compare} demonstrate that LancBiO maintains reasonable performance with stochastic gradients, exhibiting fast convergence rate, although the final convergence accuracy is slightly lower.

\begin{figure*}[htbp]
	\centering
	\begin{minipage}{0.36\textwidth}
		\centering
		\includegraphics[width=1\linewidth]{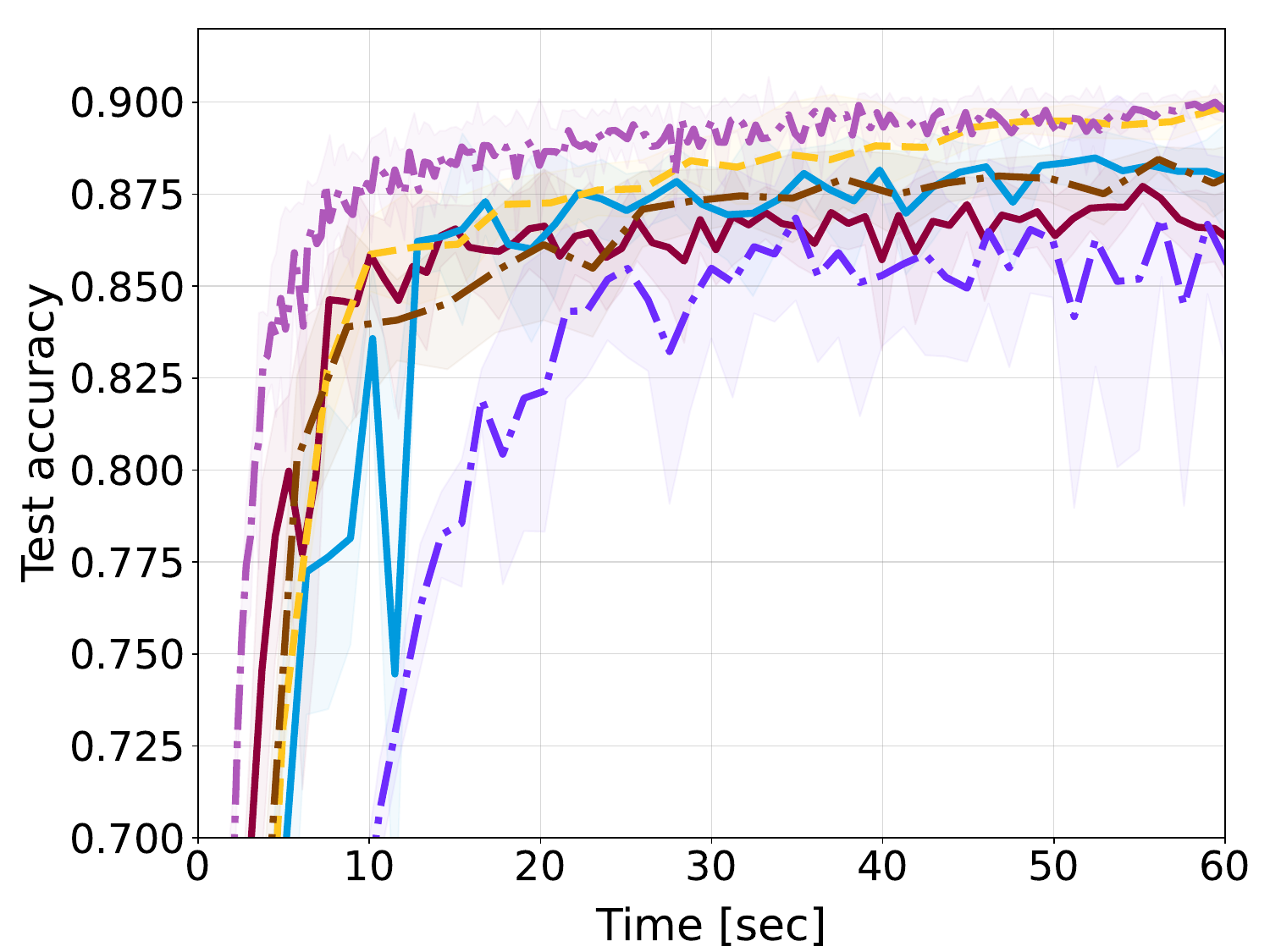}
	\end{minipage}
	\begin{minipage}{0.36\textwidth}
		\centering
		\includegraphics[width=1\linewidth]{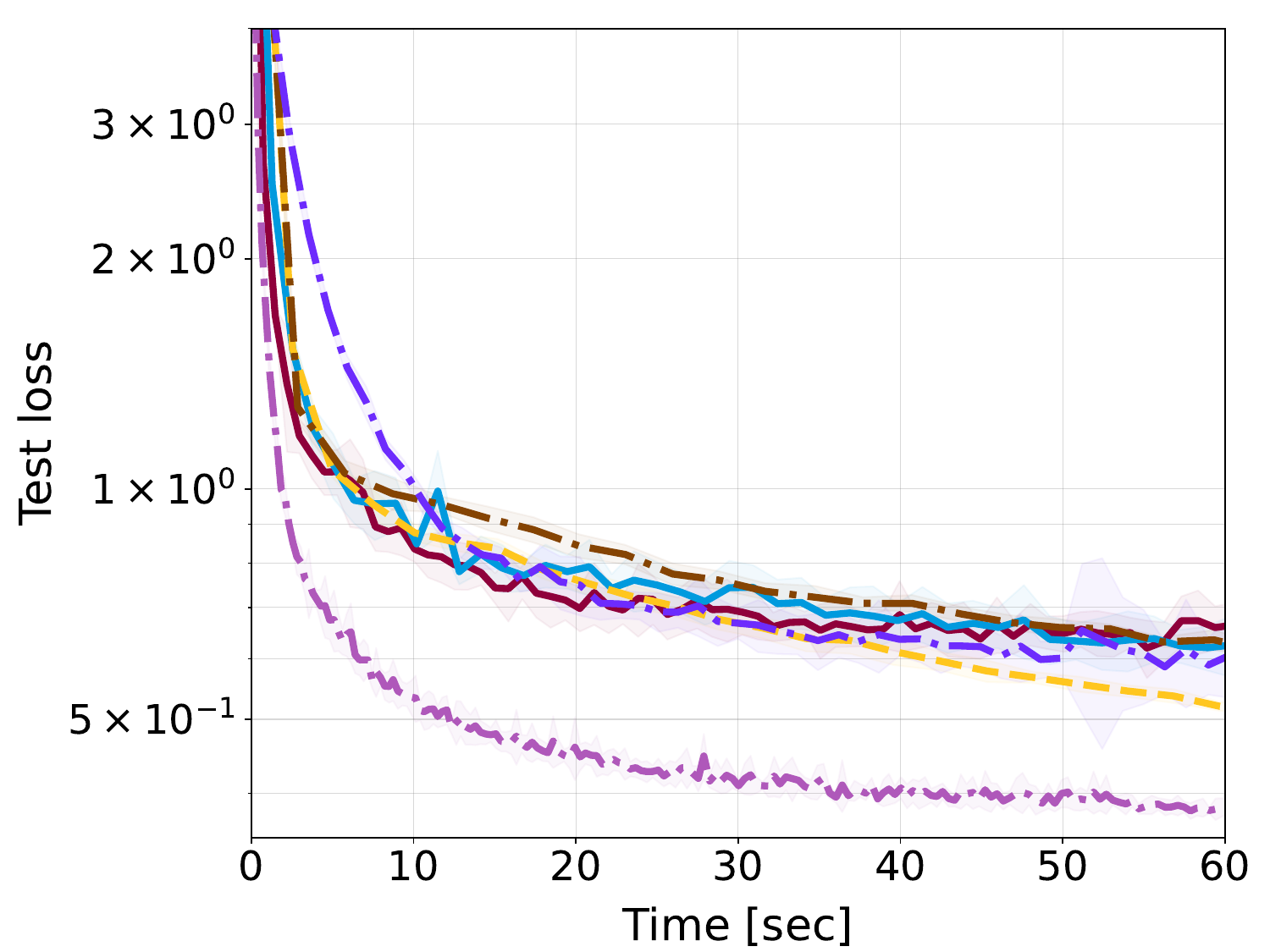}
	\end{minipage}
	\begin{minipage}{0.21\textwidth}
		\centering
		\includegraphics[width=1\linewidth]{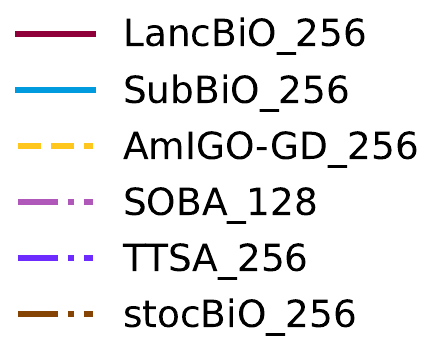}
	\end{minipage}
	\caption{Comparison of the bilevel algorithms on data hyper-cleaning task with mini-batch when~$p=0.5$. The training set, the validation set and the test set contain $20000$, $5000$ and $10000$ samples, respectively. {The post-fix of legend represents the batch size.} \boldt{Left:} test accuracy; \boldt{Right}:test loss.}
	\label{fig:clean_stoc_mean_compare}
\end{figure*}

\revise{Additionally, we also evaluate the performance of bilevel algorithms on Fashion-MNIST \citep{xiao2017fashionMNIST} and Kuzushiji-MNIST \citep{clanuwat2018KMNIST} datasets, both of which present more complexity compared to MNIST. Specifically, Fashion-MNIST serves as a modern replacement for MNIST, featuring grayscale images of clothing items across $10$ categories, and Kuzushiji-MNIST is a culturally rich dataset of handwritten Japanese characters. The results, reported in Figures~\ref{fig:clean_data_Fashion} and~\ref{fig:clean_data_KMNIST}, reveal that LancBiO performs better than other algorithms and showcases robustness across various datasets, and SubBiO delivers a comparable convergence property.}

\begin{figure*}[htbp]
\begin{minipage}{\textwidth}
    \centering
    \includegraphics[width=0.9\linewidth]{fig/hyperclean/detem_80/legend.pdf}
\end{minipage}
\\
\begin{minipage}{0.33\textwidth}
    \centering
    \includegraphics[width=1\linewidth]{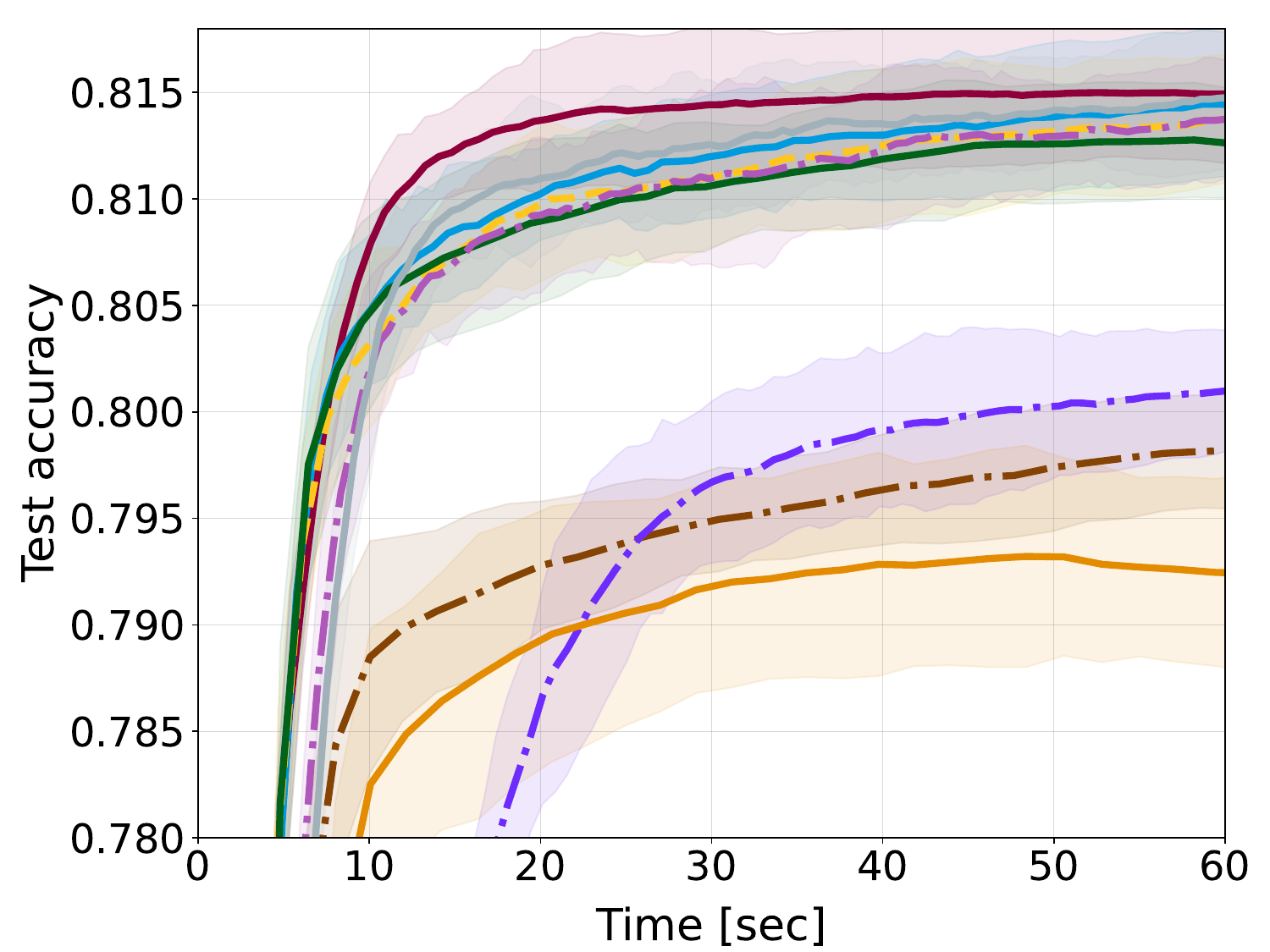}
\end{minipage}
\begin{minipage}{0.33\textwidth}
    \centering
    \includegraphics[width=1\linewidth]{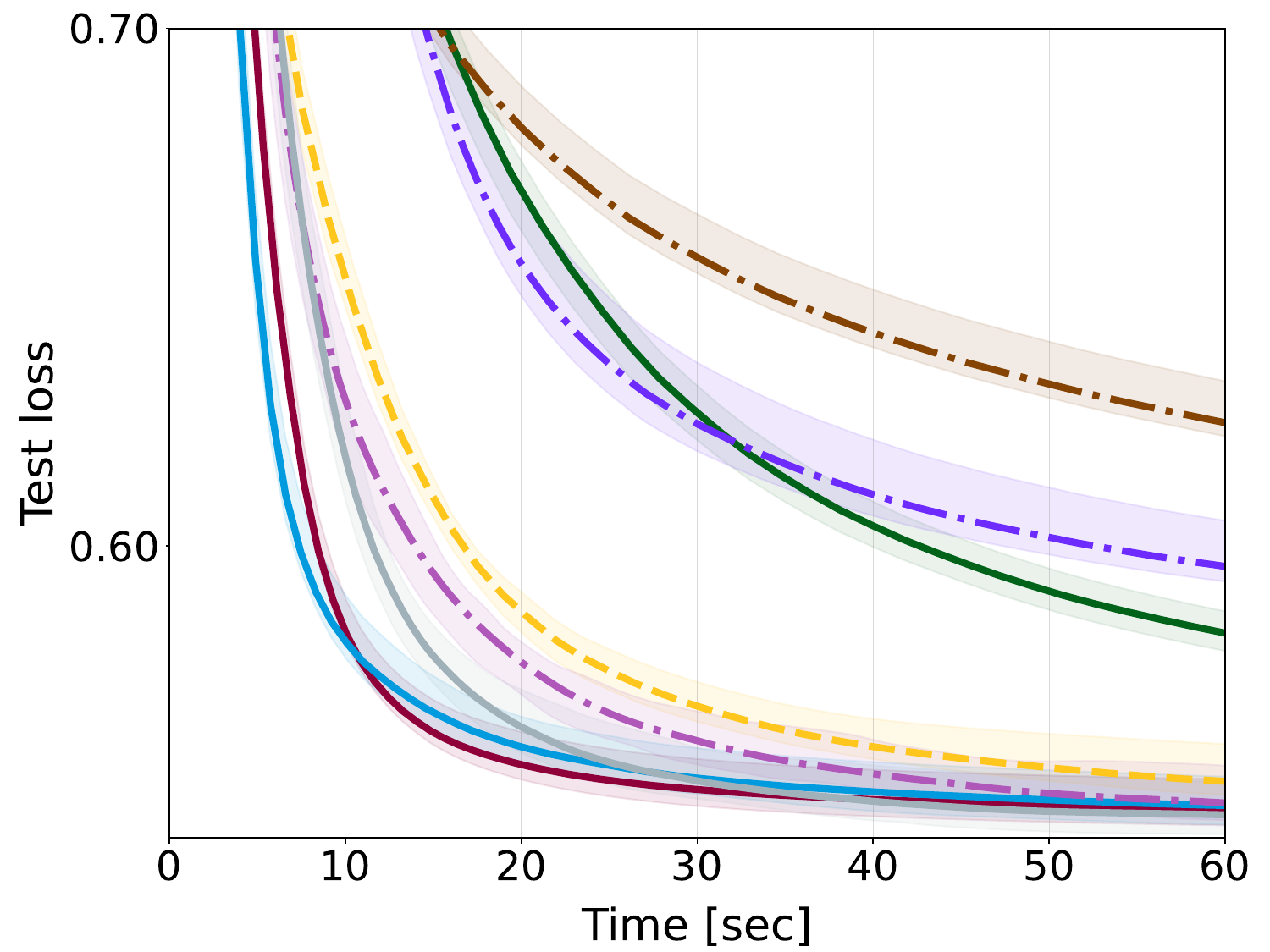}
\end{minipage}
\begin{minipage}{0.33\textwidth}
    \centering
    \includegraphics[width=1\linewidth]{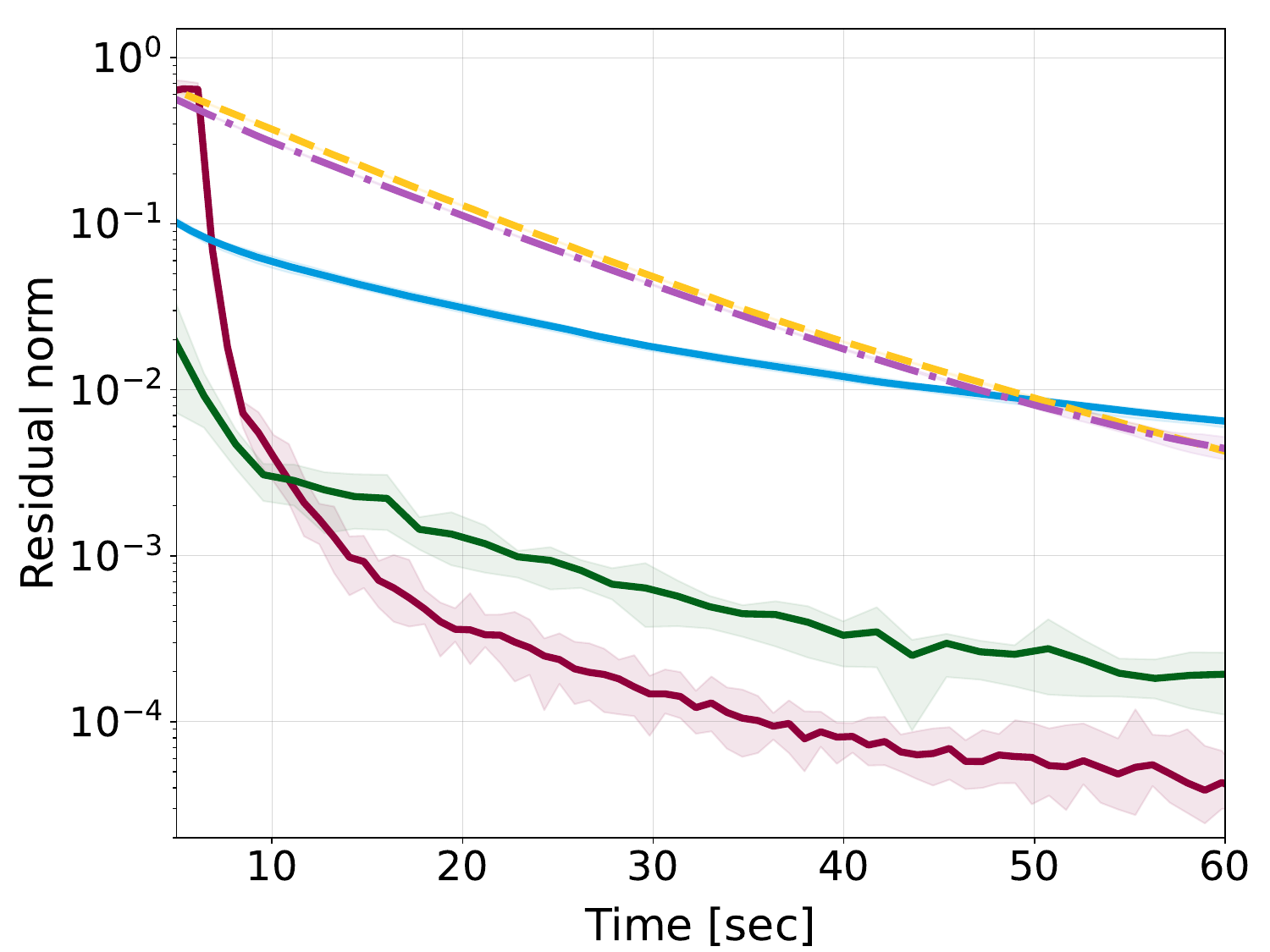}
\end{minipage}
\caption{\revise{Data hyper-cleaning task tested on the Fashion-MNIST dataset when~$p=0.5$. \boldt{Left:} test accuracy; \boldt{Center}: test loss; \boldt{Right}:~residual norm of the linear system, $\norm{A_kv_k-b_k}$.}}
\label{fig:clean_data_Fashion}
\end{figure*}

\begin{figure*}[htbp]
\begin{minipage}{\textwidth}
    \centering
    \includegraphics[width=0.9\linewidth]{fig/hyperclean/detem_80/legend.pdf}
\end{minipage}
\\
\begin{minipage}{0.33\textwidth}
    \centering
    \includegraphics[width=1\linewidth]{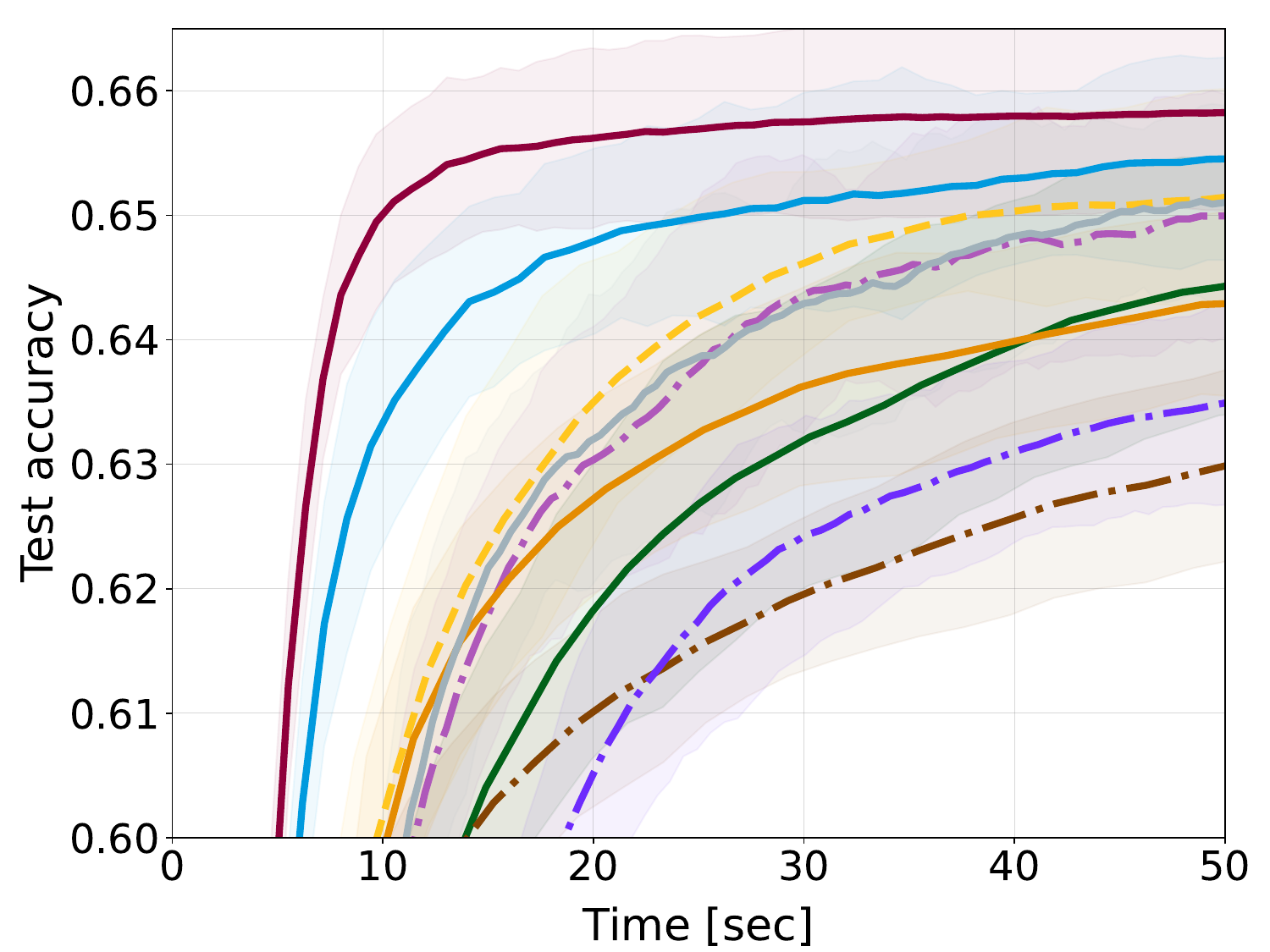}
\end{minipage}
\begin{minipage}{0.33\textwidth}
    \centering
    \includegraphics[width=1\linewidth]{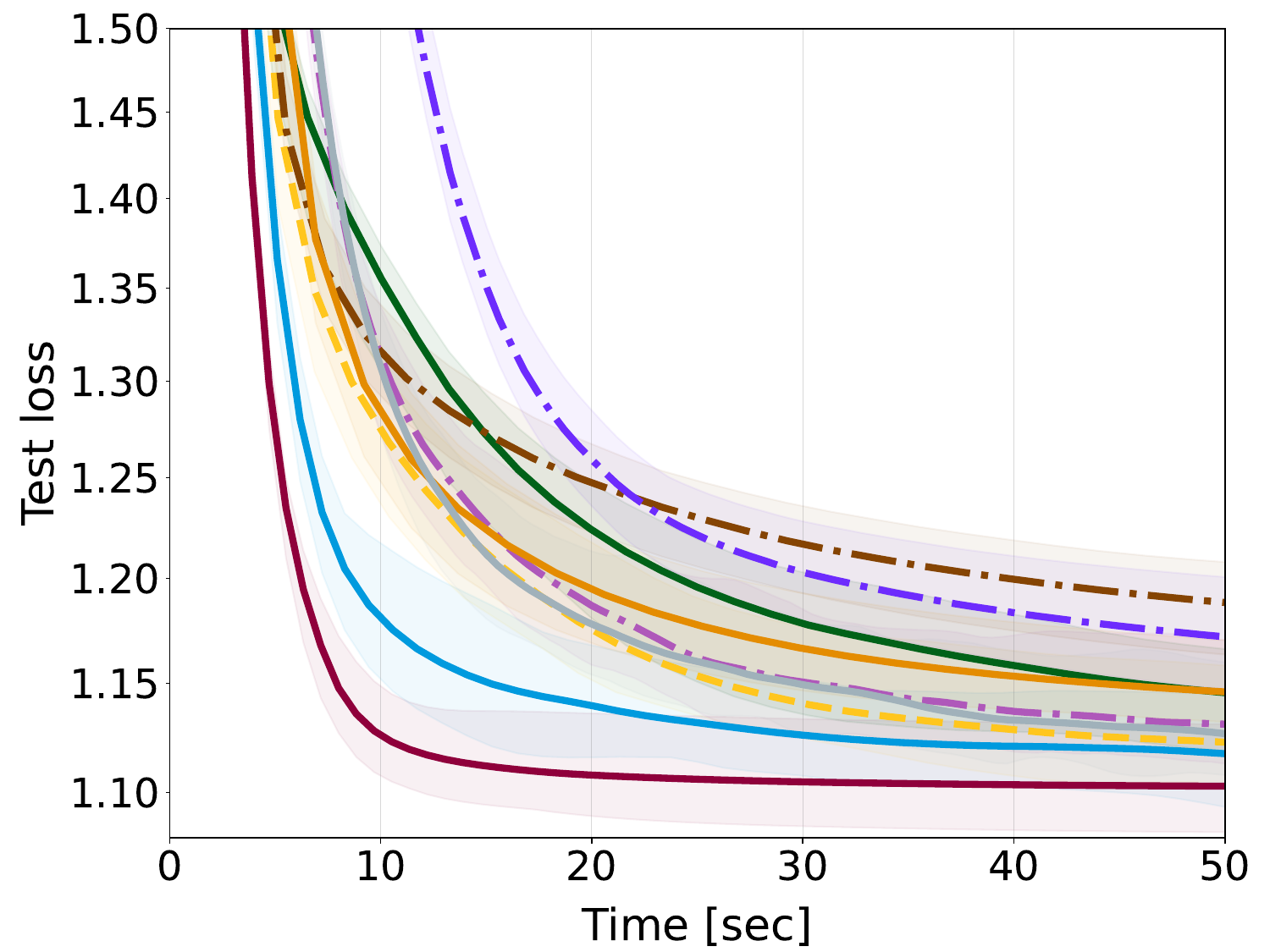}
\end{minipage}
\begin{minipage}{0.33\textwidth}
    \centering
    \includegraphics[width=1\linewidth]{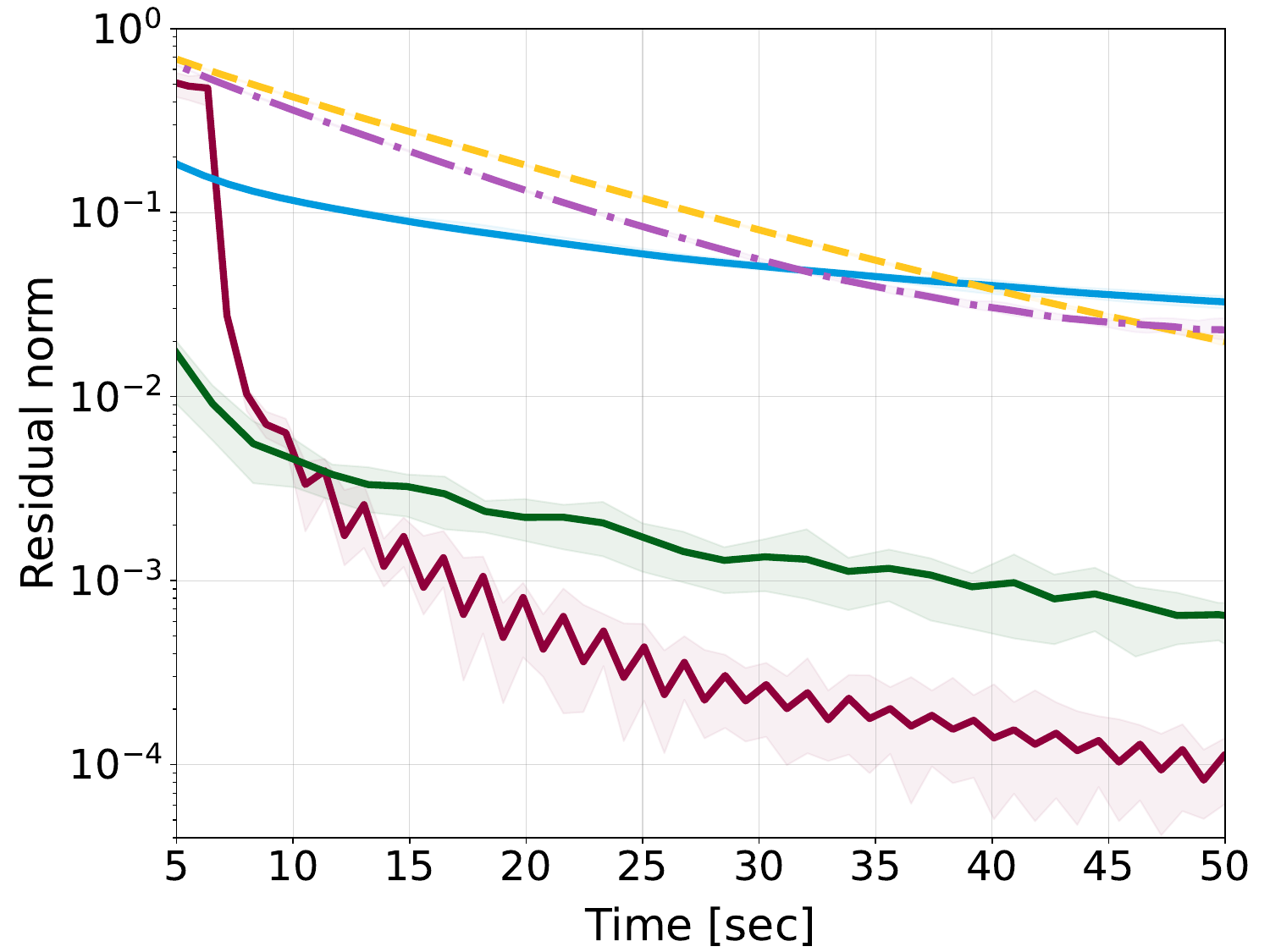}
\end{minipage}
\caption{\revise{Data hyper-cleaning task tested on the Kuzushiji-MNIST dataset when~$p=0.6$. \boldt{Left:} test accuracy; \boldt{Center}: test loss; \boldt{Right}:~residual norm of the linear system, $\norm{A_kv_k-b_k}$.}}
\label{fig:clean_data_KMNIST}
\end{figure*}

\subsection{Synthetic problems}
We concentrate on a synthetic scenario in bilevel optimization:
\begin{equation}\label{eq:synthe_sc}
\begin{array}{cl}
    \min\limits_{x \in \mathbb{R}^{d}}& f(x,y^*):= c_1\cos\kh{x^\top D_1 y^*} + \frac{1}{2} \norm{D_2x-y^*}^2,
    \\[2mm]
    \mathrm{s.\,t.}& 
    \begin{aligned}[t]
        y^* &=\argmin\limits_{y \in \mathbb{R}^{d}}  g(x,y)
        \\
        :&= c_2\sum_{i=1}^{d}{\sin(x_i+y_i)} + \log\kh{\sum_{i=1}^{d}{e^{x_iy_i}}} + \frac{1}{2}y^\top\kh{D_3+G}y,
    \end{aligned}
\end{array}
\end{equation}
where we incorporate the trigonometric and log-sum-exp functions to enhance the complexity of the objective functions. In addition, we utilize the positive-definite matrix $G$ to ensure~a~strongly convex lower-level problem, and diagonal matrices~$D_i$ ($i=1,2,3$) to control the condition numbers of both levels.

\begin{figure*}[htbp]
	\centering
	\begin{minipage}{0.37\textwidth}
		\centering
		\includegraphics[width=1\linewidth]{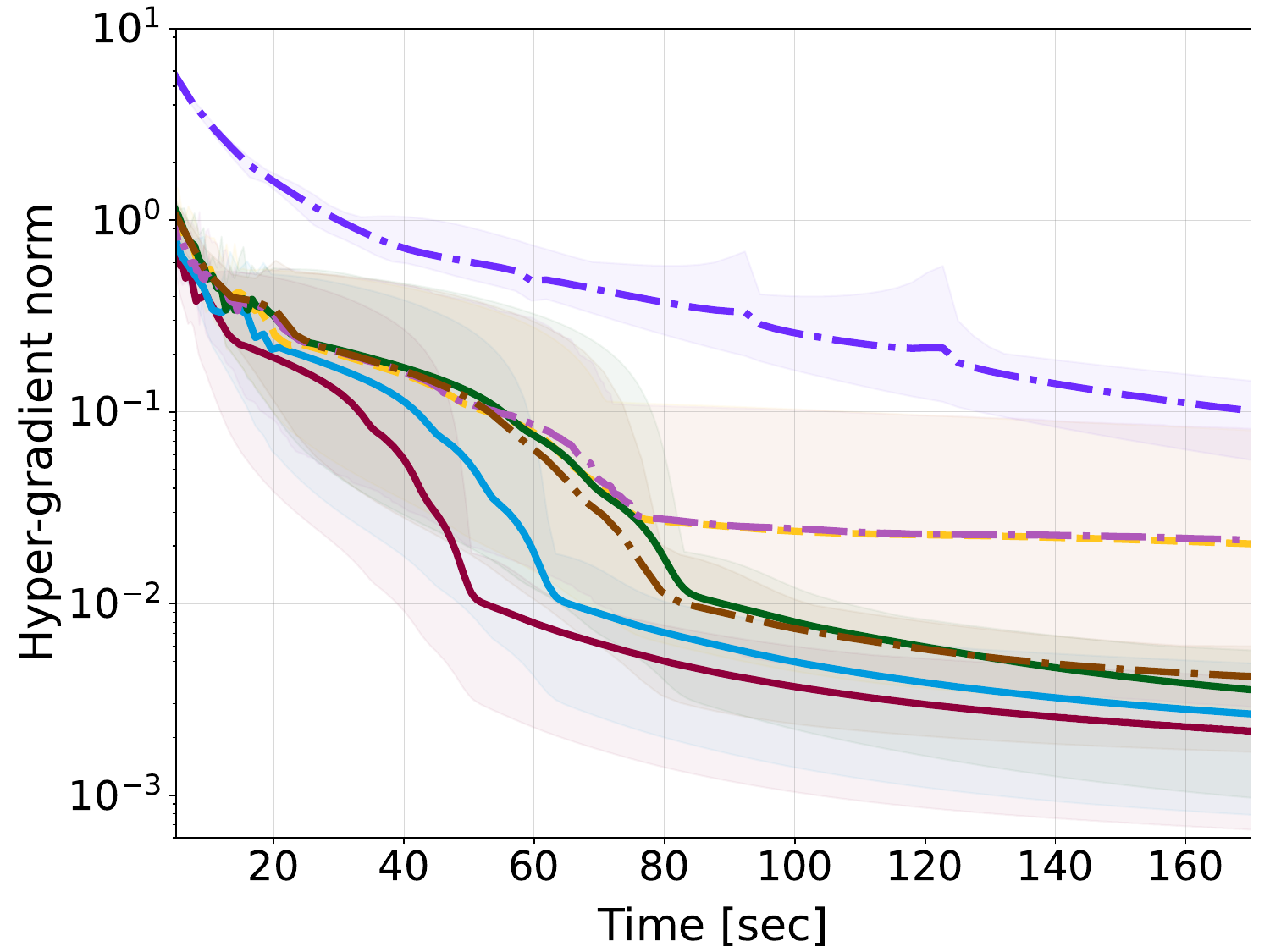}
	\end{minipage}
	\,
	\begin{minipage}{0.37\textwidth}
		\centering
		\includegraphics[width=1\linewidth]{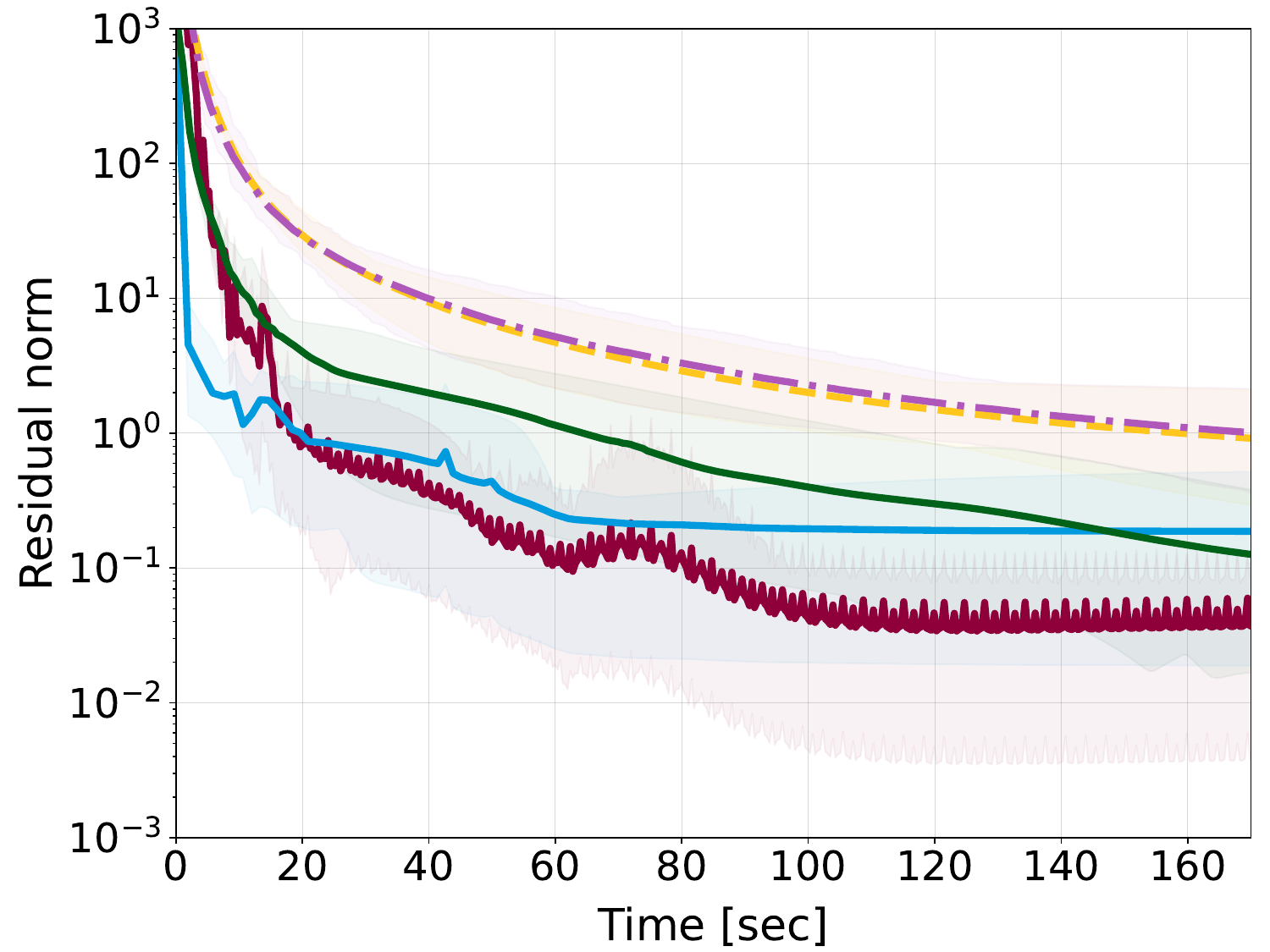}
	\end{minipage}
	\begin{minipage}{0.16\textwidth}
		\centering
		\includegraphics[width=1\linewidth]{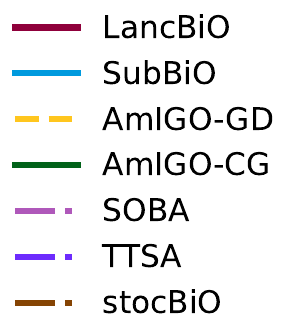}
	\end{minipage}
	\caption{Comparison of the bilevel algorithms on the synthetic problem. \boldt{Left:} norm  of the hyper-gradient; \boldt{Right}:~residual norm of the linear system, $\norm{A_kv_k-b_k}$.}
	\label{fig:sythe_mean_compare}
\end{figure*}

In this experiment, we set the problem dimension $d=10^4$ and the constants $c_1=0.1,\ c_2=0.5$. $G$~is constructed randomly with~$d$ eigenvalues from~$1$ to $10^5$. We generate entries of $D_1$, $D_2$ and $D_3$ from uniform distributions over the intervals $[-5,5]$, $[0.1,1.1]$ and $[0,0.5]$, respectively. Taking into account the condition numbers dominated by $D_i$ ($i=1,2,3$) and $G$, we choose $\lambda=1$ and $\theta = 10^{-5}$ for all algorithms compared after a manual search. 

It can be seen from \cref{fig:sythe_mean_compare} that LancBiO achieves the final accuracy the fastest, which benefits from the more accurate $v^*$ estimation. \cref{fig:sythe_seed4_compare} illustrates how variations in $m$ and $I$ influence the performance of LancBiO and AmIGO, tested across a range from $10$ to $150$ for $m$, and from $2$ to $10$ for $I$. For clarity, we set the seed of the experiment at $4$, and present typical results to encapsulate the observed trends. It is observed that the increase of $m$ accelerates the decrease in the residual norm, thus achieving better convergence of the hyper-gradient, which aligns with the spirit of the classic Lanczos process. 

When $m = 50$, the estimation of $v^*$ is sufficiently accurate to facilitate effective hyper-gradient convergence, which is demonstrated in \cref{fig:sythe_seed4_compare} that for $m\ge 50$, further increases in $m$ merely enhance the convergence of the residual norm. Under the same outer iterations, to attain a comparable convergence property, $I$ for AmIGO-CG should be set to~$10$. Furthermore, given that the number of Hessian-vector products averages at $(1+{1}/{m})$ per outer iteration for LancBiO, whereas AmIGO requires~$I\ge 2$, it follows that LancBiO is more efficient.

\begin{table*}[htbp]
\caption{\revise{Comparison on the synthetic problem~\eqref{eq:synthe_sc}. The dimension of problems is denoted by $d$. Results are averaged over 10 runs.}}
\begin{center}
\begin{footnotesize}
\setlength{\tabcolsep}{3pt} 
\begin{tabular}{lcccccccc}
\toprule
\multicolumn{1}{l}{\multirow{2}{*}{Algorithm}} & \multicolumn{2}{c}{$d=10$} & \multicolumn{2}{c}{$d=100$} & \multicolumn{2}{c}{$d=1000$} & \multicolumn{2}{c}{$d=10000$} \\ \cmidrule(l{2pt}r{2pt}){2-3} \cmidrule(l{2pt}r{2pt}){4-5} \cmidrule(l{2pt}r{2pt}){6-7} \cmidrule(l{2pt}r{2pt}){8-9}
& \multicolumn{1}{c}{UL Val.} & \multicolumn{1}{c}{Time (S)} & \multicolumn{1}{c}{UL Val.} & \multicolumn{1}{c}{Time (S)} & \multicolumn{1}{c}{UL Val.} & \multicolumn{1}{c}{Time (S)} &
\multicolumn{1}{c}{UL Val.} & \multicolumn{1}{c}{Time (S)} \\
\midrule
LancBiO   & $\expnumber{4.52}{-2}$ & ${0.32}$ & $\expnumber{6.37}{-2}$ &  $0.53$ & $\expnumber{5.29}{-2}$ & $1.30$  & $\expnumber{-1.21}{-2}$ & $\ \ \ 16.36$\\
SubBiO   & $\expnumber{3.73}{-2}$ & ${1.00}$ & $\expnumber{7.19}{-2}$&  $1.17$ & $\expnumber{4.63}{-2}$ & $1.72$  & $\expnumber{-2.91}{-2}$ & $\ \ \ 21.26$\\
AmIGO-GD   & $\expnumber{1.67}{-1}$ & ${2.44}$ & $\expnumber{1.05}{-1}$ &  $3.48$ & $\expnumber{1.05}{-1}$ & $1.46$  & $\ \ \ \expnumber{4.96}{-2}$ & $\ \ \  46.53$\\
AmIGO-CG   & $\expnumber{5.56}{-2}$ & ${0.40}$ & $\expnumber{7.65}{-2}$ &  $1.90$ & $\expnumber{5.73}{-2}$ & $2.44$  & $\ \ \ \expnumber{3.68}{-2}$ & $\ \ \ 25.00$\\
SOBA   & $\expnumber{1.70}{-1}$ & $0.60$ & $\expnumber{1.28}{-1}$ &  $1.64$ & $\expnumber{1.03}{-1}$ & $2.52$  & $\expnumber{\ \ \ 3.49}{-2}$ & ${\ \ \ 33.89}$\\
TTSA   & $\expnumber{5.66}{-2}$ & ${0.47}$ & $\expnumber{5.52}{-2}$ &  $0.89$ & $\expnumber{6.52}{-2}$ & $2.81$  & $\expnumber{\ \ \ 1.87}{-1}$ & $121.07$\\
stocBiO & $\expnumber{6.24}{-2}$ & ${0.29}$ & $\expnumber{6.02}{-2}$ &  $0.45$ & $\expnumber{5.21}{-2}$ & $1.34$  & $\expnumber{-1.30}{-2}$ & $\ \ \ 20.66$\\
\bottomrule
\end{tabular}
\end{footnotesize}
\end{center}
\label{tab:scale}
\end{table*}

\revise{Moreover, to illustrate how the proposed methods scale with increasing dimensions, we present the convergence time and the final upper-level (UL) value under different problem dimensions $d=10^{i}, i=1,2,3,4$ in~\cref{tab:scale}. The results show the robustness of the proposed methods across varying problem dimensions.}

\revise{In addition, as discussed in~\cref{app:LLPL}, to address the bilevel problem where the lower-level problem exhibits an indefinite Hessian, the framework LancBiO~(\cref{alg:LancBiO}) requires~a minor modification. Specifically, line~13 in~\cref{alg:LancBiO}, which solves a small-size tridiagonal linear system, will be replaced by solving~a low-dimensional least squares problem. We test the modified method LancBiO-MINRES on the following bilevel example borrowed from \citet{liu2023valueseq} with a non-convex lower-level problem,
\begin{equation}\label{eq:synthe_nc}
\begin{array}{cl}
    \min\limits_{x \in\mathbb{R}} & f(x,y^*):= \norm{x-a}^2+\norm{y^*-a-c}^2
    \\[2mm]
    \mathrm{s.\,t.}& y_i^* \in \argmin\limits_{y_i\in\mathbb{R}} \,\sin(x+y_i-c_i),\,\text{for}\,i=1,2,\ldots,d,
\end{array}
\end{equation}
where the subscript $i$ denotes the $i$-th component of a vector, while $a\in\mathbb{R}$ and $c\in\mathbb{R}^d$ are parameters. The results, reported in \cref{fig:non-convex}, imply the potential of extending our work to non-convex scenarios.}

\begin{figure*}[htbp]
\centering
\begin{minipage}{0.4\textwidth}
    \centering
    \includegraphics[width=1\linewidth]{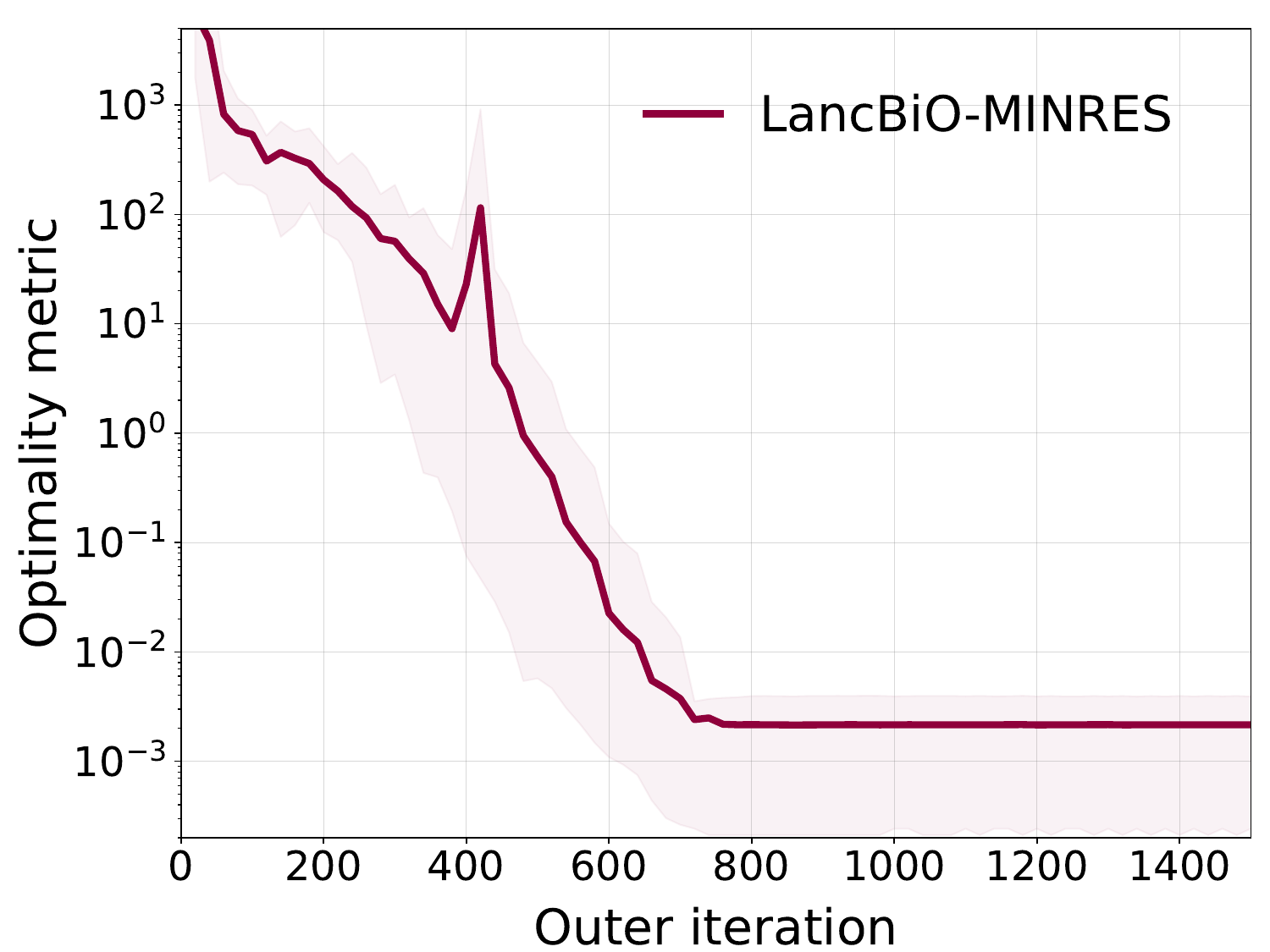}
\end{minipage}
\caption{\revise{Test LancBiO-MINRES on the synthetic problem \eqref{eq:synthe_nc} with $d=100$. The metric follows from the necessary conditions developed for bilevel problems without lower-level convexity in Theorem 1 {of~\citet{xiao2023galet}.}}}
\label{fig:non-convex}
\end{figure*}

\subsection{Logistic regression on $20$Newsgroup}
Consider the hyper-parameters selection task on the $20$Newsgroups dataset \citep{grazzi2020iteration}, which contains $c=20$ topics with around $18000$ newsgroups posts represented in a feature space of dimension $l= 130107$. The goal is to simultaneously train a linear classifier $w$ and determine the optimal regularization parameter $\zeta$. The task is formulated as follows,
\begin{equation*}
\begin{array}{cl}
         \min\limits_\lambda& \mathcal{L}_{val}(\zeta, w^*) := \frac{1}{\left|\mathcal{D}_{\text{val}}\right|} \sum_{(x_i, y_i) \in \mathcal{D}_{\text{val}}} L(w^* x_i, y_i) 
         \\[5mm]
         \mathrm{s.\,t.}&
        \begin{aligned}[t]
            w^* &= \argmin_w\ \mathcal{L}_{tr}(w, \zeta) \\
            :&= \frac{1}{\left|\mathcal{D}_{\text{tr}}\right|} \sum_{(x_i, y_i) \in \mathcal{D}_{\text{tr}}} L(w x_i, y_i) + \frac{1}{c l} \sum_{i=1}^c \sum_{j=1}^l \zeta_j^2 w_{i j}^2.
        \end{aligned}
\end{array}
\end{equation*}
where $L(\cdot)$ is the cross-entropy loss and $\{\zeta^2_j\}$ are the non-negative regularizers.

The experiment is implemented in the deterministic setting, where we implement all compared methods with full-batch, the training set, the validation set and the test set contain $5657$, $5657$ and $7532$ samples, respectively. For algorithms that incorporate inner iterations to approximate $y^*$ or $v^*$, we select the inner iteration number from the set $\left\{5 i \mid i=1,2,3,4\right\}$. To guarantee the optimality condition of the lower-level problem, we adopt a decay strategy for the outer iteration step size, i.e., $\lambda_k = \lambda/k^{0.4}$, for all algorithms. The constant step size $\theta$ of inner iteration is selected from the set $\{0.01,0.1,1,10\}$ and the initial step size $\lambda$ of outer iteration is chosen from $\left\{5\times10^i \mid i=-3,-2,-1,0,1,2,3\right\}$. The results are presented in~\cref{fig:hypasele_detem_mean_compare}. In this setting, AmIGO-CG exhibits slightly better performance in reducing the residual norm. Nevertheless, under the same time, LancBiO implements more outer iterations to update $x$, which optimizes the hyper-function more efficiently.

\end{document}